\documentclass[reqno,english]{amsart}%

\usepackage{amsfonts,amsmath,latexsym,verbatim,amscd,mathrsfs,array}
\usepackage[colorlinks=true]{hyperref}
\usepackage{amssymb,amsthm,graphicx}
\usepackage{pdfsync}
\usepackage{epstopdf}
\usepackage{cite}

\usepackage[top=3cm, bottom=3cm, left=3cm, right=3cm]{geometry}

\providecommand{\U}[1]{\protect\rule{.1in}{.1in}}

\numberwithin{equation}{section}

\newtheorem{theorem}{Theorem}[section]
\newtheorem{corollary}{Corollary}[section]
\newtheorem{lemma}{Lemma}[section]
\newtheorem{proposition}{Proposition}[section]
\newtheorem{remark}{Remark}[section]
\newtheorem{definition}{Definition}[section]
\newtheorem{claim}{Claim}[section]
\newtheorem{theoremA}{Theorem}[section]
\newtheorem{theoremB}{Theorem}[section]

\theoremstyle{plain}      
\newtheorem{aim}{Aim}

\numberwithin{equation}{section}

\newcommand{\bd}{\begin{definition}}
	\newcommand{\ed}{\end{definition}}

\newcommand{\ba}{\begin{aim}}
	\newcommand{\ea}{\end{aim}}

\newcommand{\br}{\begin{remark}}
	\newcommand{\er}{\end{remark}}

\newcommand{\be}{\begin{equation}}
	\newcommand{\ee}{\end{equation}}

\newcommand{\bc}{\begin{corollary}}
	\newcommand{\ec}{\end{corollary}}

\def\vs{\vskip4pt}

\begin{document}
	\title[Uniqueness and non-degeneracy of Chern-Simons system]{Local uniqueness and non-degeneracy of blowup solutions to a Chern-Simons system}
	\author[Z. Cheng]{Zetao Cheng}
	\address{\noindent Department of Mathematics and Research Institute for Natural Sciences, College of Natural Sciences, Hanyang University, 222 Wangsimni-ro Seongdong-gu, Seoul 04763, Republic of Korea}
	\email{chengzt20@hanyang.ac.kr}
	
	\author[H. Li]{Haoyu Li}
	\address{\noindent Department of Mathematics, Faculty of Science and Technology, University of Macau, Taipa, Macau}
	\email{hyli1994@hotmail.com}
	
	\author[L. Zhang]{Lei Zhang}
	\address{\noindent Department of Mathematics, University of Florida, 1400 Stadium Rd, Gainesville, FL, 32611}
	\email{leizhang@ufl.edu}
	
	\begin{abstract}
		 In this paper, we study blowup solutions of an important class of Chern--Simons systems. We first show that when blowup of mean-field type occurs, the corresponding blowup solution is unique under natural geometric assumptions. We also establish the non-degeneracy of the linearized system around these blowup solutions. To prove these main results, we carry out a precise blowup analysis, so that the asymptotic description of the solutions reveals the curvature information needed for the uniqueness and non-degeneracy results. Compared with related work on similar problems, our estimates are more delicate and technically involved.
		
		\vspace{3mm} \noindent{\bf Keywords:} Chern-Simons system; Bubbling solutions; Local uniqueness; Non-degeneracy; Pohozaev identity.
		
		\vspace{3mm}\noindent {\bf AMS} Subject Classification 2020: 35A02; 35B40; 35J47.

	\end{abstract}
	
	\date{}
	\maketitle

\allowdisplaybreaks[4]

\tableofcontents 
	
	\section{Introduction and main results}
\subsection{Introduction}\label{Subsection:introduction}
For decades, Chern--Simons equations and systems have attracted extensive attention from both mathematicians and physicists due to their deep connections with gauge field theory, condensed matter physics, and nonlinear differential equations. The mathematical origin of these equations can be traced back to the Chern--Simons gauge theory introduced by Chern and Simons in differential geometry, and later developed in theoretical physics as an effective model for planar gauge interactions; see \cite{JackiwWeinberg, HongKimPac}.

In condensed matter physics, Chern--Simons models arise naturally in the study of vortices in superconductivity, anyon physics, fractional quantum Hall effects, and topological phases of matter. Unlike the classical Abelian Higgs model, the Chern--Simons theory replaces the Maxwell term by a topological Chern--Simons term, producing a different vortex interaction mechanism and leading to highly nonlinear elliptic equations with exponential nonlinearities.

The scalar Chern--Simons equation
\[
\Delta u+\frac{1}{\varepsilon^2}e^u(1-e^u) = 4\pi\sum_{j=1}^N\delta_{p_j}
\]
describes self-dual vortex configurations in Abelian gauge theory. Here, the singular sources $\{p_j\}$ represent vortex locations, while the parameter $\varepsilon>0$ corresponds physically to the coupling strength or characteristic length scale of the model. The asymptotic regime $\varepsilon\to0$ is particularly important because it corresponds to the concentration of vortices and the formation of highly localized structures.

The system considered in this paper is a coupled version of the scalar equation and models interacting multi-component gauge fields. Such systems appear in non-Abelian gauge theories, multi-layer superconductivity, and interacting condensate models; see \cite{Gudnason, Dunne}. Compared with the scalar equation, the coupled structure introduces substantially more complicated interaction phenomena and analytical difficulties.

Mathematically, the study of blowup solutions plays a central role in understanding the singular limit $\varepsilon\to0$. Blowup solutions correspond to concentration phenomena where the energy density accumulates near finitely many points. Physically, these points represent vortex cores or highly concentrated field configurations. The local profile near each blowup point is closely related to entire solutions of Liouville-type systems, while the global behavior reflects geometric and topological interactions through the Green's function of the underlying surface.

The uniqueness and non-degeneracy of blowup solutions are particularly important in both analysis and physics. Uniqueness provides stability of the asymptotic vortex configuration and prevents the appearance of multiple nearby concentration patterns with the same limiting geometry. Non-degeneracy is equally fundamental because it is closely related to stability theory, Morse index computations, gluing constructions, Lyapunov--Schmidt reductions, and dynamical properties of vortices. Moreover, non-degeneracy is often an essential ingredient in the construction of more complicated solutions and in the study of vortex dynamics.

For the scalar Chern--Simons equation, extensive studies have been carried out concerning existence, classification, blowup analysis, and uniqueness of bubbling solutions; see \cite{CaffarelliYang1995, Tarantello1996, LinYan2013, LinYan2017, LinChenWang2004}. In recent years, attention has shifted toward coupled Chern--Simons systems, where the interaction between different components creates new phenomena absent in the scalar case. Important progress on sharp estimates and blowup analysis for the system considered in this paper was obtained by Huang and Zhang in \cite{HuangZhang2017}, while existence results via Lyapunov--Schmidt reduction were established in \cite{HuangArxiv2018}. We refer the readers to \cite{TarantelloBook2008} for a comprehensive survey of this problem.

In this paper, we consider 
\begin{equation}\label{e:001}
    \begin{cases}
     \Delta u_1+\frac{1}{\varepsilon^2}e^{u_2}(1-e^{u_1})=4\pi\sum_{j=1}^{N}\delta_{p_{1,j}}\mbox{ in }\Omega,\\
     \Delta u_2+\frac{1}{\varepsilon^2}e^{u_1}(1-e^{u_2})=4\pi\sum_{j=1}^{N}\delta_{p_{2,j}}\mbox{ in }\Omega,\\
     u_1,u_2\in H^1_{per}(\Omega).
   \end{cases}
\end{equation}
Here, $\Omega=\mathbb{R}/\mathbb{Z}\times\mathbb{R}/\mathbb{Z}$ denotes a flat torus. Without loss of generality, we assume the surface area of $\Omega$ is normalized so that $|\Omega|=1$. The positive constant $\varepsilon$ is assumed to be sufficiently small. $H_{per}^1(\Omega)$ consists of periodic $H^1$ functions defined on the torus. The points $\{p_{1,j_1},p_{2,j_2}\}_{j_1=1,\cdots,N}\subset \Omega$ are fixed and $\delta_p$ is the Dirac measure supported on $\{p\}$.

Regarding the scalar Chern-Simons equation defined on $\Omega$:
\begin{equation}\label{e:SingleEquation}
    \begin{cases}
     \Delta u+\frac{1}{\varepsilon^2}e^u(1-e^u)=4\pi\sum_{i=1}^N\delta_{p_i}\mbox{ in }\Omega,\\
     u\in H^1_{per}(\Omega),
   \end{cases}
\end{equation}
a Brezis-Merle type result was established, in which Choe and Kim demonstrated the possibility of the existence of the following ``blow-up'' solutions $\{u_\varepsilon\}_\varepsilon$ of Problem (\ref{e:SingleEquation}): \emph{There exists a finite set $\{x_{j,\varepsilon}\}\subset\Omega$ with $j=1,\cdots,k$ such that}
\begin{align}
u_\varepsilon(x_{j,\varepsilon})+2\ln\frac{1}{\varepsilon}\to+\infty\nonumber
\end{align}
and
\begin{align}
u_\varepsilon+2\ln\frac{1}{\varepsilon}\to-\infty\mbox{ uniformly on any compact subset of $\Omega\backslash\{q_1,\cdots,q_k\}$}.\nonumber
\end{align}
Here, $q_j=\lim_{\varepsilon\to0+}x_{j,\varepsilon}$ for $j=1,\cdots,k$. The existence of such mean-field type blow-up solutions was subsequently proven by Lin and Yan \cite{LinYan2013}. In their subsequent studies \cite{LinYan2017,LinYan2018}, comprehensive study on the blow-up solutions and classification, sharp estimate and uniqueness of blow-up solutions of mean field type under suitable non-degenerate assumptions are provided.

In recent years, attention has turned to Problem (\ref{e:001}). Given the comprehensive understanding of the scalar problem (\ref{e:SingleEquation}), it is natural to investigate whether analogous theories hold for the coupled system (\ref{e:001}). In this paper, we study the local uniqueness and the non-degeneracy of the blow-up solution of Problem (\ref{e:001}). Our main results concern the blow-up solutions of mean field type. For this problem, some studies already exist. In \cite{HuangZhang2017}, Huang and Zhang provide a sharp estimate for blow-up solution of mean field type. The existence of such solutions is shown in Huang \cite{HuangArxiv2018} via Lyapunov-Schmidt reduction. We note that there are other kind of Chern-Simons system. See, for instance, \cite{AoLinWei2016,AoLinWei20161}. 

To study Problem (\ref{e:001}), one often rewrites it into an equivalent form without Dirac measures. Define
\begin{align}\label{def:u0}
u_{1,0}(x)=-4\pi\sum_{j=1}^{N} G(x,p_{1,j})\mbox{ and }u_{2,0}(x)=-4\pi\sum_{j=1}^{N} G(x,p_{2,j})
\end{align}
where the Green's function $G(x,p)$ satisfies
\begin{equation}\label{e:GreensFunction}
    \begin{cases}
     -\Delta_x G(x,p)=\delta_p-1,\\
     \int_\Omega G(x,p)dx=0.
   \end{cases}
\end{equation}
Then, the problem can be re-written as
\begin{equation}\label{e:111}
    \begin{cases}
     \Delta u_{1,\varepsilon}+\frac{1}{\varepsilon^2}e^{u_{2,\varepsilon}+u_{2,0}}(1-e^{u_{1,\varepsilon}+u_{1,0}})=4\pi N\mbox{ in }\Omega,\\
     \Delta u_{2,\varepsilon}+\frac{1}{\varepsilon^2}e^{u_{1,\varepsilon}+u_{1,0}}(1-e^{u_{2,\varepsilon}+u_{2,0}})=4\pi N\mbox{ in }\Omega,\\
     u_{1,\varepsilon},u_{2,\varepsilon}\in H_{per}^1(\Omega).
   \end{cases}
\end{equation}
The main purpose of this article is to prove the uniqueness of blow-up solutions. Let $u^{(1)}_{\varepsilon}=(u_{1,\varepsilon}^{(1)},u_{2,\varepsilon}^{(1)})$ and $u^{(2)}_{\varepsilon}=(u_{1,\varepsilon}^{(2)},u_{2,\varepsilon}^{(2)})$ be two sequences of blow-up solutions having the same set of blow-up points $(q_1,\dots,q_k)$ that are regular points. It is worth emphasizing that we do not \emph{a priori} require $u^{(1)}_{\varepsilon}$ and $u^{(2)}_{\varepsilon}$ to share the exact same local maximum points around each $q_j$ $(j=1,\dots,k)$. Also, $q_j$ being a regular point means $q_j$ is not a singular source. For the proof of our main results, we postulate the following natural assumption: 

\begin{align}\tag{A}\label{ASSUMPTION01}
\begin{cases}
(A_1).\,\,\, u^{(1)}_{\varepsilon} \mbox{ and } \, u^{(2)}_{\varepsilon} \mbox{ have the same set of regular blow-up points }\,  q_1,\dots,q_k,\\
(A_2).\,\,\,\mbox{For small } \,\, \delta>0 \,\, \mbox{ such that } B(q_i,\delta)\cap B(q_j,\delta)=\emptyset, i\neq j, \\
\quad\quad\quad \max_{B_\delta(q_{s})}u_{i,\varepsilon}^{(a)}(x)+2\ln\frac{1}{\varepsilon}\to+\infty\mbox{ as }\varepsilon\to0; \, s=1,\dots,k, \quad a=1,2, \quad i=1,2.\\
(A_3).\,\,\,\mbox{For $\delta>0$ as above, }\\
\quad\quad\quad \max_{B_\delta(q_s)}u_{i,\varepsilon}^{(a)}(x)\to-\infty\mbox{ as }\varepsilon\to0; \, s=1,\dots,k, \quad a=1,2, \quad i=1,2.\\
(A_4).\,\,\,\mbox{For any }K\subset\subset\Omega\backslash \{q_{1},\cdots,q_{k}\},\\
\quad\quad\quad \max_{K} u_{i,\varepsilon}^{(a)}+2\ln\frac{1}{\varepsilon}\to-\infty\mbox{ as }\varepsilon\to0, \quad a=1,2, \quad i=1,2.
\end{cases}
\end{align}

Here, by Assumption $(A_1)$, we mean that for the sequences $\{x_{j,\varepsilon}^{(1)}\}_\varepsilon$ and $\{x_{j,\varepsilon}^{(2)}\}_\varepsilon$ satisfying
\begin{align}
\max_{x\in B_\delta(q_j)}u_{1,\varepsilon}^{(1)}(x)=u_{1,\varepsilon}^{(1)}(x^{(1)}_{j,\varepsilon})\,\,\mbox{ and }\max_{x\in B_\delta(q_j)}u_{2,\varepsilon}^{(2)}(x)=u_{2,\varepsilon}^{(2)}(x^{(2)}_{j,\varepsilon}),\nonumber
\end{align}
we have $\lim_{\varepsilon\to0}x_{j,\varepsilon}^{(1)}=\lim_{\varepsilon\to0}x_{j,\varepsilon}^{(2)}=q_j$ for $j=1,\cdots,k$.

From now on, all solutions we consider satisfy Assumption (\ref{ASSUMPTION01}).
To state our main result, let us introduce the following functions. Let $\vec{q}=(q_1,\cdots,q_k)\in\Omega^k$ satisfy $q_i\neq q_j$ if $i\neq j$. Then, we denote
\begin{align}
G_i^*(\vec{q})=\sum_{j=1}^k u_{0,i}(q_j)+8\pi\sum_{1\leq j< l\leq k}G(q_j,q_l).\nonumber
\end{align}
Moreover, let us denote
\begin{align}
f_{i,j}(x)=8\pi\Big(\big(\gamma(x,q_j)-\gamma(q_j,q_j)\big)+\sum_{l\neq j}\big(G(x,q_l)-G(q_j,q_l)\big)\Big)+u_{0,i}(x)-u_{0,i}(q_j)
\end{align}
for $i=1,2$ and $j=1,\cdots,k$. Let us define the following quantity:
\begin{align}\label{def:D}
D(\vec{q}):=\lim_{\delta\to0+}\Bigg[&\sum_{j=1}^k\frac{\rho_{1,j}}{e^{u_{0,1}(q_j)}}\Bigg(\int_{\Omega_j}\frac{e^{f_{1,j}}-1}{|x-q_j|^4}-\int_{\mathbb{R}^2\backslash\Omega_j}\frac{1}{|x-q_j|^4}\Bigg)\nonumber\\
&+\sum_{j=1}^k\frac{\rho_{2,j}}{e^{u_{0,2}(q_j)}}\Bigg(\int_{\Omega_j}\frac{e^{f_{2,j}}-1}{|x-q_j|^4}-\int_{\mathbb{R}^2\backslash\Omega_j}\frac{1}{|x-q_j|^4}\Bigg)\Bigg].
\end{align}
Here, $\{\Omega_j\}_{j=1,\cdots,k}$ is a set of open subsets of $\Omega$ satisfying:
\begin{align}\label{def:Omegaj}
\begin{cases}
(1)\,\,\Omega_i\cap\Omega_j=\emptyset\mbox{ for }i\neq j;\\
(2)\,\,\cup_{j=1}^k\overline{\Omega_j}=\overline{\Omega};\\
(3)\,\,\mbox{there exists a positive number }\delta\mbox{ such that }B_\delta(q_j)\subset\Omega_j\mbox{ for any }j=1,\cdots,k.
\end{cases}
\end{align}
And the constants $\rho_{i,j}$ are given by $\rho_{i,j}=e^{8\pi(\gamma(q_j,q_j)+\sum_{l\neq j}G(q_l,q_j))+u_{0,i}(q_j)}$ for $i=1,2$ and $j=1,\cdots,k$.
Our main results read as follows.

\begin{theorem}\label{t:Unique}
Let $u^{(1)}_{\varepsilon}$ and $u^{(2)}_{\varepsilon}$ be two sequences of blow-up solutions to Problem (\ref{e:001}) satisfying Assumption (\ref{ASSUMPTION01}). Assume that
\begin{itemize}
    \item [$(A)$] $\vec{q}$ is a non-degenerate critical point of $G_1^*+G_2^*$;
    \item [$(B)$] $D(\vec{q})<0$;
    \item [$(C)$] $u_{0,1}(q_{j_1})-u_{0,2}(q_{j_1})=u_{0,1}(q_{j_2})-u_{0,2}(q_{j_2})$ for any $j_1,j_2=1,\cdots,k$.
\end{itemize}
Then, there exists an $\varepsilon_0$ such that  $u^{(1)}_{\varepsilon}(x)=u^{(2)}_{\varepsilon}(x)$ for all $\varepsilon\in (0,\varepsilon_0)$ and both components. 
\end{theorem}

Furthermore, we obtain the \emph{non-degeneracy} of the blow-up solutions under the assumptions in Theorem \ref{t:Unique}. 
\begin{definition}\label{def:NonDegeneracy}
A solution $(u_1,u_2)$ of Problem (\ref{e:001}) is \emph{non-degenerate} if the linear problem
\begin{equation}\label{e:Linearized}
    \begin{cases}
     \Delta \phi_1+\frac{1}{\varepsilon^2}\Big[e^{u_2+u_{0,2}}\phi_2-e^{u_1+u_2+u_{0,1}+u_{0,2}}(\phi_1+\phi_2)\Big]=0\mbox{ in }\Omega,\\
     \Delta \phi_2+\frac{1}{\varepsilon^2}\Big[e^{u_1+u_{0,1}}\phi_1-e^{u_1+u_2+u_{0,1}+u_{0,2}}(\phi_1+\phi_2)\Big]=0\mbox{ in }\Omega,\\
     \phi_1,\phi_2\in H^1_{per}(\Omega).
   \end{cases}
\end{equation}
admits only the trivial solution $(\phi_1,\phi_2)=(0,0)$.
\end{definition}

\begin{theorem}\label{t:nondegenerate}
Under the assumptions of Theorem \ref{t:Unique}, there exists a constant $\varepsilon_1>0$ such that for any $\varepsilon\in(0,\varepsilon_1)$ the solution $(u_{1,\varepsilon},u_{2,\varepsilon})$ is non-degenerate.
\end{theorem}

Local uniqueness and non-degeneracy of blow-up solutions have recently become popular topics. To this end, researchers have developed sophisticated tools. For instance, in reference \cite{LinYan2018}, Lin and Yan employed classical methods of sharp estimates, which were later inherited by \cite{BartJevnLeeYang2019} and applied to the mean field equation. Recent developments include higher-order expansions based on Fourier analysis in \cite{BartYangZhang2024,BartYangZhang2026}, as well as an enhanced version of this method in \cite{ChengLiZhang2025Preprint}.

In this paper, our main tools are a refined version of the sharp estimate in \cite{HuangZhang2017} and various Pohozaev identities. To achieve our goal, we relax certain assumptions and complete the sharp estimate initiated in \cite{HuangZhang2017}. For example, (\ref{ASSUMPTION01}) is the relaxed version of \cite[(1-3) $\And$ ($A1-A3$)/ p.396]{HuangZhang2017}. Moreover, the fully bubbling behavior emerges as a rigorous deduction rather than a priori assumption (see Subsection \ref{Subsection:FullyBubbling}). Similar to the scalar field counterpart \cite{LinYan2018}, we obtain a refined relationship between $\varepsilon$ and $\beta_{j,\varepsilon}$ for \emph{the coupled system} (\ref{e:001}). The latter is achieved via an almost equal division of the mass (see Lemma \ref{l:AlphaMu-2}), which strongly improves the point-wise estimates. These analytical innovations play a crucial role in the proofs of uniqueness and non-degeneracy.

\subsection{Organization of this paper and notations}
In Section \ref{Section:SharperCMP}, we complete the sharp estimate of \cite{HuangZhang2017}, which are essential to our proof. Since we want to prove the local uniqueness, it is natural to select two sequences of blow-up solutions and prove that they coincide. In Section \ref{Section:TwoSolutions}, we provide several preliminary estimates on the difference of these sequences. The main results are proved in Section \ref{Section:ProofTheorems}.

Throughout the paper, we adopt the following standard notations:
\begin{itemize}
    \item $i'$ is an index mapping defined as follows: $1'=2$ and $2'=1$;
    \item $q_1,\cdots,q_k$ denote the common regular blow-up points for both components and both sequences of solutions;
    \item $p_{1,1},\cdots,p_{1,N}$ and $p_{2,1},\cdots,p_{2,N}$ denote the supports of the Dirac measures;
    \item We will rigorously demonstrate that for any $j_1, j_2 = 1, \dots, k$ and any sequence indices $a, b \in \{1, 2\}$, the maximum heights $\beta_{j_1,\varepsilon}^{(a)}$ and $\beta_{j_2,\varepsilon}^{(b)}$ are comparable up to $O(1)$, and similarly for the scaling factors $\mu_{j_1,\varepsilon}^{(a)}$ and $\mu_{j_2,\varepsilon}^{(b)}$. Consequently, in the asymptotic error bounds, we will frequently omit the sub/superscripts for brevity, simply writing $\beta$ and $\mu$ (e.g., $O(e^{\beta})$ and $O(\mu^{-2})$).
\end{itemize}

\section{Improvements of the results in  \cite{HuangZhang2017}}\label{Section:SharperCMP}
To prove the uniqueness, we need some result sharper than the ones in \cite{HuangZhang2017}. We provide them in this section.
Sharper estimates in this section are based on the results in \cite{HuangZhang2017}.

\begin{proposition}\label{prop:RefinedEstimates}
Suppose that the sequence of solutions $u_\varepsilon=(u_{1,\varepsilon},u_{2,\varepsilon})$ satisfies Assumption (\ref{ASSUMPTION01}). Then the following refined estimates hold for sufficiently small $\theta>0$ and $\tau, \tau' \in (0,1)$:
\begin{itemize}
    \item [$(1)$] The local error term $\eta_{i,j,\varepsilon}$ defined in (\ref{def:Eta}) satisfies
    \begin{align}\label{ineq:EtaFINAL}
    |\eta_{i,j,\varepsilon}(x)|\leq C(1+\mu_{j,\varepsilon}|x-x_{j,\varepsilon}|)^\tau(\mu_{j,\varepsilon}^{-2\tau}+e^{\tau'\beta_{j,\varepsilon}})
    \end{align}
    for $x\in B_\delta(x_{j,\varepsilon})$, $i=1,2$, and $j=1,\cdots,k$.
    \item [$(2)$] The local masses $\widetilde{M}_{i,j,\varepsilon}$ defined in (\ref{e:GlobalLiouvilleSystem}) satisfy $|\widetilde{M}_{i,j,\varepsilon}-8\pi|=O(\mu_{j,\varepsilon}^{-2+\theta})$ for $i=1,2$ and $j=1,\cdots,k$.
    \item [$(3)$] In $C^1(\Omega\backslash\cup_{j=1}^k B_\delta(x_{j,\varepsilon}))$, we have the expansion
    \begin{align}
    u_{i,\varepsilon}=\frac{1}{|\Omega|}\int_\Omega u_{i,\varepsilon}+\sum_{j=1}^k\widetilde{M}_{i,j,\varepsilon}G(x,x_{j,\varepsilon})+O(\mu_{j,\varepsilon}^{-2}+e^{\beta_{j,\varepsilon}}).\nonumber
    \end{align}
    \item [$(4)$] The gradient of the regular part $f_{i,j,\varepsilon}$ defined in (\ref{def:fijvarepsilon}) satisfies $|\nabla f_{i,j,\varepsilon}(x_{j,\varepsilon})|=O(\mu_{\varepsilon}^{-2})$.
    \item [$(5)$] For $x\in\Omega\backslash\cup_{j=1}^k B_\delta(x_{j,\varepsilon})$, the global error is bounded by $|\eta_{i,j,\varepsilon}(x)|=O(\mu_{j,\varepsilon}^{-2+\theta}+e^{\beta_{j,\varepsilon}+\theta})$.
    \item [$(6)$] The fundamental relation linking $\varepsilon$ and $\beta_{j,\varepsilon}$ takes the form
    \begin{align}
    2\sum_{i=1}^2\sum_{j=1}^k & e^{I_{i,j}+u_{0,i}(x_{j,\varepsilon})}\Bigg(\int_{\Omega_j\backslash B_{\theta_{j,\varepsilon}}}\frac{e^{f_{i,j,\varepsilon}}-1}{|x-x_{j,\varepsilon}|^\frac{\widetilde{M}_{i,j,\varepsilon}}{2\pi}}-\int_{\mathbb{R}^2\backslash\Omega_j}\frac{1}{|x-x_{j,\varepsilon}|^\frac{\widetilde{M}_{i,j,\varepsilon}}{2\pi}}\Bigg)\mu_{j,\varepsilon}^{-2}+\sum_{j=1}^k B_{j,\varepsilon}e^{\beta_{j,\varepsilon}}\nonumber\\
    &=O(\mu_{j,\varepsilon}^{-4+\theta}+\theta_{\varepsilon,j}^2\mu_{j,\varepsilon}^{-2}).\nonumber
    \end{align}
\end{itemize}
\end{proposition}

The rest of this section is devoted to the proof of Proposition \ref{prop:RefinedEstimates}. We divide the proof into several steps, systematically relaxing the assumptions and refining the expansions.

\subsection{Notations and a summary of the computation in \cite{HuangZhang2017}}
In order to improve the results presented in \cite{HuangZhang2017}, we first revisit the sharp estimates therein.
To begin with, we quote \cite[(1-3) $\And$ ($A1-A3$)/ p.396]{HuangZhang2017} as follows. Let $(u_{1,\varepsilon},u_{2,\varepsilon})$ be a sequence of solutions of Problem (\ref{e:111}), Huang and Zhang in \cite{HuangZhang2017} assume the existence of $\{x_{1,j,\varepsilon}\}_{j=1,...,k}$ and $\{x_{2,j,\varepsilon}\}_{j=1,...,k}$ such that
\begin{itemize}
    \item [$(B_1)$] $q_j=\lim_{\varepsilon\to0}x_{1,j,\varepsilon}=\lim_{\varepsilon\to0}x_{2,j,\varepsilon}$ for $j=1,\cdots,k$, $u_{i,\varepsilon}(x_{i,j,\varepsilon})=\max_{B_d(q_j)}u_{i,\varepsilon}$ for $i=1,2$, where $B_d(q_j)$ is the ball centered at $q_j$ with radius $d$. Here, we assume that $d>0$ is small so that $B_d(q_j)\cap B_d(q_m)=\emptyset$ of $j\neq m$;
    \item [$(B_2)$] $u_{i,\varepsilon}(x_{i,j,\varepsilon})+2\ln\frac{1}{\varepsilon}\to+\infty$ as $\varepsilon\to0$, $i=1,2$, $j=1,\cdots,k$;
    \item [$(B_3)$] $u_{i,\varepsilon}(x)+2\ln\frac{1}{\varepsilon}\to-\infty$ as $\varepsilon\to0$ uniformly on any compact set of $\Omega\backslash\{q_1,\cdots,q_k\}$;
    \item [$(B_4)$] $\beta_{j,\varepsilon}=\max\{u_{1,\varepsilon}(x_{1,j,\varepsilon}), u_{2,\varepsilon}(x_{2,j,\varepsilon})\}\to+\infty$ as $\varepsilon\to0$;
    \item [$(B_5)$] $|u_{1,\varepsilon}(x_{1,j,\varepsilon})-u_{2,\varepsilon}(x_{2,j,\varepsilon})|=O(1)$;
    \item [$(B_6)$] $|x_{i,j,\varepsilon}-q_j|<C\varepsilon e^{-\frac{1}{2}\beta_{j,\varepsilon}}$ for some constant $C>0$.
\end{itemize}

In the following, with a slight abuse of notation, we assume that the sequence of solutions $u_\varepsilon=(u_{1,\varepsilon},u_{2,\varepsilon})$ satisfies Assumptions (\ref{ASSUMPTION01}). To be precise, we assume that
\begin{align}
\left\{
\begin{aligned}
(A_1).\,\,\,& u_{\varepsilon} \mbox{ have regular blowup points }\,  q_1,...q_k;\nonumber\\
(A_2).\,\,\,&\mbox{For small } \,\, \delta>0 \,\, \mbox{ such that } B(q_i,\delta)\cap B(q_j,\delta)=\emptyset, i\neq j, \nonumber\\
&\max_{B_\delta(q_{s})}u_{i,\varepsilon}(x)+2\ln\frac{1}{\varepsilon}\to+\infty\mbox{ as }\varepsilon\to0; \, s=1,...,k,  \quad i=1,2.\nonumber\\
(A_3).\,\,\,&\mbox{For $\delta>0$ as above }
\max_{B_\delta(q_s)}u_{i,\varepsilon}(x)\to-\infty\mbox{ as }\varepsilon\to0; \, s=1,...,k, \quad i=1,2.\nonumber\\
(A_4).\,\,\,&\mbox{For any }K\subset\subset\Omega\backslash \{q_{1},\cdots,q_{k}\},
\max_{K} u_{i,\varepsilon}+2\ln\frac{1}{\varepsilon}\to-\infty\mbox{ as }\varepsilon\to0, \quad i=1,2.\nonumber
\end{aligned}
\right.
\end{align}

Here, by Assumption $(A_1)$, we mean that for the sequences $\{x_{j,\varepsilon}\}_\varepsilon$  satisfying $\max_{x\in B_\delta(q_j)}u_{1,\varepsilon}(x)=u_{1,\varepsilon}(x_{j,\varepsilon})$, we get $\lim_{\varepsilon\to0}x_{j,\varepsilon}=q_j$ for $j=1,\cdots,k$.
Although (\ref{ASSUMPTION01}) was imposed on two sequences of solutions, for the convenience of the reader, we also say that a single family of solutions satisfying a similar assumption ``satisfies (\ref{ASSUMPTION01})." Clearly, we get
\begin{lemma}
\begin{itemize}
    \item [$(1)$] $(A_2)$ implies $(B_2)$;
    \item [$(2)$] $(A_4)$ implies $(B_3)$;
    \item [$(3)$] $(A_3)$ implies $(B_4)$.
\end{itemize}
\end{lemma}
Notice that we only need to verify $(B_1)$, $(B_5)$ and $(B_6)$ if we want to use the results in \cite{HuangZhang2017}. Therefore, our first objective is to rigorously verify assumptions $(B_1)$, $(B_5)$, and $(B_6)$. To carry out a more refined analysis, we introduce the following notation. For $j=1,\cdots,k$, we denote $(V_{1,j,\varepsilon},V_{2,j,\varepsilon})$ the global solutions of
\begin{equation}\label{e:GlobalLiouvilleSystem}
    \left\{
   \begin{array}{lr}
     \Delta V_{1,j,\varepsilon}+e^{u_{0,2}(x_{j,\varepsilon})}e^{V_{2,j,\varepsilon}}=0\mbox{ in }\mathbb{R}^2,\\
     \Delta V_{2,j,\varepsilon}+e^{u_{0,1}(x_{j,\varepsilon})}e^{V_{1,j,\varepsilon}}=0\mbox{ in }\mathbb{R}^2,\\
     V_{1,j,\varepsilon}(0)=u_{1,\varepsilon}(x_{j,\varepsilon})-\beta_{j,\varepsilon}\mbox{ and }V_{2,j,\varepsilon}(0)=u_{2,\varepsilon}(x_{j,\varepsilon})-\beta_{j,\varepsilon},\\
     \max_{x\in\mathbb{R}^2}V_{1,j,\varepsilon}(x)=V_{1,j,\varepsilon}(0),\\
     \widetilde{M}_{1,j,\varepsilon}=\int_{\mathbb{R}^2}e^{u_{0,2}(x_{j,\varepsilon})}e^{V_{1,j,\varepsilon}}dx,\,\widetilde{M}_{2,j,\varepsilon}=\int_{\mathbb{R}^2}e^{u_{0,1}(x_{j,\varepsilon})}e^{V_{2,j,\varepsilon}}dx.
   \end{array}
   \right.
\end{equation}
Denoting
\begin{align}
\mu_{j,\varepsilon}:=\varepsilon^{-1}e^{\frac{1}{2}\beta_{j,\varepsilon}}\mbox{ for }j=1,\cdots,k,\nonumber
\end{align}
we define $(U^*_{1,j,\varepsilon}(x),U^*_{2,j,\varepsilon}(x))=(V_{1,j,\varepsilon}(\mu_{j,\varepsilon}^{-1}(x-x_{j,\varepsilon}^*)), V_{2,j,\varepsilon}(\mu_{j,\varepsilon}^{-1}(x-x_{j,\varepsilon}^*)))$ with $x_{j,\varepsilon}^*$ satisfying
\begin{align}
\nabla U_{1,j,\varepsilon}(x_{j,\varepsilon})=\nabla u_{0,1}(x_{j,\varepsilon})+\nabla u_{0,2}(x_{j,\varepsilon})+\sum_{l\neq j}(16\pi \nabla G(x_{j,\varepsilon},x_{l,\varepsilon})-M_{1,l,\varepsilon}\nabla G(x_{j,\varepsilon},x_{l,\varepsilon})).\nonumber
\end{align}
It is easy to see that $|x_{j,\varepsilon} - x_{j,\varepsilon}^*|=O(\mu_{\varepsilon}^{-2})$. Using these settings, the error terms are defined as
\begin{align}\label{def:Eta}
\left\{
\begin{aligned}\eta_{1,j,\varepsilon}(x)&:=u_{1,\varepsilon}(x)-\beta_{j,\varepsilon}-U_{1,j,\varepsilon}^*(x)-M_{1,j,\varepsilon}(\gamma(x,x_{j,\varepsilon})-\gamma(x_{j,\varepsilon},x_{j,\varepsilon}))\\
&\quad-\sum_{l\neq j}M_{1,l,\varepsilon}(G(x,x_{l,\varepsilon})-G(x_{j,\varepsilon},x_{l,\varepsilon})),\\
\eta_{2,j,\varepsilon}(x)&:=u_{2,\varepsilon}(x)-\beta_{j,\varepsilon}-U_{2,j,\varepsilon}^*(x)-M_{2,j,\varepsilon}(\gamma(x,x_{j,\varepsilon})-\gamma(x_{j,\varepsilon},x_{j,\varepsilon}))\\
&\quad-\sum_{l\neq j}M_{2,l,\varepsilon}(G(x,x_{l,\varepsilon})-G(x_{j,\varepsilon},x_{l,\varepsilon})).
\end{aligned}
\right.
\end{align}
Here, 
\begin{align}
\left\{
\begin{aligned}
M_{1,j,\varepsilon}=\varepsilon^{-2}\int_{\Omega_j}e^{u_{2,\varepsilon}+u_{0,2}}(1-e^{u_{1,\varepsilon}+u_{0,1}})dx,\nonumber\\
M_{2,j,\varepsilon}=\varepsilon^{-2}\int_{\Omega_j}e^{u_{1,\varepsilon}+u_{0,1}}(1-e^{u_{2,\varepsilon}+u_{0,2}})dx\nonumber
\end{aligned}
\right.
\end{align}
for $j=1,\cdots,k$. Here, $\Omega_j$ are the subdomains defined as in (\ref{def:Omegaj}). Moreover, let us define
\begin{align}
\left\{
\begin{aligned}
m_{1,j,\varepsilon}=\varepsilon^{-2}\int_{B_\delta(x_{j,\varepsilon})}e^{u_{2,\varepsilon}+u_{0,2}}(1-e^{u_{1,\varepsilon}+u_{0,1}})dx,\nonumber\\
m_{2,j,\varepsilon}=\varepsilon^{-2}\int_{B_\delta(x_{j,\varepsilon})}e^{u_{1,\varepsilon}+u_{0,1}}(1-e^{u_{2,\varepsilon}+u_{0,2}})dx\nonumber
\end{aligned}
\right.
\end{align}
and
\begin{align}\label{def:fijvarepsilon}
f_{i,j,\varepsilon}(x)=u_{0,i}(x)-u_{0,i}(x_{j,\varepsilon})+M_{i,j,\varepsilon}(\gamma(x,x_{j,\varepsilon})-\gamma(x_{j,\varepsilon},x_{j,\varepsilon}))+\sum_{l\neq j}M_{i,l,\varepsilon}(G(x,x_{l,\varepsilon})-G(x_{j,\varepsilon},x_{l,\varepsilon})).
\end{align}

\,\,

To successfully prove the local uniqueness and non-degeneracy of the mean field bubbling solutions, the preliminary estimates in \cite{HuangZhang2017} must be systematically refined to match the precision of the scalar field analysis in \cite{LinYan2018}. We outline these preliminary results and state our specific refinement objectives below.

We begin by revisiting the local error estimates. By \cite[Proposition 5.1]{HuangZhang2017}, one has the following preliminary bounds.

\begin{lemma}
\begin{itemize}
    \item [$(1)$] Assuming $(B_1 -B_5)$, we get
    \begin{align}\label{ineq:Eta}
    |\eta_{i,j,\varepsilon}(x)|\leq C(1+\mu_{j,\varepsilon}|x-x_{j,\varepsilon}|)^\tau(\mu_{j,\varepsilon}^{-2\tau}+e^{\tau'\beta_{j,\varepsilon}})
    \end{align}
    for $x\in B_\delta(x_{j,\varepsilon})$, $i=1,2$, $j=1,\cdots,k$, $\tau\in(0,\frac{1}{2}]$ and $\tau'\in(0,1)$;
    \item [$(2)$] Assuming $(B_1 - B_6)$, we get (\ref{ineq:Eta}) holds for $\tau,\tau'\in(0,1)$.
\end{itemize}
\end{lemma}

Our first objective is to strictly extend the estimate (\ref{ineq:Eta}) to the full range $\tau, \tau'\in(0,1)$, mirroring the optimal regularity in \cite[(2.12)]{LinYan2018}. We will achieve this in Corollary \ref{coro:Eta-2} after rigorously verifying the refined blow-up assumptions $(B_1)$, $(B_5)$, and $(B_6)$.

For the local masses $\widetilde{M}_{i,j,\varepsilon}$, \cite[Lemma 5.1]{HuangZhang2017} provides the following division.

\begin{lemma}
Assume $(B_1 - B_5)$ hold. For any $i=1,2$ and $j=1,\cdots,k$, we get
\begin{align}
|\widetilde{M}_{i,j,\varepsilon}-8\pi|=O(\mu_{j,\varepsilon}^{-1+\theta}+e^{\frac{\beta_{j,\varepsilon}}{2}})\nonumber
\end{align}
for sufficiently small $\theta>0$.
\end{lemma}

To perform the subsequent algebraic reductions, we will employ an alternative method to strictly establish that $|\widetilde{M}_{i,j,\varepsilon}-8\pi|=O(\mu_{j,\varepsilon}^{-2+\theta})$. This sharp estimate, corresponding to \cite[(2.13)]{LinYan2018}, will be completely proved in Lemma \ref{l:M8pimu-2}.

Outside the bubbling discs, the solution admits the following expansion (\cite[(5.16), (5.5) $\And$ (5.37)]{HuangZhang2017}).

\begin{lemma}
Assume that $(B_1 - B_5)$ hold. For $i=1,2$, we get
\begin{align}
u_{i,\varepsilon}=\frac{1}{|\Omega|}\int_\Omega u_{i,\varepsilon}+\sum_{j=1}^k\widetilde{M}_{i,j,\varepsilon}G(x,x_{j,\varepsilon})+O(\mu_{j,\varepsilon}^{-2}+e^{\beta_{j,\varepsilon}})\nonumber
\end{align}
in $C^1(\Omega\backslash\cup_{j=1}^k B_\delta(x_{j,\varepsilon}))$.
\end{lemma}

This aligns directly with \cite[(2.14)]{LinYan2018} and serves as the basis for matching inner and outer approximations. 

Regarding the vanishing order of the regular part's gradient at the blow-up centers, \cite[(5.17)]{HuangZhang2017} establishes the following estimate.

\begin{lemma}
Assume $(B_1 - B_6)$, we get
\begin{align}
|\nabla f_{i,j,\varepsilon}(x_{j,\varepsilon})|=O\Big(\sum_{j=1}^k|\widetilde{M}_{i,j,\varepsilon}-8\pi|\Big)+O(|q_j - x_{j,\varepsilon}|)=O(\mu_{j,\varepsilon}^{-1}+e^{\beta_{j,\varepsilon}}).\nonumber
\end{align}
\end{lemma}

We will rigorously show that Assumption (\ref{ASSUMPTION01}) implies the sharper vanishing order $|\nabla f_{i,j,\varepsilon}(x_{j,\varepsilon})|=O(\mu_{\varepsilon}^{-2})$. This optimal bound will be proved in Lemma \ref{l:NablaF}, mirroring the requirement in \cite[(2.15)]{LinYan2018}.

Accurately evaluating the surface integrals in the Pohozaev identities requires extremely tight control over the global error. \cite[(5.21), (5.5) $\And$ (5.37)]{HuangZhang2017} gives the preliminary bound:

\begin{lemma}
Assume that $(B_1 - B_5)$. For $x\in\Omega\backslash\cup_{j=1}^k B_\delta(x_{j,\varepsilon})$, it holds that
\begin{align}
|\eta_{i,j,\varepsilon}(x)|=O(\mu_{j,\varepsilon}^{-2+\theta}+e^{\beta_{j,\varepsilon}+\theta})\nonumber
\end{align}
for suitable small $\theta>0$.
\end{lemma}

To reach the precision demonstrated in \cite[(2.16)]{LinYan2018} and successfully evaluate the surface integrals, this estimate will be significantly refined later in this section by combining our sharpened mass bounds.

Finally, the relation between the parameter $\varepsilon$ and the maximum height $\beta_{j,\varepsilon}$ is preliminarily given by \cite[(5.17)]{HuangZhang2017}:

\begin{lemma}\label{l:BetaVarepsilonAverage}
Assume $(B_1 - B_5)$ hold. For any $i=1,2$ and $j=1,\cdots,k$, we get
\begin{align}
-\beta_{j,\varepsilon}&=8\pi\gamma(x_{j,\varepsilon},x_{j,\varepsilon})+8\pi\sum_{l\neq j}G(x_{j,\varepsilon}, x_{l,\varepsilon})-I_{1,j}-4\ln\varepsilon+\frac{1}{|\Omega|}\int_\Omega u_{1,\varepsilon}+O(\mu_{j,\varepsilon}^{-2+\theta})\nonumber\\
&=8\pi\gamma(x_{j,\varepsilon},x_{j,\varepsilon})+8\pi\sum_{l\neq j}G(x_{j,\varepsilon}, x_{l,\varepsilon})-I_{2,j}-4\ln\varepsilon+\frac{1}{|\Omega|}\int_\Omega u_{2,\varepsilon}+O(\mu_{j,\varepsilon}^{-2+\theta})\nonumber
\end{align}
in $C^1(\Omega\backslash\cup_{j=1}^k B_\delta(x_{j,\varepsilon}))$.
\end{lemma}

Combining this with the algebraic constraints from \cite[p.422]{HuangZhang2017} establishes the basic relation:

\begin{lemma}
Assume that $(B_1 - B_6)$. It holds that
\begin{itemize}
    \item [$(1)$]
    \begin{align}
        (D(\vec{q})+o(1))e^{u_{0,1}(q_j)}e^{-\ln\varepsilon^2+\frac{1}{|\Omega|}\int_\Omega u_{i,\varepsilon}}+B e^{2\ln\varepsilon^2- \frac{1}{|\Omega|}\int_\Omega u_{1,\varepsilon}}+\sum_{i=1}^2\sum_{j=1}^k\frac{(\widetilde{M}_{i,j,\varepsilon}-8\pi)^2}{\widetilde{M}_{i,j,\varepsilon}-4\pi}=0.\nonumber
    \end{align}
    Here, $B$ is certain positive number;
    \item [$(2)$] Combining with Lemma \ref{l:BetaVarepsilonAverage}, we get
\begin{align}
(D(\vec{q})+o(1))e^{u_{0,1}(q_j)} & e^{-8\pi\gamma(x_{j,\varepsilon},x_{j,\varepsilon})-8\pi\sum_{l\neq j}G(x_{j,\varepsilon},x_{l,\varepsilon})}\mu_{j,\varepsilon}^{-2}\nonumber\\
&+(B+o(1))e^{8\pi\gamma(x_{j,\varepsilon},x_{j,\varepsilon})+8\pi\sum_{l\neq j}G(x_{j,\varepsilon},x_{l,\varepsilon})}e^{\beta_{j,\varepsilon}}+\sum_{i=1}^2\sum_{j=1}^k\frac{(\widetilde{M}_{i,j,\varepsilon}-8\pi)^2}{\widetilde{M}_{i,j,\varepsilon}-4\pi}=0.\nonumber
    \end{align}
\end{itemize}
\end{lemma}

To successfully eliminate the kernel elements in the linearized equations, the preliminary algebraic relation between $\varepsilon$ and the bubble height $\beta_{j,\varepsilon}$ is far from sufficient. Specifically, to cleanly evaluate the Pohozaev identities and derive the local uniqueness, we must push this expansion to a strictly higher order. We formulate this ultimate refinement as the following proposition.

\begin{proposition}\label{prop:SharpRelation}
Suppose that the sequence of solutions $u_\varepsilon=(u_{1,\varepsilon},u_{2,\varepsilon})$ satisfies Assumption (\ref{ASSUMPTION01}). Then the fundamental algebraic relation linking the scaling parameter $\mu_{j,\varepsilon}$ and the maximum height $\beta_{j,\varepsilon}$ can be refined to the following higher-order expansion:
\begin{align}
D(\vec{q})e^{u_{0,1}(q_j)} & e^{-8\pi\gamma(x_{j,\varepsilon},x_{j,\varepsilon})-8\pi\sum_{l\neq j}G(x_{j,\varepsilon},x_{l,\varepsilon})}\mu_{j,\varepsilon}^{-2}+Be^{8\pi\gamma(x_{j,\varepsilon},x_{j,\varepsilon})+8\pi\sum_{l\neq j}G(x_{j,\varepsilon},x_{l,\varepsilon})}e^{\beta_{j,\varepsilon}}=O(\mu_{j,\varepsilon}^{-4+\theta})\nonumber
\end{align}
for any $j=1,\cdots,k$ and sufficiently small $\theta>0$. Here, $B$ is a strictly positive universal constant, and $D(\vec{q})$ is the global quantity evaluated at the blow-up set $\vec{q}$.
\end{proposition}

This strictly refined relation serves as the cornerstone of our local uniqueness proof and will be rigorously derived in equation (\ref{e:VarepsilonBeta01}).

\subsection{The fully blowup: Verification of $(B_1)$ and $(B_5)$}\label{Subsection:FullyBubbling}

We first recall the classification of solutions to the following Liouville system:
\begin{equation}\label{e:LiouvilleSystem}
    \left\{
   \begin{array}{lr}
     \Delta u_1+e^{u_2}=0\mbox{ in }\mathbb{R}^2,\\
     \Delta u_2+e^{u_1}=0\mbox{ in }\mathbb{R}^2,\\
     \int_{\mathbb{R}^2} e^{u_1}dx, \int_{\mathbb{R}^2} e^{u_2}dx<\infty.
   \end{array}
   \right.
\end{equation}
Defining the parameter set
\begin{align}
\Gamma=\{(\alpha_1,\alpha_2)\in(\mathbb{R}_+)^2|\Lambda(\alpha_1,\alpha_2):=2(\alpha_1+\alpha_2)-\alpha_1\alpha_2=0\},\nonumber
\end{align}
the classical results by Chipot-Shafrir-Wolansky \cite{ChipotShafrirWolansky1997} and Lin-Zhang \cite{LinZhang2010} state the following.

\begin{theoremA}
For Problem (\ref{e:LiouvilleSystem}), the following holds:
\begin{itemize}
    \item [$(1)$] For any solution $(u_1,u_2)$ to Problem (\ref{e:LiouvilleSystem}), there exists a point $x_0\in\mathbb{R}^2$ such that both $u_1$ and $u_2$ are radially symmetric with respect to $x_0$. Moreover, $(\frac{1}{2\pi}\int_{\mathbb{R}^2} e^{u_1}dx,\frac{1}{2\pi}\int_{\mathbb{R}^2}e^{u_2}dx)\in\Gamma$;
    \item [$(2)$] For any $(\alpha_1,\alpha_2)\in\Gamma$, there exists a solution $(u_1,u_2)$ to Problem (\ref{e:LiouvilleSystem}) such that $\int_{\mathbb{R}^2}e^{u_i}dx=2\pi\alpha_i$ for $i=1,2$. Moreover, for two radial solutions $(u_1,u_2)$ and $(v_1,v_2)$ satisfying $\int_{\mathbb{R}^2} e^{u_i}dx=\int_{\mathbb{R}^2} e^{v_i}dx$ for $i=1,2$, there exists a positive constant $\delta$ such that $u_i(\delta x)+2\ln\delta=v_i(x)$ for $i=1,2$.
\end{itemize}
\end{theoremA}

\begin{theoremB}\label{t:LiouvilleNonDegeneracy}
Let $(u_1,u_2)$ be a radial solution to Problem (\ref{e:LiouvilleSystem}). Consider the linearized problem 
\begin{equation}\label{e:LinearizedLiouvilleSystem}
    \left\{
   \begin{array}{lr}
     \Delta \phi_1+e^{u_2}\phi_2=0\mbox{ in }\mathbb{R}^2,\\
     \Delta \phi_2+e^{u_1}\phi_1=0\mbox{ in }\mathbb{R}^2,\\
     |\phi_1(x)|,|\phi_2(x)|\leq C(1+|x|)^\tau\mbox{ for some small }\tau>0.
   \end{array}
   \right.
\end{equation}
Then, $(\phi_1,\phi_2)$ is a linear combination of the kernel elements $(\partial_{x_1}u_1,\partial_{x_1}u_2)$, $(\partial_{x_2}u_1,\partial_{x_2}u_2)$, and $(\frac{d}{d\delta}u_1(\delta x)\big|_{\delta=1}+2, \frac{d}{d\delta}u_2(\delta x)\big|_{\delta=1}+2)$.
\end{theoremB}

With these preliminaries, we establish the \emph{fully bubbling} behavior. Recall our scaling assumptions:
\begin{align}
\mu_{1,\varepsilon}=\varepsilon^{-1}e^{\frac{1}{2}\beta_{1,\varepsilon}}\to+\infty\nonumber
\end{align}
where $\beta_{1,\varepsilon}:=\max_{B_\delta(q_1)}u_{1,\varepsilon}\to+\infty$. Defining the rescaled functions
\begin{align}
\overline{u}_{i,\varepsilon}(x)=u_{i,\varepsilon}(\mu_{1,\varepsilon}^{-1}x+x_{1,\varepsilon})-\beta_{1,\varepsilon},
\end{align}
where $u_{1,\varepsilon}(x_{1,\varepsilon})=\max_{B_\delta(q_1)}u_{1,\varepsilon}$, we obtain the corresponding system:
\begin{equation}
    \left\{
   \begin{array}{lr}
     \Delta \overline{u}_{1,\varepsilon}+e^{\overline{u}_{2,\varepsilon}+u_{0,2}}(1-e^{\beta_{1,\varepsilon}}e^{\overline{u}_{1,\varepsilon}+u_{0,1}})=4\pi N\mu_{1,\varepsilon}^{-2}\varepsilon^2 e^{-\beta_{1,\varepsilon}},\\
     \Delta \overline{u}_{2,\varepsilon}+e^{\overline{u}_{1,\varepsilon}+u_{0,1}}(1-e^{\beta_{1,\varepsilon}} e^{\overline{u}_{2,\varepsilon}+u_{0,2}})=4\pi N\mu_{1,\varepsilon}^{-2}\varepsilon^2 e^{-\beta_{1,\varepsilon}}.
   \end{array}
   \right.
\end{equation}

Under Assumption (\ref{ASSUMPTION01}), the limit of the rescaled solutions converges to a Liouville system, confirming their ``mean field'' nature. Since the scaling is centered at a maximum point, the limit functions $(U_1,U_2)$ are strictly radially symmetric.

\begin{lemma}\label{l:FullyBubbling}
(Fully bubbling) Under the above assumptions, for any $i=1,2$ and any $j=1,\cdots,k$, we have
\begin{align}
(\overline{u}_{1,\varepsilon},\overline{u}_{2,\varepsilon})\to(U_1,U_2)\mbox{ in }C_{loc}^2(\mathbb{R}^2)\nonumber
\end{align}
where $(U_1,U_2)$ satisfies
\begin{equation}
    \left\{
   \begin{array}{lr}
     \Delta U_1+e^{u_{0,2}(q_j)+U_2}=0\mbox{ in }\mathbb{R}^2,\nonumber\\
     \Delta U_2+e^{u_{0,1}(q_j)+U_1}=0\mbox{ in }\mathbb{R}^2,\nonumber\\
     \int_{\mathbb{R}^2} e^{U_1}dx, \int_{\mathbb{R}^2} e^{U_2}dx<\infty.\nonumber
   \end{array}
   \right.
\end{equation}
\end{lemma}

An immediate corollary of the radial symmetry is the coincidence of the blow-up sets.
\begin{corollary}\label{coro:S1=S2}
Under the above assumptions, we have $\lim_{\varepsilon\to 0} x_{1,\varepsilon} = \lim_{\varepsilon\to 0} x_{2,\varepsilon} = q_j$.
\end{corollary}

\vs

\noindent{\bf Proof of Lemma \ref{l:FullyBubbling}.}
Arguing by contradiction, suppose that 
\begin{align}
(\overline{u}_{1,\varepsilon},\overline{u}_{2,\varepsilon})\to(U_1,-\infty)\mbox{ in }C_{loc}(\mathbb{R}^2).\nonumber
\end{align}
Passing to the limit $\varepsilon\to0+$, the rescaling procedure forces $\Delta U_1=0$ in $\mathbb{R}^2$ with $U_1$ strictly bounded from above. By Harnack's inequality and Liouville's theorem, we deduce:
\begin{claim}
$U_1$ is constant in $\mathbb{R}^2$.
\end{claim}
Conversely, the rescaling procedure guarantees the uniform integrability bound:
\begin{claim}
For any domain $\Omega'\subset\mathbb{R}^2$, $\int_{\Omega'}e^{U_1}dx\leq 4\pi N_1 +1$.
\end{claim}
This contradicts the constancy of $U_1$ over $\mathbb{R}^2$, thus concluding the proof.

\begin{flushright}
$\Box$
\end{flushright}

\vs

\begin{corollary}
Consequently, assumptions $(B_1)$ and $(B_5)$ are rigorously verified.
\end{corollary}

\subsection{Division of the masses}
As established above, the limit system at the blow-up points $q_j$ is given by:
\begin{equation}
    \left\{
   \begin{array}{lr}
     \Delta U_{1}+e^{u_{0,2}(q_j)+U_{2}}=0\mbox{ in }\mathbb{R}^2,\nonumber\\
     \Delta U_{2}+e^{u_{0,1}(q_j)+U_{1}}=0\mbox{ in }\mathbb{R}^2,\nonumber\\
     e^{u_{0,1}(q_j)}\int_{\mathbb{R}^2} e^{U_1}dx, e^{u_{0,2}(q_j)}\int_{\mathbb{R}^2} e^{U_2}dx<\infty.\nonumber
   \end{array}
   \right.
\end{equation}
Denoting the solutions to this limit system by $(U_{j,1},U_{j,2})$ for $j=1,\cdots,k$, we derive the following sharp estimates for the local masses.

Following \cite[Lemmas 5.1, 5.2 $\And$ 5.3]{HuangZhang2017}, we get for any $j,j'=1,\cdots,k$ with $j\neq j'$
\begin{align}
\beta_{j,\varepsilon}-\beta_{j',\varepsilon}&=-8\pi\big(\gamma(x_{j,\varepsilon},x_{j,\varepsilon})-\gamma(x_{j',\varepsilon},x_{j',\varepsilon})\big)-8\pi\Big(\sum_{l\neq j}G(x_{j,\varepsilon},x_{l,\varepsilon})-\sum_{l'\neq j}G(x_{j',\varepsilon},x_{l',\varepsilon})\Big)\nonumber\\
&+(I_{1,j,\varepsilon}-I_{2,j',\varepsilon})+O(\mu^{\delta-2})\nonumber
\end{align}
and
\begin{align}\label{ineq:MMComparePreliminary}
|\widetilde{M}_{1,j,\varepsilon}-\widetilde{M}_{1,j',\varepsilon}|+|\widetilde{M}_{2,j,\varepsilon}-\widetilde{M}_{2,j',\varepsilon}|=O(\mu^{\delta-1}).
\end{align}
Using \cite{LinZhang2010}, we get $I_{1,j,\varepsilon}-I_{2,j',\varepsilon}=O(|\widetilde{M}_{1,j,\varepsilon}-\widetilde{M}_{1,j',\varepsilon}|+|\widetilde{M}_{2,j,\varepsilon}-\widetilde{M}_{2,j',\varepsilon}|)=O(\mu^{\delta-1})$. This implies that
\begin{align}\label{ineq:BetaComparePreliminary}
\beta_{j_1,\varepsilon}-\beta_{j_2,\varepsilon}&=-8\pi\big(\gamma(x_{j_1,\varepsilon},x_{j_1,\varepsilon})-\gamma(x_{j_2,\varepsilon},x_{j_2,\varepsilon})\big)-8\pi\Big(\sum_{l_1\neq j_1}G(x_{j_1,\varepsilon},x_{l_1,\varepsilon})-\sum_{l_2\neq j_2}G(x_{j_2,\varepsilon},x_{l_2,\varepsilon})\Big)\nonumber\\
&+O(\mu^{\delta-1}).
\end{align}
In follows we use the idea in \cite{ChengLiZhang2025Preprint} to improve the above errors.

\begin{lemma}\label{l:AlphaMu-2}
For any $j,j'=1,\cdots,k$, it holds that 
$e^{u_{0,1}(q_{j})}\int_{\mathbb{R}^2}e^{U_{j,1,\varepsilon}}dx-e^{u_{0,1}(q_{j'})}\int_{\mathbb{R}^2}e^{U_{j',1,\varepsilon}}dx=O(\mu_{j,\varepsilon}^{-2})$ and 
$e^{u_{0,2}(q_{j})}\int_{\mathbb{R}^2}e^{U_{j,2,\varepsilon}}dx-e^{u_{0,2}(q_{j'})}\int_{\mathbb{R}^2}e^{U_{j',2,\varepsilon}}dx=O(\mu_{j,\varepsilon}^{-2})$.
\end{lemma}
\noindent{\bf Proof.}
Let $(U_{1,1,\varepsilon},U_{1,2,\varepsilon})$ and $(U_{2,1,\varepsilon},U_{2,2,\varepsilon})$ denote the approximate solutions at the blow-up points $q_{1}$ and $q_{2}$, respectively satisfying:
\begin{equation}
    \left\{
   \begin{array}{lr}
     \Delta U_{1,1,\varepsilon}+e^{u_{0,2}(q_{1})+U_{1,2,\varepsilon}}=0\mbox{ in }\mathbb{R}^2,\nonumber\\
     \Delta U_{1,2,\varepsilon}+e^{u_{0,1}(q_{1})+U_{1,1,\varepsilon}}=0\mbox{ in }\mathbb{R}^2,\nonumber\\
     e^{u_{0,1}(q_{1})}\int_{\mathbb{R}^2} e^{U_{1,1,\varepsilon}}dx=M_{1,1,\varepsilon},\, e^{u_{0,2}(q_{1})}\int_{\mathbb{R}^2} e^{U_{1,2,\varepsilon}}dx=M_{1,2,\varepsilon}\nonumber
   \end{array}
   \right.
\end{equation}
and
\begin{equation}
    \left\{
   \begin{array}{lr}
     \Delta U_{2,1,\varepsilon}+e^{u_{0,2}(q_{2})+U_{2,2,\varepsilon}}=0\mbox{ in }\mathbb{R}^2,\nonumber\\
     \Delta U_{2,2,\varepsilon}+e^{u_{0,1}(q_{2})+U_{2,1,\varepsilon}}=0\mbox{ in }\mathbb{R}^2,\nonumber\\
     e^{u_{0,1}(q_{2})}\int_{\mathbb{R}^2} e^{U_{2,1,\varepsilon}}dx=M_{2,1,\varepsilon},\, e^{u_{0,2}(q_{2})}\int_{\mathbb{R}^2} e^{U_{2,2,\varepsilon}}dx=M_{2,2,\varepsilon}.\nonumber
   \end{array}
   \right.
\end{equation}

Setting $\eta_\varepsilon=(u_{0,1}(q_{1})+\beta_{1,\varepsilon})-(u_{0,1}(q_{2})+\beta_{2,\varepsilon})$, we define the difference function:
\begin{align}
w_{i,\varepsilon}(y):=U_{1,i,\varepsilon}(y)-\Big(U_{2,i,\varepsilon}(e^\frac{\eta_\varepsilon}{2} y)+\eta_\varepsilon\Big).\nonumber
\end{align}

This formulation naturally leads to the following pointwise bound.
\begin{claim}\label{c:U1-U2}
$|w_{i,\varepsilon}(y)|=O(\mu_{j,\varepsilon}^{-2}+\mu_{j,\varepsilon}^{-\delta}\alpha)(1+|y|)^{3\delta}$ for sufficiently small $\delta>0$, where $\alpha=\sum_{i=1}^2|M_{1,i,\varepsilon}-M_{2,i,\varepsilon}|$.
\end{claim}
\noindent{\bf Proof of Claim \ref{c:U1-U2}.}
Suppose, for the sake of contradiction, that the supreme
\begin{align}
\Lambda_{\varepsilon}=\sup_{i=1,2}\sup_{y\in B_{\mu_{j,\varepsilon}\theta}(0)}\frac{w_{i,\varepsilon}(y)}{(\mu_{j,\varepsilon}^{-2}+\mu_{j,\varepsilon}^{-\delta}\alpha)(1+|y|)^{3\delta}}\to+\infty.
\end{align}
Let $y_\varepsilon^*$ denote the point achieving its maximum. Normalizing the solution by setting
\begin{align}
\overline{w}_{i,\varepsilon}(y):=\frac{w_{i,\varepsilon}(y)}{\Lambda_{\varepsilon}(\mu_{j,\varepsilon}^{-2}+\mu_{j,\varepsilon}^{-\delta}\alpha)(1+|y|)^{3\delta}},\nonumber
\end{align}
an elementary observation implies the global bound $|\overline{w}_{i,\varepsilon}(y)|\leq\frac{(1+|y|)^{3\delta}}{(1+|y_\varepsilon^*|)^{3\delta}}$. The normalized functions satisfy:
\begin{equation}
    \left\{
   \begin{array}{lr}
     \Delta \overline{w}_{1,\varepsilon}+e^{\xi_{1,2}(r)}\overline{w}_{2,\varepsilon}=0\mbox{ for }|y|\leq\theta\mu_{j,\varepsilon},\nonumber\\
     \Delta \overline{w}_{2,\varepsilon}+e^{\xi_{1,1}(r)}\overline{w}_{1,\varepsilon}=0\mbox{ for }|y|\leq\theta\mu_{j,\varepsilon},\nonumber\\
     \overline{w}_{i,\varepsilon}(y)|_{|y|=\theta\mu_{j,\varepsilon}}=O(\mu_{j,\varepsilon}^{-2}\ln\mu_{j,\varepsilon}+\alpha)\mbox{ for }i=1,2.\nonumber
   \end{array}
   \right.
\end{equation}
Here, $\xi_{1,i}(r)$ arises from the Mean Value Theorem and tends to $U_{1,i}(r)$ as $\varepsilon\to0$. Following \cite{Zhang2009}, we proceed by analyzing two distinct cases regarding the sequence $\{y_\varepsilon^*\}_\varepsilon$.

\,\,

\noindent{\bf Case 1. $\{|y_\varepsilon^*|\}_\varepsilon$ is bounded.} Passing to the limit as $\varepsilon\to0$, we obtain $\overline{w}_{i,\varepsilon}\to\overline{w}_{i,0}\mbox{ in }C_{loc}^2(\mathbb{R}^2)$ with the limit satisfying:
\begin{equation}
    \left\{
   \begin{array}{lr}
     \Delta \overline{w}_{1,0}+e^{U_{1,2}(y)}\overline{w}_{2,0}=0\mbox{ for }y\in\mathbb{R}^2,\nonumber\\
     \Delta \overline{w}_{2,0}+e^{U_{1,1}(y)}\overline{w}_{1,0}=0\mbox{ for }y\in\mathbb{R}^2,\nonumber\\
     \overline{w}_{i,0}\mbox{ is radial, }\overline{w}_{1,0}(0)=0\mbox{ and }\overline{w}_{2,0}(0)=O(\alpha).\nonumber
   \end{array}
   \right.
\end{equation}

By the classification of Liouville systems \cite{LinZhang2010}, the boundary condition $\overline{w}_{2,0}(0)=O(\alpha)$ is strictly maintained. From \cite[Lemma 5.1]{HuangZhang2017}, we know $\alpha=O(\mu_{\varepsilon,j}^{-1+\theta})$ for some small $\theta>0$. 

Applying Theorem \ref{t:LiouvilleNonDegeneracy}, we deduce that $\overline{w}_{i,0}\equiv 0$ for $i=1,2$. This directly contradicts the definition of $\overline{w}_{i,\varepsilon}$ maximizing at $y_\varepsilon^*$ and the boundedness of $\{|y_\varepsilon^*|\}_\varepsilon$.

\,\,

\noindent{\bf Case 2. $\{|y_\varepsilon^*|\}_\varepsilon$ is unbounded.} 
In this regime, we employ the classical potential analysis established in \cite{Zhang2009,LinZhang2013}. Utilizing Green's representation, we obtain:
\begin{align}
\pm 1=\overline{w}_{i,\varepsilon}(y_\varepsilon^*)-\overline{w}_{i,\varepsilon}(0)=\int_{B_{\theta\mu_{j,\varepsilon}}(0)}(G(y_\varepsilon^*,\eta)-G(0,\eta))e^{\xi_{1,i'}(\eta)}\overline{w}_{i',\varepsilon}(\eta)d\eta+O(\mu_{j,\varepsilon}^{-2}\ln\mu_{j,\varepsilon}+\alpha).\nonumber
\end{align}
where $G(y,\eta)$ denotes the Green's function on $B_{\theta\mu_{j,\varepsilon}}(0)$. Recalling the sharp estimates from \cite[Lemma 3.2]{LinZhang2013}, the Green's function difference is bounded by:
\begin{equation}
    |G(y,\eta)-G(0,\eta)|\leq\left\{
   \begin{array}{lr}
     C(\ln|y|+|\ln|\eta||)\mbox{ for }\eta\in\Sigma_1:=\{\eta\in B_{\theta\mu_{j,\varepsilon}}(0)||\eta|<\frac{|y|}{2}\},\nonumber\\
    C(\ln|y|+|\ln|y-\eta||)\mbox{ for }\eta\in\Sigma_2:=\{\eta\in B_{\theta\mu_{j,\varepsilon}}(0)||y-\eta|<\frac{|y|}{2}\},\nonumber\\
     C\frac{|y|}{|\eta|}\mbox{ for }\eta\in\Sigma_3:= B_{\theta\mu_{j,\varepsilon}}(0)\backslash(\Sigma_1\cup\Sigma_2).\nonumber
   \end{array}
   \right.
\end{equation}
Consequently, we partition the integral domain to evaluate the upper bound:
\begin{align}\label{ineq:GGwTOTAL}
\frac{1}{2}&\leq\int_{B_{\theta\mu_{j,\varepsilon}}(0)}|G(y_\varepsilon^*,\eta)-G(0,\eta)|e^{\xi_{1,i'}(\eta)}|\overline{w}_{i',\varepsilon}(\eta)|d\eta\nonumber\\
&=\int_{\Sigma_1}+\int_{\Sigma_2}+\int_{\Sigma_3}|G(y_\varepsilon^*,\eta)-G(0,\eta)|e^{\xi_{1,i'}(\eta)}|\overline{w}_{i',\varepsilon}(\eta)|d\eta:=I_1+I_2+I_3.
\end{align}
Direct evaluation of these integrals yields:
\begin{align}
I_1\leq C\int_{|\eta|<|y_\varepsilon^*|/2}\frac{(\ln|y_\varepsilon^*|+|\ln|\eta||)(1+|\eta|)^{-4+4\delta}}{(1+|y_\varepsilon^*|)^{3\delta}}d\eta=o_\varepsilon(1).\nonumber
\end{align}
A symmetric computation confirms $I_2=o_\varepsilon(1)$, while the tail integral evaluates to:
\begin{align}
I_3\leq C\int_{\Sigma_3}\frac{|y_\varepsilon^*|(1+|\eta|)^{-4+3\delta}}{|\eta|}d\eta=o_\varepsilon(1).\nonumber
\end{align}
This sum violates the lower bound in (\ref{ineq:GGwTOTAL}), thereby concluding the proof of Claim \ref{c:U1-U2}. Finally, Lemma \ref{l:AlphaMu-2} follows directly by evaluating Claim \ref{c:U1-U2} at $y=0$.

\begin{flushright}
$\Box$
\end{flushright}

\subsection{Solutions satisfying Assumption (\ref{ASSUMPTION01}) are fully bubbling in the sense of \cite{HuangZhang2017}: Completion of the refined error estimates}

\begin{lemma}
For any blowup solution $(u_{1,\varepsilon},u_{2,\varepsilon})$ satisfying Assumption (\ref{ASSUMPTION01}), it satisfies Assumption $(B_1)$.
\end{lemma}

Furthermore, Lemma \ref{l:FullyBubbling} guarantees the following condition:
\begin{lemma}
For any blowup solution $(u_{1,\varepsilon},u_{2,\varepsilon})$ satisfying Assumption (\ref{ASSUMPTION01}), it satisfies Assumption $(B_2)$.
\end{lemma}

To verify Assumption $(B_6)$, we combine the established results with those from \cite{HuangZhang2017} that are independent of $(B_6)$. We first derive the following preliminary estimate.
\begin{lemma}\label{l:PohozaevMu-1+delta}
Without Assumption $(B_6)$, it holds that
\begin{align}
\partial_h\Bigg[\big(u_{0,1}(x_{j,\varepsilon})+u_{0,2}(x_{j,\varepsilon})\big)+16\pi\Big(\gamma(x_{j,\varepsilon},x_{j,\varepsilon})+\sum_{l\neq j}G(x_{l,\varepsilon},x_{j,\varepsilon})\Big)\Bigg]=O(\mu_{j,\varepsilon}^{-1+\delta}).\nonumber
\end{align}
Here, $\delta$ is a generically small positive constant.
\end{lemma}

Invoking standard non-degeneracy arguments, this bound forces the blow-up centers to converge at a controlled rate:
\begin{corollary}\label{coro:xj-qj00}
Assuming $q=(q_1,\cdots,q_k)$ is a non-degenerate critical point of
\begin{align}
\sum_{j=1}^k\big(u_{0,1}(q_j)+u_{0,2}(q_j)\big)+16\pi\sum_{i=1}^k\Big(\gamma(q_i,q_i)+\sum_{j\neq i}G(q_i,q_j)\Big),\nonumber
\end{align}
we get $|x_{j,\varepsilon}-q_j|=O(\mu_{j,\varepsilon}^{-1+\delta})$. Here, $\delta$ is a generically small positive constant.
\end{corollary}

Adapting the framework from \cite[pp. 420]{HuangZhang2017}, we obtain the local error decay:
\begin{corollary}\label{coro:Eta-2}
It holds that
\begin{align}\label{e:EtaSize}
|\eta_{i,j,\varepsilon}(x)|=O(\mu_{j,\varepsilon}^{-2+\delta}(1+\mu_{j,\varepsilon}|x-x_{j,\varepsilon}|)^{\delta})
\end{align}
where $\delta$ is a generically small positive constant.
\end{corollary}

\begin{remark}
Therefore, the refined error estimate (\ref{ineq:EtaFINAL}) stated in Proposition \ref{prop:RefinedEstimates} (1) is established for $\tau, \tau' \in (0,1)$.
\end{remark}

\,\,

\noindent{\bf Proof of Lemma \ref{l:PohozaevMu-1+delta}.}
Following the scaling methods developed in \cite[Theorem 5.1]{LinZhang2013} and \cite[Lemma 3.2]{HuangZhang2017}, we utilize \cite[(4.7) $\And$ (5.38)]{HuangZhang2017} to expand the solutions for $x\in B_\delta(x_{j,\varepsilon})$ as:
\begin{align}\label{e:uLocalExpansion}
u_{i,\varepsilon}(x)=\beta_{j,\varepsilon}+U_{i,j,\varepsilon}^*(x)+M_{i,j,\varepsilon}\Big(\gamma(x,x_{j,\varepsilon})-\gamma(x_{j,\varepsilon},x_{j,\varepsilon})\Big)\nonumber\\
+\sum_{l\neq j}M_{i,l,\varepsilon}\Big(G(x,x_{l,\varepsilon})-G(x_{j,\varepsilon},x_{l,\varepsilon})\Big)+\eta_{i,j,\varepsilon}(x)
\end{align}
where the error term exhibits the bound
\begin{align}\label{e:EtaSize01}
|\eta_{i,j,\varepsilon}(x)|=O(\mu_{j,\varepsilon}^{-1+\delta}(1+\mu_{j,\varepsilon}|x-x_{j,\varepsilon}|)^{\delta})
\end{align}
for small $\delta>0$. The rescaled profile is defined as $U_{i,j,\varepsilon}^*(x)=V_{i,j,\varepsilon}(\mu_{j,\varepsilon}(x-x_{j,\varepsilon}^*))$ with the center shift bounded by $|x_{j,\varepsilon}-x_{j,\varepsilon}^*|=O(\mu_{j,\varepsilon}^{-2})$. Because this expansion strictly relies on \cite[Proposition 5.1]{HuangZhang2017}, it safely bypasses Assumption $(B_3)$. We now evaluate the Pohozaev identity \cite[(3.16)]{HuangZhang2017}:
\begin{align}\label{e:Pohozaev01}
&\int_{\partial B_\delta(x_{j,\varepsilon})}\Bigg(\Big<\nu,\nabla\overline{u}_{1,\varepsilon}\Big>\partial_h\overline{u}_{2,\varepsilon}+\Big<\nu,\nabla\overline{u}_{2,\varepsilon}\Big>\partial_h\overline{u}_{1,\varepsilon}-\nu_h\Big<\nabla\overline{u}_{1,\varepsilon},\nabla\overline{u}_{2,\varepsilon}\Big>\Bigg)d\mathcal{H}^1\nonumber\\
&=\frac{1}{\varepsilon^2}\int_{B_\delta(x_{j,\varepsilon})}\Bigg[e^{u_{2,\varepsilon}+u_{0,2}}\Big(1-e^{u_{1,\varepsilon}+u_{0,1}}\Big)\Bigg(\partial_h u_{0,2}+\partial_h M_{2,j,\varepsilon}\Big(\gamma(x,x_{j,\varepsilon})-\gamma(x_{j,\varepsilon},x_{j,\varepsilon})\Big)\nonumber\\
&\quad\quad\quad +\partial_h\sum_{l\neq j}M_{2,l,\varepsilon}\Big(G(x,x_{l,\varepsilon})-G(x_{j,\varepsilon},x_{l,\varepsilon})\Big)+\frac{2N_2\pi(x-x_{j,\varepsilon})_h}{|\Omega|}\Bigg)\nonumber\\
&\quad+\Bigg(e^{u_{1,\varepsilon}+u_{0,1}}\Big(1-e^{u_{2,\varepsilon}+u_{0,2}}\Big)\Bigg(\partial_h u_{0,1}+\partial_h M_{1,j,\varepsilon}\Big(\gamma(x,x_{j,\varepsilon})-\gamma(x_{j,\varepsilon},x_{j,\varepsilon})\Big)\nonumber\\
&\quad\quad\quad+\partial_h\sum_{l\neq j}M_{1,l,\varepsilon}\Big(G(x,x_{l,\varepsilon})-G(x_{j,\varepsilon},x_{l,\varepsilon})\Big)+\frac{2N_1\pi(x-x_{j,\varepsilon})_h}{|\Omega|}\Bigg)\Bigg]dx\nonumber\\
&\quad-\frac{1}{\varepsilon^2}\int_{\partial B_\delta(x_{j,\varepsilon})}\nu_h\Big(e^{u_{1,\varepsilon}+u_{0,1}}+e^{u_{2,\varepsilon}+u_{0,2}}-e^{u_{1,\varepsilon}+u_{0,1}+u_{2,\varepsilon}+u_{0,2}}\Big)d\mathcal{H}^1.
\end{align}
Here, the modified regular part is
\begin{align*}
\overline{u}_{i,\varepsilon}&=u_{i,\varepsilon}-M_{i,j,\varepsilon}\Big(\gamma(x,x_{j,\varepsilon})-\gamma(x_{j,\varepsilon},x_{j,\varepsilon})\Big)\\
&-\sum_{l\neq j}M_{i,l,\varepsilon}\Big(G(x,x_{l,\varepsilon})-G(x_{j,\varepsilon},x_{l,\varepsilon})\Big)-\frac{N_i\pi|x-q_j|^2}{|\Omega|}\nonumber
\end{align*}
for $x\in B_\delta(x_{j,\varepsilon})$, which satisfies the uncoupled system:
\begin{equation}\label{e:002}
    \left\{
   \begin{array}{lr}
    \Delta \overline{u}_{1,\varepsilon}+\frac{1}{\varepsilon^2}h_2(x)e^{\overline{u}_{2,\varepsilon}}(1-h_1(x)e^{\overline{u}_{1,\varepsilon}})=0,\\
    \\
    \Delta \overline{u}_{2,\varepsilon}+\frac{1}{\varepsilon^2}h_1(x)e^{\overline{u}_{1,\varepsilon}}(1-h_2(x)e^{\overline{u}_{2,\varepsilon}})=0
   \end{array}
   \right.
\end{equation}
with the background coefficients
\begin{align*}
h_i(x)=&\ e^{\frac{N_i|x-q_j|^2}{|\Omega|}+u_{0,i}(x)+M_{i,j,\varepsilon}\Big(\gamma(x,x_{j,\varepsilon})-\gamma(x_{j,\varepsilon},x_{j,\varepsilon})\Big)}\nonumber\\
&\cdot e^{\sum_{l\neq j}M_{i,l,\varepsilon}\Big(G(x,x_{l,\varepsilon})-G(x_{j,\varepsilon},x_{l,\varepsilon})\Big)}.\nonumber
\end{align*}
Because the double exponential terms decay at order $O(\mu_{j,\varepsilon}^{-2})$, the right-hand side of the Pohozaev identity reduces to:
\begin{align}\label{e:Pohozaev01RHS}
\mbox{RHS of (\ref{e:Pohozaev01})}&=\frac{1}{\varepsilon^2}\int_{B_\delta(x_{j,\varepsilon})}e^{u_{2,\varepsilon}+u_{0,2}}\Bigg(\partial_h u_{0,2}+\partial_hM_{2,j,\varepsilon}\Big(\gamma(x,x_{j,\varepsilon})-\gamma(x_{j,\varepsilon},x_{j,\varepsilon})\Big)\nonumber\\
&\quad\quad\quad+\partial_h\sum_{l\neq j}M_{2,l,\varepsilon}\Big(G(x,x_{l,\varepsilon})-G(x_{j,\varepsilon},x_{l,\varepsilon})\Big)+\frac{2N_2\pi(x-x_{j,\varepsilon})_h}{|\Omega|}\Bigg)dx\nonumber\\
&+\frac{1}{\varepsilon^2}\int_{B_\delta(x_{j,\varepsilon})}e^{u_{1,\varepsilon}+u_{0,1}}\Bigg(\partial_h u_{0,1}+\partial_hM_{1,j,\varepsilon}\Big(\gamma(x,x_{j,\varepsilon})-\gamma(x_{j,\varepsilon},x_{j,\varepsilon})\Big)\nonumber\\
&\quad\quad\quad+\partial_h\sum_{l\neq j}M_{1,l,\varepsilon}\Big(G(x,x_{l,\varepsilon})-G(x_{j,\varepsilon},x_{l,\varepsilon})\Big)+\frac{2N_1\pi(x-x_{j,\varepsilon})_h}{|\Omega|}\Bigg)dx\nonumber\\
&\quad-\frac{1}{\varepsilon^2}\int_{\partial B_\delta(x_{j,\varepsilon})}\nu_h\Big(e^{u_{1,\varepsilon}+u_{0,1}}+e^{u_{2,\varepsilon}+u_{0,2}}\Big)d\mathcal{H}^1+O(\mu_{j,\varepsilon}^{-2})\nonumber\\
&=\frac{1}{\varepsilon^2}\int_{B_\delta(x_{j,\varepsilon})}e^{u_{2,\varepsilon}+u_{0,2}}\partial_h u_{0,2}dx+\frac{1}{\varepsilon^2}\int_{B_\delta(x_{j,\varepsilon})}e^{u_{1,\varepsilon}+u_{0,1}}\partial_h u_{0,1}dx+O(\mu_{j,\varepsilon}^{-1+\delta})\nonumber\\
&=8\pi\Big[ \partial_h u_{0,1}(x_{j,\varepsilon})+ \partial_h u_{0,2}(x_{j,\varepsilon})+16\pi \Big(\partial_h\gamma(x_{j,\varepsilon},x_{j,\varepsilon})+\partial_h\sum_{l\neq j}G(x_{l,\varepsilon},x_{j,\varepsilon})\Big)\Big]+O(\mu_{j,\varepsilon}^{-1+\delta}).
\end{align}
This algebraic reduction relies on the expansion (\ref{e:uLocalExpansion}), the error bound (\ref{e:EtaSize01}), and the mass estimates in Lemma \ref{l:M8pimu-2}. Symmetrically, utilizing \cite[(5.6)]{HuangZhang2017}, the boundary term resolves as:
\begin{align}\label{e:Pohozaev01LHS}
\mbox{LHS of (\ref{e:Pohozaev01})}=O(\mu_{j,\varepsilon}^{-1}).
\end{align}
Equating (\ref{e:Pohozaev01RHS}) and (\ref{e:Pohozaev01LHS}) strictly verifies the lemma.

\begin{flushright}
$\Box$
\end{flushright}

Building upon Corollary \ref{coro:Eta-2}, we now refine the mass estimates from \cite[(5.37)]{HuangZhang2017} to obtain a sharp $O(\mu_{j,\varepsilon}^{-2})$ bound.
\begin{lemma}\label{l:M8pimu-2}
For $i=1,2$ and $j=1,2,\cdots,k$, we get $\widetilde{M}_{i,j,\varepsilon}-8\pi=O(\mu_{j,\varepsilon}^{-2})$.
\end{lemma}

\begin{remark}
This directly establishes the sharp mass estimate stated in Proposition \ref{prop:RefinedEstimates} (2).
\end{remark}

\noindent{\bf Proof of Lemma \ref{l:M8pimu-2}.}
We proceed with $i=1$ without loss of generality. Integrating the first equation of (\ref{e:001}) over $\Omega$, we isolate three key analytical components:
\begin{itemize}
    \item [$(A)$] $\int_\Omega \Delta u_{1,\varepsilon}dx=0$ via the divergence theorem;
    \item [$(B)$] $\varepsilon^{-2}\int_\Omega e^{u_{2,\varepsilon}+ u_{1,\varepsilon}+u_{0,1}+u_{0,2}}dx=O(\mu_{j,\varepsilon}^{-2})$;
    \item [$(C)$] $\varepsilon^{-2}\int_\Omega e^{u_{2,\varepsilon}+u_{0,2}}dx=\sum_{j=1}^k\widetilde{M}_{2,j,\varepsilon}+O(\mu_{j,\varepsilon}^{-2})$ as dictated by \cite[(5.16), (5.17) $\And$ (5.38)]{HuangZhang2017} alongside Claim \ref{c:U1-U2}.
\end{itemize}

While Assertions $(A)$ and $(B)$ hold trivially, we verify Assertion $(C)$ through a direct domain decomposition:
\begin{align}
\varepsilon^{-2}\int_\Omega e^{u_{2,\varepsilon}+u_{0,2}}dx=\sum_{j=1}^k \varepsilon^{-2}\int_{B(x_{j,\varepsilon},\delta)}e^{u_{2,\varepsilon}+u_{0,2}}dx+\varepsilon^{-2}\int_{\Omega\backslash\cup_{j=1}^k B(x_{j,\varepsilon},\delta)}e^{u_{2,\varepsilon}+u_{0,2}}dx.\nonumber
\end{align}
The outer integral exhibits quadratic decay, bounded by:
\begin{align}
\varepsilon^{-2}\int_{\Omega\backslash\cup_{j=1}^k B(x_{j,\varepsilon},\delta)}e^{u_{2,\varepsilon}+u_{0,2}}dx=O(\mu_{j,\varepsilon}^{-2}),\nonumber
\end{align}
while the inner integral captures the local mass behavior:
\begin{align}
\varepsilon^{-2}\int_{B(x_{j,\varepsilon},\delta)}e^{u_{2,\varepsilon}+u_{0,2}}dx=\int_{B(x_{j,\varepsilon},\delta)}e^{U_{2,j,\varepsilon}^*+f_j+\eta_{2,j,\varepsilon}+\beta_{j,\varepsilon}-2\ln\varepsilon}dx=\widetilde{M}_{2,j,\varepsilon}+O(\mu^{-2}_{j,\varepsilon})\nonumber
\end{align}
where the regular term incorporates the Green's function interactions:
\begin{align}
f_{j}(x)=M_{2,j,\varepsilon}\Big(\gamma(x,x_{j,\varepsilon})-\gamma(x_{j,\varepsilon},x_{j,\varepsilon})\Big)
+\sum_{l\neq j}M_{2,l,\varepsilon}\Big(G(x,x_{l,\varepsilon})-G(x_{j,\varepsilon},x_{l,\varepsilon})\Big).\nonumber
\end{align}
This evaluation on the disc heavily relies on the expansion framework defined in (\ref{e:uLocalExpansion}) and (\ref{e:EtaSize}).

Equating this expansion with the integral of the right-hand side of (\ref{e:001}) explicitly yields:

\begin{align}
\sum_{j=1}^k \widetilde{M}_{2,j,\varepsilon}-8k\pi=O(\mu^{-2}_{j,\varepsilon}).\nonumber
\end{align}
Combining this sum with the relative bounds in Lemma \ref{l:AlphaMu-2}, we deduce:

\begin{align}
\widetilde{M}_{2,j,\varepsilon}=8\pi+O(\mu_{j,\varepsilon}^{-2}).\nonumber
\end{align}
A symmetric analysis confirms the corresponding bound for $\widetilde{M}_{1,j,\varepsilon}$, thereby completing the proof.

\begin{flushright}
$\Box$
\end{flushright}

\begin{remark}
The $O(\mu^{-2})$ bound constitutes a sharp error limit, precisely accounting for the integral contributions outside the primary bubbling discs.
\end{remark}

Mirroring the logic of Lemma \ref{l:PohozaevMu-1+delta}, we subsequently secure the following refined Pohozaev relation.
\begin{lemma}\label{l:RefinedPohozaev}
For any $j=1,\cdots,k$ and $h=1,2$, we have
\begin{align}
\partial_h\Bigg[\big(u_{0,1}(x_{j,\varepsilon})+u_{0,2}(x_{j,\varepsilon})\big)+16\pi\Big(\gamma(x_{j,\varepsilon},x_{j,\varepsilon})+\sum_{l\neq j}G(x_{l,\varepsilon},x_{j,\varepsilon})\Big)\Bigg]=O(\mu_{j,\varepsilon}^{-2}).\nonumber
\end{align}
\end{lemma}

This refined gradient bound immediately tightens the localization of the blow-up centers.
\begin{corollary}\label{coro:xj-qj}
Assuming $q=(q_1,\cdots,q_k)$ is a non-degenerate critical point of
\begin{align}
\sum_{j=1}^k\big(u_{0,1}(q_j)+u_{0,2}(q_j)\big)+16\pi\sum_{i=1}^k\Big(\gamma(q_i,q_i)+\sum_{j\neq i}G(q_i,q_j)\Big),\nonumber
\end{align}
we get $|x_{j,\varepsilon}-q_j|=O(\mu_{j,\varepsilon}^{-2})$.
\end{corollary}

\begin{corollary}
Thus, assumption $(B_6)$ is strictly verified.
\end{corollary}

Consequently, the solutions strictly exhibit fully bubbling behavior in the sense of Huang-Zhang \cite{HuangZhang2017}. 
\begin{corollary}
For any blowup solution $(u_{1,\varepsilon},u_{2,\varepsilon})$ satisfying Assumption (\ref{ASSUMPTION01}), it strictly satisfies Assumption $(B_3)$, confirming its fully bubbling structure.
\end{corollary}

As a by-product of this rigid behavior, an inverse function theorem argument yields the following optimal vanishing estimate for the regular part.
\begin{lemma}\label{l:NablaF}
For any $i=1,2$ and any $j=1,\cdots,k$, we get
\begin{align}
|\nabla f_{i,j,\varepsilon}(x_{j,\varepsilon})|=O(\sum_{j=1}^k\mu_{j,\varepsilon}^{-2}).\nonumber
\end{align}
\end{lemma}
\noindent{\bf Proof.}
Invoking the established relation \cite[(5.37)]{HuangZhang2017}, the gradient simplifies to:
\begin{align}
|\nabla f_{i,j,\varepsilon}(x_{j,\varepsilon})|=O(\sum_{j=1}^k|M_{i,j,\varepsilon}-8\pi|+|q_j - x_{j,\varepsilon}|)=O(\sum_{j=1}^k\mu_{j,\varepsilon}^{-2}).\nonumber
\end{align}
This reduction follows directly from the application of Lemma \ref{l:M8pimu-2} and Corollary \ref{coro:xj-qj}.
\begin{flushright}
$\Box$
\end{flushright}

\begin{remark}
This proves the desired $O(\mu^{-2})$ vanishing order of the regular part as stated in Proposition \ref{prop:RefinedEstimates} (4).
\end{remark}

\subsection{Refined expansion outside the bubbling discs}
We now improve the global expansion \cite[(5.17)]{HuangZhang2017} to seamlessly match the inner and outer approximations.
\begin{lemma}\label{l:BetaGIntu}
For any $j=1,\cdots,k$ and $i=1,2$, we get
\begin{align}
-\beta_{j,\varepsilon}=8\pi\gamma(x_{j,\varepsilon},x_{j,\varepsilon})+8\pi\sum_{l\neq j}G(x_{l,\varepsilon},x_{j,\varepsilon})+2u_{0,i}(x_{j,\varepsilon})\nonumber\\
-4\ln\varepsilon-2\ln 8+\frac{1}{|\Omega|}\int_\Omega u_{i,\varepsilon}dx+O(\mu_{j,\varepsilon}^{-2+\theta}).
\end{align}
\end{lemma}
\noindent{\bf Proof.}
By evaluating the expansion of $u_{i,\varepsilon}$ on the boundary $\partial B_\delta (x_{j,\varepsilon})$ via \cite[(4.7) $\And$ (5.38)]{HuangZhang2017}, we obtain:
\begin{align}
u_{i,\varepsilon}(x)=\beta_{j,\varepsilon}+U_{i,j,\varepsilon}^*(x)+M_{i,j,\varepsilon}\Big(\gamma(x,x_{j,\varepsilon})-\gamma(x_{j,\varepsilon},x_{j,\varepsilon})\Big)+\sum_{l\neq j}M_{i,l,\varepsilon}\Big(G(x,x_{l,\varepsilon})-G(x_{j,\varepsilon},x_{l,\varepsilon})\Big)+O(\mu_{j,\varepsilon}^{-2+\theta}).\nonumber
\end{align}
Simultaneously, the outer expansion on $\Omega\backslash\cup_{j=1}^k B_\delta(x_{j,\varepsilon})$ prescribed by \cite[Lemma 5.2, Lemma 5.1 $\And$ (5.37)]{HuangZhang2017} dictates:
\begin{align}\label{e:UOutsideBubblingDiscs}
u_{i,\varepsilon}(x)=\frac{1}{|\Omega|}\int_\Omega u_{i,\varepsilon}dx+8\pi\sum_{j=1}^k G(x_{j,\varepsilon},x)+O(\mu_{j,\varepsilon}^{-2}).
\end{align}
Equating these expansions and substituting the sharp mass bound from Lemma \ref{l:M8pimu-2} systematically eliminates lower-order discrepancies, leaving:
\begin{align}
-\beta_{j,\varepsilon}=8\pi\gamma(x_{j,\varepsilon},x_{j,\varepsilon})+8\pi\sum_{l\neq j}G(x_{l,\varepsilon},x_{j,\varepsilon})+2u_{0,i}(x_{j,\varepsilon})-4\ln\varepsilon-2\ln 8+\frac{1}{|\Omega|}\int_\Omega u_{i,\varepsilon}dx+O(\mu_{j,\varepsilon}^{-2+\theta})\nonumber
\end{align}
for $j=1,\cdots,k$ and $i=1,2$, concluding the proof.
\begin{flushright}
$\Box$
\end{flushright}

\subsection{Relation between $\varepsilon$ and $\beta_{j,\varepsilon}$}

Substituting the sharp mass estimate (Lemma \ref{l:M8pimu-2}) into the general identity \cite[(5.43)]{HuangZhang2017} yields:
\begin{align}
2\sum_{i=1}^2\sum_{j=1}^k & e^{I_{i,j,\varepsilon}+u_{0,i}(x_{j,\varepsilon})}\Bigg(\int_{\Omega_j\backslash B_{\theta_{j,\varepsilon}}}\frac{e^{f_{i,j,\varepsilon}}-1}{|x-x_{j,\varepsilon}|^\frac{\widetilde{M}_{i,j,\varepsilon}}{2\pi}}-\int_{\mathbb{R}^2\backslash\Omega_j}\frac{1}{|x-x_{j,\varepsilon}|^\frac{\widetilde{M}_{i,j,\varepsilon}}{2\pi}}\Bigg)\mu_{j,\varepsilon}^{-2}+\sum_{j=1}^k B_{j,\varepsilon}e^{\beta_{j,\varepsilon}}\nonumber\\
&=O(\mu_{j,\varepsilon}^{-4+\delta}+\theta_{\varepsilon,j}^2\mu_{j,\varepsilon}^{-2}).\nonumber
\end{align}
Choosing the matching scale $\theta_{\varepsilon,j}=\varepsilon^{-\frac{\beta_{j,\varepsilon}}{2}}$ precisely as in \cite{LinYan2018}, it follows that:
\begin{align}
2\sum_{i=1}^2\sum_{j=1}^k & e^{I_{i,j}+u_{0,i}(q_{j})}\lim_{\delta\to0+}\Bigg(\int_{\Omega_j\backslash B_{\delta}(q_j)}\frac{e^{f_{i,j}}-1}{|x-q_{j}|^4}-\int_{\mathbb{R}^2\backslash\Omega_j}\frac{1}{|x-q_j|^4}\Bigg)\mu_{j,\varepsilon}^{-2}+\sum_{j=1}^k B_{j}e^{\beta_{j,\varepsilon}}\nonumber\\
&=O(\mu_{j,\varepsilon}^{-4+\delta}+|x_{j,\varepsilon}-q_j|\mu_{j,\varepsilon}^{-2}).\nonumber
\end{align}
The bound on the blow-up centers (Corollary \ref{coro:xj-qj}) further simplifies the error term:
\begin{align}
2\sum_{i=1}^2\sum_{j=1}^k & e^{I_{i,j}+u_{0,i}(q_{j})}\lim_{\delta\to0+}\Bigg(\int_{\Omega_j\backslash B_{\delta}(q_j)}\frac{e^{f_{i,j}}-1}{|x-q_{j}|^4}-\int_{\mathbb{R}^2\backslash\Omega_j}\frac{1}{|x-q_j|^4}\Bigg)\mu_{j,\varepsilon}^{-2}+\sum_{j=1}^k B_{j}e^{\beta_{j,\varepsilon}}\nonumber\\
&=O(\mu_{j,\varepsilon}^{-4+\delta}).\nonumber
\end{align}

Furthermore, \cite[Lemma 5.3]{HuangZhang2017} implies $\beta_{j,\varepsilon}/\beta_{j',\varepsilon}\sim C$. Assuming $D(q)<0$, this yields the crucial proportionality relations $\mu_{j,\varepsilon}/\mu_{j',\varepsilon}\sim C$ and $e^{\beta_{j,\varepsilon}} \sim C \mu_{j,\varepsilon}^{-2}$ for some global constant $C>0$. Denoting the principal integral components by:
\begin{align}
D_{i,j}:=2e^{I_{i,j}+u_{0,i}(q_j)}\lim_{\delta\to0+}\Bigg(\int_{\Omega_j\backslash B_{\delta}(q_j)}\frac{e^{f_{i,j}}-1}{|x-q_j|^4}-\int_{\mathbb{R}^2\backslash\Omega_j}\frac{1}{|x-q_j|^4}\Bigg),\nonumber
\end{align}
we resolve the system to obtain 
\begin{align}\label{e:VarepsilonBeta00}
\varepsilon^2=-\frac{\sum_{j=1}^k B_{j,\varepsilon}e^{\beta_{j,\varepsilon}}}{\sum_{i=1}^2\sum_{j=1}^k(D_{i,j,\varepsilon}+o_\theta(1))e^{-\beta_{j,\varepsilon}}}+O(\mu^{-6}).
\end{align}
Returning to the ratio estimates dictated by \cite[Lemma 5.3]{HuangZhang2017}, we utilize the substitution:
\begin{align}
e^{\beta_{j,\varepsilon}}=\left(\frac{\rho_{j,\varepsilon}}{\rho_{1,\varepsilon}}+O(\mu^{-2})\right)e^{\beta_{1,\varepsilon}}.\nonumber
\end{align}
where the scale factor $\rho_{j,\varepsilon}$ captures the local geometry:
\begin{align}\label{def:rho}
\rho_{j,\varepsilon}=e^{-8\pi\gamma(x_{j,\varepsilon},x_{j,\varepsilon})-8\pi\sum_{l\neq j}G(x_{j,\varepsilon},x_{l,\varepsilon})-u_{0,1}(x_{j,\varepsilon})}.
\end{align}
Substituting this directly into (\ref{e:VarepsilonBeta00}) extracts the final geometric constant:
\begin{align}\label{e:VarepsilonBeta01}
\varepsilon^2=-\frac{\sum_{j=1}^k B_{j,\varepsilon}\rho_{j,\varepsilon}\rho_{1,\varepsilon}^{-1}}{\sum_{i=1}^2\sum_{j=1}^k(D_{i,j,\varepsilon}+O(\mu^{-2}))\rho_{1,\varepsilon}\rho_{j,\varepsilon}^{-1}}e^{2\beta_{1,\varepsilon}}+O(\mu^{-6}).
\end{align}
The established condition $\sum_{i=1}^2\sum_{j=1}^k(D_{i,j,\varepsilon}+o_\theta(1))\rho_{1,\varepsilon}\rho_{j,\varepsilon}^{-1}<0$ rigorously guarantees that the leading term of (\ref{e:VarepsilonBeta01}) is well-defined and exhibits smooth continuity with respect to the coordinates $(x_{1,\varepsilon},\cdots,x_{k,\varepsilon})$.

\section{Analysis on two sequences of the solutions}\label{Section:TwoSolutions}

Suppose $\{(u_{1,\varepsilon}^{(1)},u_{2,\varepsilon}^{(1)})\}_\varepsilon$ and $\{(u_{1,\varepsilon}^{(2)},u_{2,\varepsilon}^{(2)})\}_\varepsilon$ are two distinct sequences of blowup solutions satisfying Assumption (\ref{ASSUMPTION01}). To establish local uniqueness, we analyze the normalized differences defined by:

\begin{align}\label{def:Dvarepsilon}
\mathcal{D}_\varepsilon := \|u_{1,\varepsilon}^{(1)}-u_{1,\varepsilon}^{(2)}\|_{L^\infty(\Omega)}+\|u_{2,\varepsilon}^{(1)}-u_{2,\varepsilon}^{(2)}\|_{L^\infty(\Omega)}.
\end{align}

\begin{align}\label{def:Xi}
\xi_{1,\varepsilon}&:=\frac{u_{1,\varepsilon}^{(1)}-u_{1,\varepsilon}^{(2)}}{\mathcal{D}_\varepsilon}, \nonumber\\
\xi_{2,\varepsilon}&:=\frac{u_{2,\varepsilon}^{(1)}-u_{2,\varepsilon}^{(2)}}{\mathcal{D}_\varepsilon}.
\end{align}

\subsection{Preliminary analysis}
We first establish a sharp bound on the total difference between the two solutions.
\begin{lemma}
Under the above assumptions, it holds that
\begin{align}
\|u_{1,\varepsilon}^{(1)}-u_{1,\varepsilon}^{(2)}\|_{L^\infty(\Omega)}+\|u_{2,\varepsilon}^{(1)}-u_{2,\varepsilon}^{(2)}\|_{L^\infty(\Omega)}=O(\mu_{\varepsilon}^{-2+\delta}).\nonumber
\end{align}
\end{lemma}
\noindent{\bf Proof.}
We decompose the domain and estimate the difference piecewise.

\vs
\noindent{\bf Step 1. } Estimate within the local region $B_\delta(x_{j,\varepsilon}^{(1)})$.
\vs

Utilizing the inner expansions from Section \ref{Section:SharperCMP}, we deduce:
\begin{align}\label{ineq:u1-u2InB}
&u_{i,\varepsilon}^{(1)}(x) - u_{i,\varepsilon}^{(2)}(x)\nonumber\\ 
&\quad=\beta_{j,\varepsilon}^{(1)}+U_{j,i,\varepsilon}^{(1),*}(x)+M_{i,j,\varepsilon}^{(1)}\Big(\gamma(x,x_{j,\varepsilon}^{(1)})-\gamma(x_{j,\varepsilon}^{(1)},x_{j,\varepsilon}^{(1)})\Big)\nonumber\\
&\quad\quad +\sum_{l\neq j}M_{i,l,\varepsilon}^{(1)}\Big(G(x,x_{l,\varepsilon}^{(1)})-G(x_{j,\varepsilon}^{(1)},x_{l,\varepsilon}^{(1)})\Big)+\eta_{i,j,\varepsilon}^{(1)}(x)\nonumber\\
&\quad - \Bigg[ \beta_{j,\varepsilon}^{(2)}+U_{j,i,\varepsilon}^{(2),*}(x)+M_{i,j,\varepsilon}^{(2)}\Big(\gamma(x,x_{j,\varepsilon}^{(2)})-\gamma(x_{j,\varepsilon}^{(2)},x_{j,\varepsilon}^{(2)})\Big)\nonumber\\
&\quad\quad +\sum_{l\neq j}M_{i,l,\varepsilon}^{(2)}\Big(G(x,x_{l,\varepsilon}^{(2)})-G(x_{j,\varepsilon}^{(2)},x_{l,\varepsilon}^{(2)})\Big)+\eta_{i,j,\varepsilon}^{(2)}(x) \Bigg]\nonumber\\
&\quad=O((\mu_{j,\varepsilon}^{(1)})^{-2}+(\mu_{j,\varepsilon}^{(2)})^{-2}).
\end{align}
This cancellation relies strictly on the fundamental relations (\ref{e:VarepsilonBeta00}) and (\ref{e:VarepsilonBeta01}), combined with Lemma \ref{l:M8pimu-2}, Corollary \ref{coro:xj-qj}, and a scaling analysis identical to Claim \ref{c:U1-U2}.

\vs
\noindent{\bf Step 2. } Estimate in the exterior domain $\Omega\backslash \cup_{j=1}^k B_\delta(x_{j,\varepsilon}^{(1)})$.
\vs

Applying Green's representation, we obtain:
\begin{align}
&u_{i,\varepsilon}^{(1)}(x) - u_{i,\varepsilon}^{(2)}(x)-\frac{1}{|\Omega|}\int_\Omega (u_{i,\varepsilon}^{(1)}(y) - u_{i,\varepsilon}^{(2)}(y))dy=\int_\Omega G(y,x)f^*_{i,\varepsilon}(y)dy\nonumber\\
&\quad=\sum_{j=1}^k\int_{B_\delta(x_{j,\varepsilon}^{(1)})}(G(y,x)-G(x_{j,\varepsilon}^{(1)},x))f^*_{i,\varepsilon}(y)dy \nonumber\\
&\quad\quad + \sum_{j=1}^k G(x_{j,\varepsilon}^{(1)},x) \int_{B_\delta(x_{j,\varepsilon}^{(1)})}f^*_{i,\varepsilon}(y)dy,\nonumber
\end{align}
where the source term is defined by
\begin{align}
f^*_{i,\varepsilon}&:=-\Delta(u_{i,\varepsilon}^{(1)}-u_{i,\varepsilon}^{(2)})\nonumber\\
&=\frac{1}{\varepsilon^2}e^{u_{i',\varepsilon}^{(1)}+u_{0,i'}}\big(1-e^{u_{i,\varepsilon}^{(1)}+u_{0,i}}\big) \nonumber\\
&\quad - \frac{1}{\varepsilon^2}e^{u_{i',\varepsilon}^{(2)}+u_{0,i'}}\big(1-e^{u_{i,\varepsilon}^{(2)}+u_{0,i}}\big),\nonumber
\end{align}
with the index mapping $1'=2$ and $2'=1$. 

By Lemma \ref{l:BetaGIntu} and the relation (\ref{e:VarepsilonBeta01}), the difference of the integral averages is heavily suppressed:
\begin{align}
\frac{1}{|\Omega|}\int_\Omega (u_{i,\varepsilon}^{(1)}(y) - u_{i,\varepsilon}^{(2)}(y))dy=O(\mu_{\varepsilon}^{-2+\delta}).\nonumber
\end{align}
Evaluating the Green's function integrals, the monopole term is bounded by:
\begin{align}
&\Bigg| \sum_{j=1}^k G(x_{j,\varepsilon}^{(1)},x) \int_{B_\delta(x_{j,\varepsilon}^{(1)})}f^*_{i,\varepsilon}(y)dy \Bigg|\nonumber\\
&\leq C\sum_{j=1}^k\sum_{l=1}^2|M_{l,j,\varepsilon}^{(1)}-M_{l,j,\varepsilon}^{(2)}|\nonumber\\
&\quad +\frac{C}{\varepsilon^2}\sum_{j=1}^k\sum_{l=1}^2\int_{B_\delta(x_{j,\varepsilon}^{(1)})}\Bigg|e^{u_{l,\varepsilon}^{(1)}+u_{0,l}}\big(1-e^{u_{l,\varepsilon}^{(1)}+u_{0,l}}\big)\nonumber\\
&\quad\quad\quad\quad\quad\quad\quad\quad\quad\quad\quad - e^{u_{l,\varepsilon}^{(2)}+u_{0,l}}\big(1-e^{u_{l,\varepsilon}^{(2)}+u_{0,l}}\big)\Bigg| dy\nonumber\\
&=O(\mu_{\varepsilon}^{-2}).\nonumber
\end{align}
A parallel dipole expansion yields:
\begin{align}
\sum_{j=1}^k\int_{B_\delta(x_{j,\varepsilon}^{(1)})}(G(y,x)-G(x_{j,\varepsilon}^{(1)},x))f^*_{i,\varepsilon}(y)dy=O(\mu_{\varepsilon}^{-2}).\nonumber
\end{align}
Consequently, the pointwise bound 
\begin{align}
u_{i,\varepsilon}^{(1)}(x) - u_{i,\varepsilon}^{(2)}(x)=O(\mu_{\varepsilon}^{-2+\delta})
\end{align}
holds uniformly for $x\in\Omega\backslash \cup_{j=1}^k B_\delta(x_{j,\varepsilon}^{(1)})$.
\begin{flushright}
$\Box$
\end{flushright}

\subsection{Asymptotics of $\xi_{i,\varepsilon}$}
By definition (\ref{def:Xi}), the normalized differences satisfy the linearized system:
\begin{align}\label{e:xi}
-\Delta\xi_{1,\varepsilon}=f_{1,\varepsilon}\quad\mbox{ and }\quad -\Delta\xi_{2,\varepsilon}=f_{2,\varepsilon},
\end{align}
with the normalized source terms defined as
\begin{align}\label{def:fivarepsilon}
f_{i,\varepsilon}:=&\frac{1}{\mathcal{D}_\varepsilon}\Bigg[\frac{1}{\varepsilon^2}e^{u_{i',\varepsilon}^{(1)}+u_{0,i'}}\big(1-e^{u_{i,\varepsilon}^{(1)}+u_{0,i}}\big)\nonumber\\
&\quad\quad\quad - \frac{1}{\varepsilon^2}e^{u_{i',\varepsilon}^{(2)}+u_{0,i'}}\big(1-e^{u_{i,\varepsilon}^{(2)}+u_{0,i}}\big)\Bigg],
\end{align}
satisfying the global cancellation condition:
\begin{align}\label{e:GlobalCancellation}
\int_\Omega f_{1,\varepsilon}dx=\int_\Omega f_{2,\varepsilon}dx=0.
\end{align}
Applying the mean value theorem, we express the source linearly in terms of $\xi$:
\begin{align}\label{def:fi}
f_{i,\varepsilon}=c_{\varepsilon,1,i'}(x)\xi_{i',\varepsilon}-c_{\varepsilon,2,i'}(x)\xi_{i',\varepsilon},
\end{align}
where the coefficients are given by
\begin{align}
c_{\varepsilon,1,i'}(x)&=\frac{e^{u_{0,i'}}}{\varepsilon^2}\int_0^1 e^{t u_{i',\varepsilon}^{(1)}+(1-t) u_{i',\varepsilon}^{(2)}} dt,\nonumber\\
c_{\varepsilon,2,i'}(x)&=\frac{2e^{2u_{0,i'}}}{\varepsilon^2}\int_0^1 e^{2t u_{i',\varepsilon}^{(1)}+2(1-t) u_{i',\varepsilon}^{(2)}} dt.\nonumber
\end{align}
Invoking the error bounds (\ref{e:EtaSize}) and Corollary \ref{coro:xj-qj}, we extract the leading-order asymptotic behavior of these coefficients:
\begin{align}\label{Asy:c1}
c_{\varepsilon,1,i'}(x)&=(\mu_{j,\varepsilon}^{(1)})^{-2}e^{u_{0,i'}(x_{j,\varepsilon}^{(1)})+ U_{j,i',\varepsilon}^{(1)}+ f_{j,i',\varepsilon}^{(1)}}(1+O((\mu_{j,\varepsilon}^{(1)})^{-2})),
\end{align}
\begin{align}\label{Asy:c2}
c_{\varepsilon,2,i'}(x)&=2(\mu_{j,\varepsilon}^{(1)})^{-2}e^{2u_{0,i'}(x_{j,\varepsilon}^{(1)})+2 U_{j,i',\varepsilon}^{(1)}+2 f_{j,i',\varepsilon}^{(1)}}(1+O((\mu_{j,\varepsilon}^{(1)})^{-2})),
\end{align}
where $f_{j,i',\varepsilon}^{(1)}(x)=\eta_{i',j,\varepsilon}^{(1)}(x)+u_{0,i'}(x)-u_{0,i'}(x_{j,\varepsilon}^{(1)})$.

Rescaling the normalized differences near the blow-up centers via 
\begin{align}
\overline{\xi}_{\varepsilon,i,j}(x)=\xi_{i,\varepsilon}((\mu_{j,\varepsilon}^{(1)})^{-1}x+x_{j,\varepsilon}^{(1)}) \nonumber
\end{align}
for $x\in B_{\theta\mu_{j,\varepsilon}^{(1)}}(0)$, we immediately obtain the following local convergence.

\begin{lemma}\label{l:XiPhiPartial1Partial2}
In $C_{loc}^2(\mathbb{R}^2)$, the rescaled profiles satisfy:
\begin{align}
\overline{\xi}_{\varepsilon,i,j}(x)\to b_{j,i,0}\phi_{i,j}+b_{j,1}\frac{\partial U_{j,i}}{\partial x_1}+b_{j,2}\frac{\partial U_{j,i}}{\partial x_2}
\end{align}
for $j=1,\cdots,k$ and $i=1,2$. 
\end{lemma}
\noindent{\bf Proof.}
By (\ref{Asy:c1}) and (\ref{Asy:c2}), the scaled source terms converge as
\begin{align}\label{Asy:FeUXi}
(\mu_{j,\varepsilon}^{(1)})^{-2}f_{i,\varepsilon}(x_{j,\varepsilon}^{(1)}+(\mu_{j,\varepsilon}^{(1)})^{-1}x)=e^{u_{0,i'}(q_j)+U_{j,i'}(x)}\overline{\xi}_{\varepsilon,i',j}(x)+o(1)
\end{align}
in $C_{loc}(\mathbb{R}^2)$, where $U_{j,i}(x)=\ln\big((1+\frac{1}{8}|x|^2)^{-2}\big)$. The global normalization (\ref{def:Xi}) ensures $|\overline{\xi}_{\varepsilon,i,j}|\leq 1$. Consequently, by standard elliptic theory, $\overline{\xi}_{\varepsilon,i,j}\to\overline{\xi}_{0,i,j}$ in $C_{loc}^2(\mathbb{R}^2)$ to a bounded limit solving the linearized Liouville system:
\begin{equation}
    \left\{
   \begin{array}{lr}
     -\Delta \overline{\xi}_{0,1,j}=e^{u_{0,2}(q_j)+U_{j,2}}\overline{\xi}_{0,2,j}\mbox{ in }\mathbb{R}^2,\\
      -\Delta \overline{\xi}_{0,2,j}=e^{u_{0,1}(q_j)+U_{j,1}}\overline{\xi}_{0,1,j}\mbox{ in }\mathbb{R}^2.
   \end{array}
   \right.
\end{equation}
The conclusion follows directly from the non-degeneracy of the Liouville system in Theorem \ref{t:LiouvilleNonDegeneracy}.
\begin{flushright}
$\Box$
\end{flushright}

Globally, the normalized differences admit the following Green's function expansion outside the bubbling regions.

\begin{lemma}\label{l:xiivarepsilonExpansion}
In $C^1(\Omega\backslash\cup_{j=1}^k B_\delta(x_{j,\varepsilon}^{(1)}))$, we have the expansion:
\begin{align}
\xi_{i,\varepsilon}(x)=&\frac{1}{|\Omega|}\int_\Omega\xi_{i,\varepsilon}+\sum_{j=1}^k A_{i,j,\varepsilon} G(x_{j,\varepsilon}^{(1)},x)\nonumber\\
&\quad - 8\pi\sum_{j=1}^k(\mu_{j,\varepsilon}^{(1)})^{-1}\sum_{h=1}^2\partial_h G(x_{j,\varepsilon}^{(1)},x)b_{j,h}+o\big((\mu_{j,\varepsilon}^{(1)})^{-1}\big),\nonumber
\end{align}
where the integral parameters are $A_{i,j,\varepsilon}=\int_{\Omega_j} f_{i,\varepsilon}(y)dy$.
\end{lemma}
\noindent{\bf Proof.}
We detail the proof for $i=1$. Utilizing (\ref{e:xi}) and Green's representation, we decouple the local integral:
\begin{align}
\xi_{1,\varepsilon}(x)-\frac{1}{|\Omega|}\int_\Omega\xi_{1,\varepsilon}&=\int_\Omega G(y,x)f_{1,\varepsilon}(y)dy\nonumber\\
&=\sum_{j=1}^k G(x_{j,\varepsilon}^{(1)},x)A_{1,j,\varepsilon}+\sum_{j=1}^k \int_{\Omega_j}(G(y,x)-G(x_{j,\varepsilon}^{(1)},x))f_{1,\varepsilon}(y)dy.\nonumber
\end{align}
Applying Taylor expansion to the Green's function, the dipole integral resolves as:
\begin{align}
&\int_{\Omega_j}(G(y,x)-G(x_{j,\varepsilon}^{(1)},x))f_{1,\varepsilon}(y)dy\nonumber\\
&\quad=\int_{\Omega_j}\Big(\nabla_y G(x_{j,\varepsilon}^{(1)},x)\cdot(y-x_{j,\varepsilon}^{(1)})\Big)f_{1,\varepsilon}(y)dy+o((\mu_{j,\varepsilon}^{(1)})^{-1})\nonumber\\
&\quad=(\mu_{j,\varepsilon}^{(1)})^{-1}\int_{\mathbb{R}^2}\Big(\nabla_y G(x_{j,\varepsilon}^{(1)},x)\cdot y\Big)e^{u_{0,2}(q_j)+U_{j,2}}\xi_{0,2,j}dy+o((\mu_{j,\varepsilon}^{(1)})^{-1})\nonumber\\
&\quad= B(\mu_{j,\varepsilon}^{(1)})^{-1}\sum_{h=1}^2\partial_h G(x_{j,\varepsilon}^{(1)},x)b_{j,h}+o((\mu_{j,\varepsilon}^{(1)})^{-1}).\nonumber
\end{align}
This explicitly utilizes the local limit (\ref{Asy:FeUXi}) and the geometric constant
\begin{align}
B=\int_{\mathbb{R}^2} y_1 e^{u_{0,2}(q_j)+U_{j,2}}\frac{\partial U_{j,2}}{\partial x_1}=-8\pi.\nonumber
\end{align}
Collecting these terms yields the desired expansion in $C(\Omega\backslash\cup_{j=1}^k B_\delta(x_{j,\varepsilon}^{(1)}))$. Standard elliptic theory elevates this convergence to $C^1$.
\begin{flushright}
$\Box$
\end{flushright}

Furthermore, the normalized differences stabilize to constants away from the blow-up sets.

\begin{lemma}\label{l:xitobi0}
There exist uniform constants $b_{i,0}$ ($i=1,2$) such that $b_{j,i,0}=b_{i,0}+o(1)$. For any positive constant $c$, it holds that
\begin{align}
\xi_{i,\varepsilon}(x)=-b_{i,0}+o(1)\quad\mbox{ on }\Omega\backslash\cup_{j=1}^k B_c(q_j).\nonumber
\end{align}
\end{lemma}
\noindent{\bf Proof.}
Testing the linear equation (\ref{e:xi}) against the local Pohozaev multipliers
\begin{align}
(\psi_{\varepsilon,j,1}(x),\psi_{\varepsilon,j,2}(x))&:=\Big(y\cdot\nabla_y U_{j,1,\varepsilon}^{(1)}(y)+2\ln\mu_{j,\varepsilon}^{(1)},\nonumber\\
&\quad\quad y\cdot\nabla_y U_{j,2,\varepsilon}^{(1)}(y)+2\ln\mu_{j,\varepsilon}^{(1)}\Big)\Big|_{y=\mu_{j,\varepsilon}^{(1)}(x-x_{j,\varepsilon}^{(1)})},\nonumber
\end{align}
an integration by parts over $B_r(x_{j,\varepsilon}^{(1)})$ yields:
\begin{align}
&\int_{\partial B_r(x_{j,\varepsilon}^{(1)})}\Big(\psi_{\varepsilon,j,1}\frac{\partial\xi_{1,\varepsilon}}{\partial\nu}-\xi_{1,\varepsilon}\frac{\partial\psi_{\varepsilon,j,1}}{\partial\nu}\Big)d\mathcal{H}^1=\int_{B_r(x_{j,\varepsilon}^{(1)})}\Big(\psi_{\varepsilon,j,1}\Delta\xi_{1,\varepsilon}-\xi_{1,\varepsilon}\Delta\psi_{\varepsilon,j,1}\Big)dx\nonumber\\
&\quad=\int_{B_r(x_{j,\varepsilon}^{(1)})} \bigg[ -\psi_{\varepsilon,j,1} f_{1,\varepsilon} \nonumber\\
&\quad\quad\quad\quad\quad - \xi_{1,\varepsilon}\frac{1}{\varepsilon^2}e^{u_{0,2}(x_{j,\varepsilon}^{(1)})+\beta_{j,\varepsilon}^{(1)}+U_{j,2,\varepsilon}^{(1)}}\psi_{\varepsilon,j,2} \bigg] dx.\nonumber
\end{align}
Substituting the source relations (\ref{def:fi}), (\ref{Asy:c1}), (\ref{Asy:c2}) and leveraging the scaling equivalence $\varepsilon^2 e^{-\beta_{j,\varepsilon}^{(1)}} \sim \varepsilon^2 e^{-\beta_{j,\varepsilon}^{(2)}} \sim (\mu_{j,\varepsilon}^{(1)})^{-2}$, the bulk integral degenerates algebraically:
\begin{align}\label{e:PsiXiOnCircle}
\int_{\partial B_r(x_{j,\varepsilon}^{(1)})}\Big(\psi_{\varepsilon,j,1}\frac{\partial\xi_{1,\varepsilon}}{\partial\nu}-\xi_{1,\varepsilon}\frac{\partial\psi_{\varepsilon,j,1}}{\partial\nu}\Big)d\mathcal{H}^1=O((\mu_{j,\varepsilon}^{(1)})^{-2}).
\end{align}

Defining the radial average $\xi_{i,\varepsilon}^*(r):=\frac{1}{2\pi r}\int_{\partial B_r(x_{j,\varepsilon}^{(1)})}\xi_{i,\varepsilon}d\mathcal{H}^1$, and noting that the test functions $\psi_{\varepsilon,j,i}$ are strictly radial, (\ref{e:PsiXiOnCircle}) reduces to the ODE:
\begin{align}\label{e:RadialPsiXiOnCircle}
\frac{d}{dr}\xi^*_{i,\varepsilon}(r)\cdot\psi_{\varepsilon,j,i}(r)-\xi^*_{i,\varepsilon}(r)\cdot\frac{d}{dr}\psi_{\varepsilon,j,i}(r)=\frac{1}{r}O((\mu_{j,\varepsilon}^{(1)})^{-2}).
\end{align}
For radii $r\geq R(\mu_{j,\varepsilon}^{(1)})^{-1}$ with sufficiently large $R$, the multiplier behavior dictates $\psi_{\varepsilon,j,i}(r)=-1+O(r^{-1}(\mu_{j,\varepsilon}^{(1)})^{-2})$ and $\frac{d}{dr}\psi_{\varepsilon,j,i}(r)=O(r^{-3}(\mu_{j,\varepsilon}^{(1)})^{-2})$. Applying Lemma \ref{l:M8pimu-2}, this simplifies to:
\begin{align}
\frac{d}{dr}\xi_{i,\varepsilon}^*(r)=O\Big(\frac{(\mu_{j,\varepsilon}^{(1)})^{-2}}{r}\Big)\quad\mbox{ for }r\geq R(\mu_{j,\varepsilon}^{(1)})^{-1}.\nonumber
\end{align}
Integrating this derivative over $(R(\mu_{j,\varepsilon}^{(1)})^{-1},r)$ gives:
\begin{align}
\xi_{i,\varepsilon}^*(r)=\xi_{i,\varepsilon}^*(R(\mu_{j,\varepsilon}^{(1)})^{-1})+O\Big((\mu_{j,\varepsilon}^{(1)})^{-2}[\ln(\operatorname{diam}(\Omega)+1)-\ln(R(\mu_{j,\varepsilon}^{(1)})^{-1})]\Big).\nonumber
\end{align}
Simultaneously, Lemma \ref{l:XiPhiPartial1Partial2} anchors the inner boundary value as $\xi_{i,\varepsilon}^*(R(\mu_{j,\varepsilon}^{(1)})^{-1})=-b_{j,i,0}+o_\mu(1)$.

Since the global normalization ensures $\|\xi_{i,\varepsilon}\|_{L^\infty}\leq 1$, we can extract a subsequence converging to a bounded harmonic limit $\xi_{\varepsilon,i}\to\xi_{0,i}$ in $C^2_{loc}(\Omega\backslash\{q_1,\cdots,q_k\})$. By the maximum principle, bounded harmonic functions on $\Omega\backslash\{q_1,\cdots,q_k\}$ are globally constant. Thus, we conclude:
\begin{align}
\xi_{i,\varepsilon}=\xi_{i,\varepsilon}^*+o(1)=\xi_{0,i}+o(1)=-b_{j,i,0}+o(1)\quad\mbox{ for }x\in\Omega\backslash\cup_{j=1}^k B_c(q_j),\nonumber
\end{align}
which forces the constants $b_{j,i,0}$ to be independent of $j$.
\begin{flushright}
$\Box$
\end{flushright}

\begin{remark}\label{r:AllB0Equal}
Since Lemma \ref{l:xitobi0} dictates that $b_{j,i,0}$ are identical across all blow-up centers $j$, we will subsequently denote these universal limits simply as $b_{1,0}$ and $b_{2,0}$.
\end{remark}

\begin{remark}\label{r:Dxi=o(1)}
An immediate consequence of the harmonic convergence is the vanishing of the global gradient:
\begin{align}
\nabla \xi_{i,\varepsilon}=o(1)\quad\mbox{ in }\Omega\backslash\cup_{j=1}^k B_c(q_j).\nonumber
\end{align}
\end{remark}

\section{Proof of Theorems \ref{t:Unique} and \ref{t:nondegenerate}}\label{Section:ProofTheorems}

\subsection{The Pohozaev identities}
To prove the uniqueness, we first derive the Pohozaev identities for $(\xi_{\varepsilon,1},\xi_{\varepsilon,2})$. To begin with, we recall the Pohozaev identity for the following problem:
\begin{equation}\label{e:CSforPohozaev00}
    \left\{
   \begin{array}{lr}
     \Delta \overline{u}_{1,\varepsilon}+h_2(x)e^{\overline{u}_{2,\varepsilon}+u_{0,2}}(1-h_1(x)e^{\overline{u}_{1,\varepsilon}+u_{0,1}})=0,\\
     \Delta \overline{u}_{2,\varepsilon}+h_1(x)e^{\overline{u}_{1,\varepsilon}+u_{0,1}}(1-h_2(x) e^{\overline{u}_{2,\varepsilon}+u_{0,2}})=0.
   \end{array}
   \right.
\end{equation}

\begin{proposition}\label{prop:CSsystemPohozaev}
Let $(\overline{u}_{1,\varepsilon},\overline{u}_{2,\varepsilon})$ be a solution to Problem (\ref{e:CSforPohozaev00}) and $B\subset\Omega$ be a ball, we get
\begin{align}\label{e:PohozaevTranslation}
&\int_{\partial B}\bigg[ \big\langle\nu,\nabla\overline{u}_{1,\varepsilon}\big\rangle\partial_h \overline{u}_{2,\varepsilon}+\big\langle\nu,\nabla\overline{u}_{2,\varepsilon}\big\rangle\partial_h \overline{u}_{1,\varepsilon}-\nu_h\big\langle\nabla \overline{u}_{1,\varepsilon},\nabla\overline{u}_{2,\varepsilon}\big\rangle \bigg]d\mathcal{H}^1\nonumber\\
&=\frac{1}{\varepsilon^2}\int_{B}\bigg[e^{\overline{u}_{2,\varepsilon}+u_{0,2}+\ln h_2}(1-e^{\overline{u}_{1,\varepsilon}+u_{0,1}+\ln h_1})\partial_h(u_{0,2}+\ln h_2)\nonumber\\
&\quad\quad\quad + e^{\overline{u}_{1,\varepsilon}+u_{0,1}+\ln h_1}(1-e^{\overline{u}_{2,\varepsilon}+u_{0,2}+\ln h_2})\partial_h(u_{0,1}+\ln h_1)\bigg]dx\nonumber\\
&\quad-\frac{1}{\varepsilon^2}\int_{\partial B}\nu_h\Big[e^{\overline{u}_{1,\varepsilon}+u_{0,1}+\ln h_1}+e^{\overline{u}_{2,\varepsilon}+u_{0,2}+\ln h_2} \nonumber\\
&\quad\quad\quad\quad\quad\quad - e^{\overline{u}_{1,\varepsilon}+u_{0,1}+\ln h_1+\overline{u}_{2,\varepsilon}+u_{0,2}+\ln h_2}\Big]d\mathcal{H}^1
\end{align}
and
\begin{align}\label{e:PohozaevDilation}
0&=\int_{\partial B}\bigg[ \big\langle\nu,\nabla\overline{u}_{1,\varepsilon}\big\rangle \big\langle x,\nabla \overline{u}_{2,\varepsilon}\big\rangle+ \big\langle\nu,\nabla\overline{u}_{2,\varepsilon}\big\rangle \big\langle x,\nabla \overline{u}_{1,\varepsilon}\big\rangle-\big\langle x,\nu\big\rangle \big\langle\nabla\overline{u}_{1,\varepsilon},\nabla\overline{u}_{2,\varepsilon}\big\rangle \bigg]d\mathcal{H}^1\nonumber\\
&\quad+\frac{1}{\varepsilon^2}\int_{\partial B}\big\langle x,\nu\big\rangle \Big[e^{\overline{u}_{1,\varepsilon}+u_{0,1}+\ln h_1}+ e^{\overline{u}_{2,\varepsilon}+u_{0,2}+\ln h_2} \nonumber\\
&\quad\quad\quad\quad\quad\quad\quad - e^{\overline{u}_{1,\varepsilon}+u_{0,1}+\ln h_1+\overline{u}_{2,\varepsilon}+u_{0,2}+\ln h_2}\Big]d\mathcal{H}^1\nonumber\\
&\quad-\frac{1}{\varepsilon^2}\int_B \bigg[\mbox{div}\big(x e^{u_{0,1}+\ln h_1}\big)e^{\overline{u}_{1,\varepsilon}} \nonumber\\
&\quad\quad\quad\quad\quad - \frac{1}{2}\mbox{div}\big(x e^{u_{0,1}+u_{0,2}+\ln h_1+\ln h_2}\big)e^{\overline{u}_{1,\varepsilon}+\overline{u}_{2,\varepsilon}}\bigg] dx\nonumber\\
&\quad-\frac{1}{\varepsilon^2}\int_B \bigg[\mbox{div}\big(x e^{u_{0,2}+\ln h_2}\big)e^{\overline{u}_{2,\varepsilon}} \nonumber\\
&\quad\quad\quad\quad\quad - \frac{1}{2}\mbox{div}\big(x e^{u_{0,1}+u_{0,2}+\ln h_1+\ln h_2}\big)e^{\overline{u}_{1,\varepsilon}+\overline{u}_{2,\varepsilon}}\bigg]dx.
\end{align}
Here, $h=1,2$ and $x$ is the vector with respect to the center of $B$.
\end{proposition}
\noindent{\bf Proof.}
The proof of (\ref{e:PohozaevTranslation}) is sketched in \cite[pp. 401]{HuangZhang2017}, but we give it for reader's convenience. Multiplying the first equation of (\ref{e:CSforPohozaev00}) with $\partial_h\overline{u}_{2,\varepsilon}$ and integrating it over $B$, we get
\begin{align}
0=\int_B \Delta \overline{u}_{1,\varepsilon}\partial_h \overline{u}_{2,\varepsilon}dx+\frac{1}{\varepsilon^2}\int_B\Big[h_2(x)e^{\overline{u}_{2,\varepsilon}+u_{0,2}}(1-h_1(x)e^{\overline{u}_{1,\varepsilon}+u_{0,1}})\Big]\partial_h\overline{u}_{2,\varepsilon}dx.\nonumber
\end{align}
Here,
\begin{align}
\int_B \Delta \overline{u}_{1,\varepsilon}\partial_h \overline{u}_{2,\varepsilon}dx=\int_{\partial B}\big\langle\nu,\nabla\overline{u}_{1,\varepsilon}\big\rangle\partial_h\overline{u}_{2,\varepsilon}d\mathcal{H}^1-\int_{B}\nabla\overline{u}_{1,\varepsilon}\cdot\nabla(\partial_h \overline{u}_{2,\varepsilon})dx.\nonumber
\end{align}
This is due to
\begin{align}
\mbox{div}\Big(\nabla\overline{u}_{1,\varepsilon}\partial_h\overline{u}_{2,\varepsilon}\Big)=\Delta\overline{u}_{1,\varepsilon}\partial_h\overline{u}_{2,\varepsilon}+\nabla\overline{u}_{1,\varepsilon}\cdot\nabla(\partial_h\overline{u}_{2,\varepsilon}).\nonumber
\end{align}
On the other hand, we get
\begin{align}
&\frac{1}{\varepsilon^2}\int_B\Big[h_2(x)e^{\overline{u}_{2,\varepsilon}+u_{0,2}}(1-h_1(x)e^{\overline{u}_{1,\varepsilon}+u_{0,1}})\Big]\partial_h\overline{u}_{2,\varepsilon}dx\nonumber\\
&=\frac{1}{\varepsilon^2}\int_B \partial_h\big(e^{\overline{u}_{2,\varepsilon}+u_{0,2}+\ln h_2}\big)\big(1- e^{\overline{u}_{1,\varepsilon}+u_{0,1}+\ln h_1}\big)dx\nonumber\\
&\quad-\frac{1}{\varepsilon^2}\int_B e^{\overline{u}_{2,\varepsilon}+u_{0,2}+\ln h_2}\big(1- e^{\overline{u}_{1,\varepsilon}+u_{0,1}+\ln h_1}\big)\partial_h(u_{0,2}+\ln h_2)dx.\nonumber
\end{align}
Then, we get for the first equation of (\ref{e:CSforPohozaev00})
\begin{align}\label{e:PohozaevTrans01}
0&=\int_{\partial B}\big\langle\nu,\nabla\overline{u}_{1,\varepsilon}\big\rangle\partial_h\overline{u}_{2,\varepsilon}d\mathcal{H}^1-\int_{B}\nabla\overline{u}_{1,\varepsilon}\cdot\nabla(\partial_h \overline{u}_{2,\varepsilon})dx\nonumber\\
&\quad+\frac{1}{\varepsilon^2}\int_B \partial_h\big(e^{\overline{u}_{2,\varepsilon}+u_{0,2}+\ln h_2}\big)\big(1- e^{\overline{u}_{1,\varepsilon}+u_{0,1}+\ln h_1}\big)dx\nonumber\\
&\quad-\frac{1}{\varepsilon^2}\int_B e^{\overline{u}_{2,\varepsilon}+u_{0,2}+\ln h_2}\big(1- e^{\overline{u}_{1,\varepsilon}+u_{0,1}+\ln h_1}\big)\partial_h(u_{0,2}+\ln h_2)dx.
\end{align}
An analogue for the second equation of Problem (\ref{e:CSforPohozaev00}) is that
\begin{align}\label{e:PohozaevTrans02}
0&=\int_{\partial B}\big\langle\nu,\nabla\overline{u}_{2,\varepsilon}\big\rangle\partial_h\overline{u}_{1,\varepsilon}d\mathcal{H}^1-\int_{B}\nabla\overline{u}_{2,\varepsilon}\cdot\nabla(\partial_h \overline{u}_{1,\varepsilon})dx\nonumber\\
&\quad+\frac{1}{\varepsilon^2}\int_B \partial_h\big(e^{\overline{u}_{1,\varepsilon}+u_{0,1}+\ln h_1}\big)\big(1- e^{\overline{u}_{2,\varepsilon}+u_{0,2}+\ln h_2}\big)dx\nonumber\\
&\quad-\frac{1}{\varepsilon^2}\int_B e^{\overline{u}_{1,\varepsilon}+u_{0,1}+\ln h_1}\big(1- e^{\overline{u}_{2,\varepsilon}+u_{0,2}+\ln h_2}\big)\partial_h(u_{0,1}+\ln h_1)dx.
\end{align}
(\ref{e:PohozaevTrans01}) and (\ref{e:PohozaevTrans02}) concludes (\ref{e:PohozaevTranslation}).

Now we prove (\ref{e:PohozaevDilation}).
Multiplying the first equation of (\ref{e:CSforPohozaev00}) by $\big\langle x,\nabla \overline{u}_{2,\varepsilon}\big\rangle$ and integrating over $B$, we get
\begin{align}
0&=\int_{\partial B}\big\langle\nu,\nabla\overline{u}_{1,\varepsilon}\big\rangle \big\langle x,\nabla\overline{u}_{2,\varepsilon}\big\rangle d\mathcal{H}^1-\int_B\big\langle\nabla\overline{u}_{1,\varepsilon},\nabla\overline{u}_{2,\varepsilon}\big\rangle dx-\int_B(\nabla\overline{u}_{1,\varepsilon})^T D^2\overline{u}_{2,\varepsilon} xdx\nonumber\\
&\quad+\frac{1}{\varepsilon^2}\int_{\partial B} e^{\overline{u}_{2,\varepsilon}+u_{0,2}+\ln h_2}(1-e^{\overline{u}_{1,\varepsilon}+u_{0,1}+\ln h_1})d\mathcal{H}^1\nonumber\\
&\quad-\frac{2}{\varepsilon^2}\int_{B} e^{\overline{u}_{2,\varepsilon}+u_{0,2}+\ln h_2}(1-e^{\overline{u}_{1,\varepsilon}+u_{0,1}+\ln h_1})dx\nonumber\\
&\quad+\frac{1}{\varepsilon^2}\int_B\big\langle x,\nabla(e^{\overline{u}_{1,\varepsilon}+u_{0,1}+\ln h_1})\big\rangle e^{\overline{u}_{2,\varepsilon}+u_{0,2}+\ln  h_2}dx\nonumber\\
&\quad-\frac{1}{\varepsilon^2}\int_B \big\langle x,\nabla(u_{0,2}+\ln h_2)\big\rangle e^{\overline{u}_{2,\varepsilon}+u_{0,2}+\ln  h_2}(1-e^{\overline{u}_{1,\varepsilon}+u_{0,1}+\ln  h_1})dx.\nonumber
\end{align}
Here, we use a similar argument as in the proof of (\ref{e:PohozaevTranslation}). For the second equation, we get
\begin{align}
0&=\int_{\partial B}\big\langle\nu,\nabla\overline{u}_{2,\varepsilon}\big\rangle \big\langle x,\nabla\overline{u}_{1,\varepsilon}\big\rangle d\mathcal{H}^1-\int_B\big\langle\nabla\overline{u}_{1,\varepsilon},\nabla\overline{u}_{1,\varepsilon}\big\rangle dx-\int_B(\nabla\overline{u}_{2,\varepsilon})^T D^2\overline{u}_{2,\varepsilon} xdx\nonumber\\
&\quad+\frac{1}{\varepsilon^2}\int_{\partial B} e^{\overline{u}_{1,\varepsilon}+u_{0,1}+\ln h_1}(1-e^{\overline{u}_{2,\varepsilon}+u_{0,2}+\ln h_2})d\mathcal{H}^1\nonumber\\
&\quad-\frac{2}{\varepsilon^2}\int_{B} e^{\overline{u}_{1,\varepsilon}+u_{0,1}+\ln h_1}(1-e^{\overline{u}_{2,\varepsilon}+u_{0,2}+\ln h_2})dx\nonumber\\
&\quad+\frac{1}{\varepsilon^2}\int_B\big\langle x,\nabla(e^{\overline{u}_{2,\varepsilon}+u_{0,2}+\ln h_2})\big\rangle e^{\overline{u}_{1,\varepsilon}+u_{0,1}+\ln  h_1}dx\nonumber\\
&\quad-\frac{1}{\varepsilon^2}\int_B \big\langle x,\nabla(u_{0,1}+\ln h_1)\big\rangle e^{\overline{u}_{1,\varepsilon}+u_{0,1}+\ln  h_1}(1-e^{\overline{u}_{2,\varepsilon}+u_{0,2}+\ln  h_2})dx.\nonumber
\end{align}
Together, these imply 
\begin{align}
0&=\int_{\partial B}\bigg[ \big\langle\nu,\nabla\overline{u}_{1,\varepsilon}\big\rangle \big\langle x,\nabla \overline{u}_{2,\varepsilon}\big\rangle+ \big\langle\nu,\nabla\overline{u}_{2,\varepsilon}\big\rangle \big\langle x,\nabla \overline{u}_{1,\varepsilon}\big\rangle-\big\langle x,\nu\big\rangle \big\langle\nabla\overline{u}_{1,\varepsilon},\nabla\overline{u}_{2,\varepsilon}\big\rangle \bigg]d\mathcal{H}^1\nonumber\\
&\quad+\frac{1}{\varepsilon^2}\int_{\partial B}\big\langle x,\nu\big\rangle \Big[e^{\overline{u}_{1,\varepsilon}+u_{0,1}+\ln h_1}+ e^{\overline{u}_{2,\varepsilon}+u_{0,2}+\ln h_2} \nonumber\\
&\quad\quad\quad\quad\quad\quad\quad - e^{\overline{u}_{1,\varepsilon}+u_{0,1}+\ln h_1+\overline{u}_{2,\varepsilon}+u_{0,2}+\ln h_2}\Big]d\mathcal{H}^1\nonumber\\
&\quad-\frac{2}{\varepsilon^2}\int_B \Big[e^{\overline{u}_{1,\varepsilon}+u_{0,1}+\ln h_1}+ e^{\overline{u}_{2,\varepsilon}+u_{0,2}+\ln h_2} \nonumber\\
&\quad\quad\quad\quad\quad - e^{\overline{u}_{1,\varepsilon}+u_{0,1}+\ln h_1+\overline{u}_{2,\varepsilon}+u_{0,2}+\ln h_2}\Big]dx\nonumber\\
&\quad-\frac{1}{\varepsilon^2}\int_B\big\langle x,\nabla(u_{0,1}+\ln  h_1)\big\rangle e^{\overline{u}_{1,\varepsilon}+u_{0,1}+\ln h_1}(1- e^{\overline{u}_{2,\varepsilon}+u_{0,2}+\ln h_2})dx\nonumber\\
&\quad-\frac{1}{\varepsilon^2}\int_B\big\langle x,\nabla(u_{0,2}+\ln  h_2)\big\rangle e^{\overline{u}_{2,\varepsilon}+u_{0,2}+\ln h_2}(1- e^{\overline{u}_{1,\varepsilon}+u_{0,1}+\ln h_1})dx.\nonumber
\end{align}
In other words, we get (\ref{e:PohozaevDilation}).
\begin{flushright}
$\Box$
\end{flushright}

By Proposition \ref{prop:CSsystemPohozaev}, we get
\begin{align}\label{e:PohozaevXiDilation}
&\int_{\partial B_r(x_{j,\varepsilon}^{(1)})}\bigg[2r\Big(\langle\nu,\nabla \overline{u}_{1,\varepsilon}^{(1)}\rangle \langle\nu,\nabla\xi_{2,\varepsilon}\rangle+\langle\nu,\nabla \overline{u}_{2,\varepsilon}^{(1)}\rangle \langle\nu,\nabla\xi_{1,\varepsilon}\rangle\Big)\nonumber\\
&\quad\quad\quad\quad\quad -r\Big(\langle\nabla \overline{u}_{1,\varepsilon}^{(1)},\nabla\xi_{2,\varepsilon}\rangle+\langle\nabla \overline{u}_{2,\varepsilon}^{(1)},\nabla\xi_{1,\varepsilon}\rangle\Big)\bigg]d\mathcal{H}^1\nonumber\\
&=-\frac{1}{\varepsilon^2}\int_{\partial B_r(x_{j,\varepsilon}^{(1)})}r\cdot g_1d\mathcal{H}^1+\frac{2}{\varepsilon^2}\int_{ B_r(x_{j,\varepsilon}^{(1)})}g_1 dx+\frac{1}{\varepsilon^2} \int_{B_r(x_{j,\varepsilon}^{(1)})} g_2dx
\end{align}
with $\mathcal{D}_\varepsilon := \|u_{1,\varepsilon}^{(1)}-u_{1,\varepsilon}^{(2)}\|_{L^\infty(\Omega)}+\|u_{2,\varepsilon}^{(1)}-u_{2,\varepsilon}^{(2)}\|_{L^\infty(\Omega)}$ properly defining the normalized source terms as
\begin{align}\label{def:g1}
g_1=&\frac{1}{\mathcal{D}_\varepsilon}\bigg[\big(e^{\overline{u}_{1,\varepsilon}^{(1)}+u_{0,1}+\ln h_1}+ e^{\overline{u}_{2,\varepsilon}^{(1)}+u_{0,2}+\ln h_2} - e^{\overline{u}_{1,\varepsilon}^{(1)}+u_{0,1}+\ln h_1+\overline{u}_{2,\varepsilon}^{(1)}+u_{0,2}+\ln h_2}\big)\nonumber\\
&\quad\quad-\big(e^{\overline{u}_{1,\varepsilon}^{(2)}+u_{0,1}+\ln h_1}+ e^{\overline{u}_{2,\varepsilon}^{(2)}+u_{0,2}+\ln h_2} - e^{\overline{u}_{1,\varepsilon}^{(2)}+u_{0,1}+\ln h_1+\overline{u}_{2,\varepsilon}^{(2)}+u_{0,2}+\ln h_2}\big)\bigg]
\end{align}
and
\begin{align}\label{def:g2}
g_2=&\frac{1}{\mathcal{D}_\varepsilon}\bigg[ \Big\langle x-x_{j,\varepsilon}^{(1)},\Big(\nabla(e^{u_{0,1}+\ln h_1})\Big)(e^{\overline{u}_{1,\varepsilon}^{(1)}}-e^{\overline{u}_{1,\varepsilon}^{(2)}}) + \Big(\nabla(e^{u_{0,2}+\ln h_2})\Big)(e^{\overline{u}_{2,\varepsilon}^{(1)}}-e^{\overline{u}_{2,\varepsilon}^{(2)}})\Big\rangle\nonumber\\
&\quad\quad-\Big\langle x-x_{j,\varepsilon}^{(1)}, \Big(\nabla e^{u_{0,1}+\ln h_1+u_{0,2}+\ln h_2}\Big)(e^{\overline{u}_{1,\varepsilon}^{(1)}+\overline{u}_{2,\varepsilon}^{(1)}}-e^{\overline{u}_{1,\varepsilon}^{(2)}+\overline{u}_{2,\varepsilon}^{(2)}})\Big\rangle \bigg].
\end{align}
This is due to (\ref{e:PohozaevDilation}) and the definition of $(\xi_{1,\varepsilon},\xi_{2,\varepsilon})$ (see (\ref{def:Xi})).

Similarly, we get from (\ref{e:PohozaevTranslation}) that
\begin{align}\label{e:PohozaevXiTranslation}
&\int_{\partial B_r(x_{j,\varepsilon}^{(1)})}\bigg( \langle\nu,\nabla\xi_{1,\varepsilon}\rangle\partial_h\overline{u}_{2,\varepsilon}^{(1)}+\langle\nu,\nabla\xi_{2,\varepsilon}\rangle\partial_h\overline{u}_{1,\varepsilon}^{(1)}+\langle\nu,\nabla\overline{u}_{1,\varepsilon}^{(2)}\rangle\partial_h\xi_{2,\varepsilon}+\langle\nu,\nabla\overline{u}_{2,\varepsilon}^{(2)}\rangle\partial_h\xi_{1,\varepsilon}\nonumber\\
&\quad\quad\quad\quad\quad - \langle\nabla\xi_{1,\varepsilon},\nabla\overline{u}_{2,\varepsilon}^{(1)}\rangle\nu_h-\langle\nabla\overline{u}_{1,\varepsilon}^{(2)},\nabla\xi_{2,\varepsilon}\rangle\nu_h\bigg)d\mathcal{H}^1\nonumber\\
&=\int_{B_r(x_{j,\varepsilon}^{(1)})}\Big[f_{1,\varepsilon}\partial_h(u_{0,2}+\ln h_2)+f_{2,\varepsilon}\partial_h(u_{0,2}+\ln h_2)\Big]dx-\frac{1}{\varepsilon^2}\int_{\partial B_r(x_{j,\varepsilon}^{(1)})}\nu_h g_1 d\mathcal{H}^1.
\end{align}
Here, $g_1$ and $f_{i,\varepsilon}$ are defined as in (\ref{def:g1}) and in (\ref{def:fivarepsilon}), respectively.

For the terms in (\ref{e:PohozaevTranslation}) and in (\ref{e:PohozaevDilation}), we have the following propositions, whose proofs are left to Subsection \ref{Subsection:ProofTermsPohozaev}.

\begin{proposition}\label{prop:PohozaevTerms1}
It holds that
\begin{align}\label{e:DilationTerms1}
&\sum_{j=1}^k\int_{\partial B_r(x_{j,\varepsilon}^{(1)})}\bigg[2r\Big(\langle\nu,\nabla \overline{u}_{1,\varepsilon}^{(1)}\rangle \langle\nu,\nabla\xi_{2,\varepsilon}\rangle+\langle\nu,\nabla \overline{u}_{2,\varepsilon}^{(1)}\rangle \langle\nu,\nabla\xi_{1,\varepsilon}\rangle\Big)\nonumber\\
&\quad\quad\quad\quad\quad\quad - r\Big(\langle\nabla \overline{u}_{1,\varepsilon}^{(1)},\nabla\xi_{2,\varepsilon}\rangle+\langle\nabla \overline{u}_{2,\varepsilon}^{(1)},\nabla\xi_{1,\varepsilon}\rangle\Big)\bigg]d\mathcal{H}^1\nonumber\\
&=4\int_\Omega (f_{1,\varepsilon}+f_{2,\varepsilon})+ 256\sum_{i=1}^2 b_{i,0}\sum_{j=1}^k\int_{\Omega_j\backslash B_r(x_{j,\varepsilon}^{(1)})}\frac{(\mu_{j,\varepsilon}^{(1)})^{-2} e^{f_{i,j,\varepsilon}^{(1)}(y)}}{e^{u_{0,i}(x_{j,\varepsilon}^{(1)})}|y-x_{j,\varepsilon}^{(1)}|}dy\nonumber\\
&\quad+o(\mu_{\varepsilon}^{-2})+O\Big(\mu_{\varepsilon}^{-2}\sum_{j=1}^k\big(|A_{1,j,\varepsilon}|+|A_{2,j,\varepsilon}|\big)\Big)
\end{align}
with
\begin{align}
f_{i,j,\varepsilon}^{(1)}(y)=u_{0,i}(y)-u_{0,i}(x_{j,\varepsilon}^{(1)})+M_{i,j,\varepsilon}^{(1)}(\gamma(x,x_{j,\varepsilon}^{(1)})-\gamma(x_{j,\varepsilon}^{(1)},x_{j,\varepsilon}^{(1)}))+\sum_{l\neq j}M_{i,l,\varepsilon}^{(1)}(G(x,x_{l,\varepsilon}^{(1)})-G(x_{j,\varepsilon}^{(1)},x_{l,\varepsilon}^{(1)})),\nonumber
\end{align}
\begin{align}\label{e:DilationTerms2}
-\frac{1}{\varepsilon^2}\sum_{j=1}^k\int_{\partial B_r(x_{j,\varepsilon}^{(1)})}r\cdot g_1d\mathcal{H}^1=128\sum_{i=1}^2\sum_{j=1}^k\frac{b_{i,0}(\mu_{j,\varepsilon}^{(1)})^{-2}}{e^{u_{0,i}(x_{j,\varepsilon}^{(1)})}}\int_{\mathbb{R}^2\backslash B_r(x_{j,\varepsilon}^{(1)})}\frac{dy}{|y-x_{j,\varepsilon}^{(1)}|^4}+o(\mu_{\varepsilon}^{-2}),
\end{align}
\begin{align}\label{e:DilationTerms3}
\frac{2}{\varepsilon^2}\int_{ B_r(x_{j,\varepsilon}^{(1)})}g_1 dx&=2A\sum_{i=1}^2 b_{i,0} \sum_{j=1}^k e^{u_{0,i}(x_{j,\varepsilon}^{(1)})}e^{\beta_{j,\varepsilon}^{(1)}} \nonumber\\
&\quad + \sum_{i=1}^2 b_{i,0}\sum_{j=1}^k\frac{128 (\mu_{j,\varepsilon}^{(1)})^{-2}}{e^{u_{0,i}(x_{j,\varepsilon}^{(1)})}}\int_{\Omega_j\backslash B_r(x_{j,\varepsilon}^{(1)})}\frac{e^{f_{i,j,\varepsilon}^{(1)}(y)}}{|y-x_{j,\varepsilon}^{(1)}|}dy+o(\mu_{\varepsilon}^{-2})
\end{align}
with $A=8\int_{\mathbb{R}^2}\frac{1-|x|^2}{(1+|x|^2)^5}$ and
\begin{align}\label{e:DilationTerms4}
\frac{1}{\varepsilon^2} \sum_{j=1}^k\int_{B_r(x_{j,\varepsilon}^{(1)})} g_2dx=o(\mu_{\varepsilon}^{-2})+o_r(1)\mu_{\varepsilon}^{-2}.
\end{align}
\end{proposition}

Similarly, for the terms in (\ref{e:PohozaevXiTranslation}), we get
\begin{proposition}\label{prop:PohozaevTerms2}
It holds that
\begin{align}\label{e:TranslationTerms1}
&\int_{\partial B_r(x_{j,\varepsilon}^{(1)})}\bigg( \langle\nu,\nabla\xi_{1,\varepsilon}\rangle\partial_h\overline{u}_{2,\varepsilon}^{(1)}+\langle\nu,\nabla\xi_{2,\varepsilon}\rangle\partial_h\overline{u}_{1,\varepsilon}^{(1)}+\langle\nu,\nabla\overline{u}_{1,\varepsilon}^{(2)}\rangle\partial_h\xi_{2,\varepsilon}+\langle\nu,\nabla\overline{u}_{2,\varepsilon}^{(2)}\rangle\partial_h\xi_{1,\varepsilon}\nonumber\\
&\quad\quad\quad\quad\quad - \langle\nabla\xi_{1,\varepsilon},\nabla\overline{u}_{2,\varepsilon}^{(1)}\rangle\nu_h-\langle\nabla\overline{u}_{1,\varepsilon}^{(2)},\nabla\xi_{2,\varepsilon}\rangle\nu_h\bigg)d\mathcal{H}^1\nonumber\\
&=64\pi^2\bigg[(\mu_{j,\varepsilon}^{(1)})^{-1}\sum_{h'=1}^2\partial_h\partial_{h'} \gamma(x_{j,\varepsilon}^{(1)},x_{j,\varepsilon}^{(1)})b_{j,h'}\nonumber\\
&\quad\quad\quad\quad +\sum_{m=1,m\neq j}^k(\mu_{m,\varepsilon}^{(1)})^{-1}\sum_{h'=1}^2\partial_h\partial_{h'} G(x_{m,\varepsilon}^{(1)},x_{j,\varepsilon}^{(1)})b_{m,h'}\bigg]\nonumber\\
&\quad+o_r(1)\mu_{\varepsilon}^{-1}+o(\mu_{\varepsilon}^{-1}),
\end{align}
\begin{align}\label{e:TranslationTerms2}
&\int_{B_r(x_{j,\varepsilon}^{(1)})}\Big[f_{1,\varepsilon}\partial_h(u_{0,2}+\ln h_2)+f_{2,\varepsilon}\partial_h(u_{0,2}+\ln h_2)\Big]dx\nonumber\\
&=-8\pi\sum_{i=1}^2(\mu_{j,\varepsilon}^{(1)})^{-1}\big\langle\nabla\partial_h(u_{0,i}+\phi_{j,\varepsilon}^{(1)})(x_{j,\varepsilon}^{(1)}),\vec{b}_{j}\big\rangle+o(\mu_{\varepsilon}^{-1})
\end{align}
with $\vec{b}_{j}=(b_{j,1},b_{j,2})$
and
\begin{align}\label{e:TranslationTerms3}
-\frac{1}{\varepsilon^2}\int_{\partial B_r(x_{j,\varepsilon}^{(1)})}\nu_h g_1 d\mathcal{H}^1=O(\mu_{\varepsilon}^{-2}).
\end{align}

\end{proposition}

We prove it in Subsection \ref{Subsection:ProofTermsPohozaev}.

\subsection{Proof of Theorem \ref{t:Unique}}
In this subsection, we prove Theorem \ref{t:Unique}.

\,\,

\noindent{\bf Proof of Theorem \ref{t:Unique}.}

\,\,

\noindent{\bf Step 1. }$b_{1,0}=b_{2,0}=0$.

\,\,

Summing (\ref{e:PohozaevXiDilation}) and plugging (\ref{e:DilationTerms1}), (\ref{e:DilationTerms2}), (\ref{e:DilationTerms3}) and (\ref{e:DilationTerms4}) into it, we get
\begin{align}
&256\sum_{i=1}^2 b_{i,0}\sum_{j=1}^k\int_{\Omega_j\backslash B_r(x_{j,\varepsilon}^{(1)})}\frac{(\mu_{j,\varepsilon}^{(1)})^{-2} e^{f_{i,j,\varepsilon}^{(1)}(y)}}{e^{u_{0,i}(x_{j,\varepsilon}^{(1)})}|y-x_{j,\varepsilon}^{(1)}|}dy\nonumber\\
&\quad+o(\mu_{\varepsilon}^{-2})+O\Big(\mu_{\varepsilon}^{-2}\sum_{j=1}^k\big(|A_{1,j,\varepsilon}|+|A_{2,j,\varepsilon}|\big)\Big)\nonumber\\
&=128\sum_{i=1}^2\sum_{j=1}^k\frac{b_{i,0}(\mu_{j,\varepsilon}^{(1)})^{-2}}{e^{u_{0,i}(x_{j,\varepsilon}^{(1)})}}\int_{\mathbb{R}^2\backslash B_r(x_{j,\varepsilon}^{(1)})}\frac{dy}{|y-x_{j,\varepsilon}^{(1)}|^4}+2A\sum_{i=1}^2 b_{i,0} \sum_{j=1}^k e^{u_{0,i}(x_{j,\varepsilon}^{(1)})}e^{\beta_{j,\varepsilon}^{(1)}}\nonumber\\
&\quad+ 128\sum_{i=1}^2b_{i,0}\sum_{j=1}^k\frac{(\mu_{j,\varepsilon}^{(1)})^{-2}}{e^{u_{0,i}(x_{j,\varepsilon}^{(1)})}}\int_{\Omega_j\backslash B_r(x_{j,\varepsilon}^{(1)})}\frac{e^{f_{i,j,\varepsilon}^{(1)}(y)}}{|y-x_{j,\varepsilon}^{(1)}|}dy+o(\mu_{\varepsilon}^{-2})+o_r(\mu_{\varepsilon}^{-2})
\end{align}
with $A=8\int_{\mathbb{R}^2}\frac{1-|x|^2}{(1+|x|^2)^5}$. This is equivalence to
\begin{align}\label{e:b0=0-1}
&-128\sum_{i=1}^2 b_{i,0}\sum_{j=1}^k\frac{(\mu_{j,\varepsilon}^{(2)})^{-2}}{e^{u_{0,i}(x_{j,\varepsilon}^{(1)})}} \Bigg(\int_{\Omega\backslash B_r(x_{j,\varepsilon}^{(1)})} \frac{e^{f_{i,j,\varepsilon}^{(1)}(y)}-1}{|y-x_{j,\varepsilon}^{(1)}|^4}-\int_{\mathbb{R}^2\backslash\Omega_j}\frac{1}{|y-x_{j,\varepsilon}^{(1)}|}\Bigg)\nonumber\\
&\quad + 2A\sum_{i=1}^2 b_{i,0}\sum_{j=1}^k e^{u_{0,i}(x_{j,\varepsilon}^{(1)})+\beta_{j,\varepsilon}^{(1)}}\nonumber\\
&=o_r(1)\mu_{\varepsilon}^{-2}+o(\mu_{\varepsilon}^{-2})+O\Big(\mu_{\varepsilon}^{-2}\sum_{j=1}^k\big(|A_{1,j,\varepsilon}|+|A_{2,j,\varepsilon}|\big)\Big).
\end{align}
On the other hand, we get
\begin{claim}\label{c:Aij}
It holds that
\begin{align}
|A_{i,j,\varepsilon}|=O(\mu_{\varepsilon}^{-2}).\nonumber
\end{align}
\end{claim}
Now we prove Claim \ref{c:Aij}. Analogues to (\ref{e:DilationTerms1}), (\ref{e:DilationTerms2}), (\ref{e:DilationTerms3}) and (\ref{e:DilationTerms4})
on $\Omega_j$ imply that
\begin{align}
-4A_{i,j,\varepsilon}=-2\int_{\Omega_j}g_1+O(\mu_{\varepsilon}^{-2}).\nonumber
\end{align}
Using (\ref{def:g1}), we get
\begin{align}
\int_{\Omega_j}g_1=A_{i,j,\varepsilon}+O(\mu_{\varepsilon}^{-2}).\nonumber
\end{align}
These imply the result.

Combining (\ref{e:b0=0-1}) and Claim \ref{c:Aij}, we get
\begin{align}
&-128\sum_{i=1}^2 b_{i,0}\sum_{j=1}^k\frac{(\mu_{j,\varepsilon}^{(2)})^{-2}}{e^{u_{0,i}(x_{j,\varepsilon}^{(1)})}} \Bigg(\int_{\Omega\backslash B_r(x_{j,\varepsilon}^{(1)})} \frac{e^{f_{i,j,\varepsilon}^{(1)}(y)}-1}{|y-x_{j,\varepsilon}^{(1)}|^4}-\int_{\mathbb{R}^2\backslash\Omega_j}\frac{1}{|y-x_{j,\varepsilon}^{(1)}|}\Bigg)\nonumber\\
&\quad +2A\sum_{i=1}^2 b_{i,0}\sum_{j=1}^k e^{u_{0,i}(x_{j,\varepsilon}^{(1)})+\beta_{j,\varepsilon}^{(1)}}\nonumber\\
&=o_r(1)\mu_{\varepsilon}^{-2}+o(\mu_{\varepsilon}^{-2}).\nonumber
\end{align}
Since we assume $D(\vec{q})<0$, passing $r\to0+$, we get $b_{1,0}=b_{2,0}=0$.

\,\,

\noindent{\bf Step 2. }$b_{j,1}=b_{j,2}=0$ for any $j=1,\cdots,k$.

\,\,

Plugging (\ref{e:TranslationTerms1}), (\ref{e:TranslationTerms2}) and (\ref{e:TranslationTerms3}) into (\ref{e:PohozaevXiTranslation}), we get

\begin{align}
D^2 \mathcal{F}(x_{1,\varepsilon}^{(1)},\cdots,x_{k,\varepsilon}^{(1)}) \cdot \vec{\overline{b}}_{h}=o(1)\nonumber
\end{align}

for $h=1,2$. Here, the functional $\mathcal{F}$ is defined as

\begin{align}
\mathcal{F}(x_1,\cdots,x_k) := \sum_{m=1}^k\big(u_{0,1}(x_m)+u_{0,2}(x_m)\big)+16\pi\sum_{m=1}^k\Big(\gamma(x_m,x_m)+\sum_{l\neq m}G(x_m,x_l)\Big).\nonumber
\end{align}

The vector is given by
\begin{align}
\vec{\overline{b}}_{h}=(\rho_1 b_{1,h},\cdots,\rho_k b_{k,h}).\nonumber
\end{align}
The non-degeneracy of the critical point of $\sum_{m=1}^k\big(u_{0,1}(x_m)+u_{0,2}(x_m)\big)+16\pi\sum_{m=1}^k\Big(\gamma(x_m,x_m)+\sum_{l\neq m}G(x_m,x_l)\Big)$ implies the result.

\,\,

\noindent{\bf Step 3. Completing the proof.}

\,\,

Under the above consideration, the contradiction follows the definition (\ref{def:Xi}) immediately.

\begin{flushright}
$\Box$
\end{flushright}

\subsection{Proof of Theorem \ref{t:nondegenerate}}
In this subsection, we give a sketch of the proof of Theorem \ref{t:nondegenerate} since it is similar with Theorem \ref{t:Unique}.

\,\,

\noindent{\bf Proof of Theorem \ref{t:nondegenerate}.}
Let us argue by contradiction.
By the definition of the non-degeneracy (Definition \ref{def:NonDegeneracy}), we assume that the equation
\begin{equation}\label{e:Linear02}
    \begin{cases}
     \Delta \phi_1+\frac{1}{\varepsilon^2}\Big[e^{u_2+u_{0,2}}\phi_2 - e^{u_1+u_2+u_{0,1}+u_{0,2}}(\phi_1+\phi_2)\Big]=0\mbox{ in }\Omega,\\ \\
     \Delta \phi_2+\frac{1}{\varepsilon^2}\Big[e^{u_1+u_{0,1}}\phi_1 - e^{u_1+u_2+u_{0,1}+u_{0,2}}(\phi_1+\phi_2)\Big]=0\mbox{ in }\Omega,\\ \\
     \phi_1,\phi_2\in H^1_{per}(\Omega).
   \end{cases}
\end{equation}
admits a nontrivial solution $(\phi_1,\phi_2)$ for sufficiently small $\varepsilon$. Let
\begin{align}\label{def:xi02}
\xi_{1}:=\frac{\phi_1}{\|\phi_1\|_{L^\infty(\Omega)}+\|\phi_2\|_{L^\infty(\Omega)}}\mbox{ and }\xi_{2}:=\frac{\phi_2}{\|\phi_1\|_{L^\infty(\Omega)}+\|\phi_2\|_{L^\infty(\Omega)}}.
\end{align}
We know that $(\xi_1,\xi_2)$ is also a solution to (\ref{e:Linear02}).
Following a similar argument as in the proof of Theorem \ref{t:Unique}, we find that $\|\xi_1\|_{L^\infty(\Omega)}+\|\xi_2\|_{L^\infty(\Omega)}=o_\varepsilon(1)$. This contradicts with (\ref{def:xi02}).

\begin{flushright}
$\Box$
\end{flushright}

\subsection{Proof of Proposition \ref{prop:PohozaevTerms1} and Proposition \ref{prop:PohozaevTerms2}}\label{Subsection:ProofTermsPohozaev}
In this subsection, we prove Proposition \ref{prop:PohozaevTerms1} and Proposition \ref{prop:PohozaevTerms2}.

\,\,

\noindent{\bf Proof of Proposition \ref{prop:PohozaevTerms1}.}

\,\,

\noindent{\bf Proof of (\ref{e:DilationTerms1}).}
Denote
\begin{align}\label{def:OverlineG}
\overline{G}(x)=8\pi\sum_{j=1}^k G(x,x_{j,\varepsilon}^{(1)}).
\end{align}
Then, we get
\begin{align}
&\sum_{j=1}^k\int_{\partial B_r(x_{j,\varepsilon}^{(1)})}\bigg[2r\Big(\langle\nu,\nabla \overline{u}_{1,\varepsilon}^{(1)}\rangle \langle\nu,\nabla\xi_{2,\varepsilon}\rangle+\langle\nu,\nabla \overline{u}_{2,\varepsilon}^{(1)}\rangle \langle\nu,\nabla\xi_{1,\varepsilon}\rangle\Big)\nonumber\\
&\quad\quad\quad\quad\quad\quad - r\Big(\langle\nabla \overline{u}_{1,\varepsilon}^{(1)},\nabla\xi_{2,\varepsilon}\rangle+\langle\nabla \overline{u}_{2,\varepsilon}^{(1)},\nabla\xi_{1,\varepsilon}\rangle\Big)\bigg]d\mathcal{H}^1\nonumber\\
&=\sum_{j=1}^k\int_{\partial B_r(x_{j,\varepsilon}^{(1)})}\bigg[2r\langle\nu,\nabla(\overline{G}(x)-\phi_{j,\varepsilon}^{(1)}(x))\rangle \langle\nu,\nabla(\xi_{1,\varepsilon}+\xi_{2,\varepsilon})\rangle \nonumber\\
&\quad\quad\quad\quad\quad\quad\quad - r\langle\nabla(\overline{G}(x)-\phi_{j,\varepsilon}^{(1)}(x)),\nabla(\xi_{1,\varepsilon}+\xi_{2,\varepsilon})\rangle \bigg]d\mathcal{H}^1\nonumber\\
&\quad\quad+o(\mu_{\varepsilon}^{-2})+O\Big(\mu_{\varepsilon}^{-2}\sum_{j=1}^k\big(|A_{1,j,\varepsilon}|+|A_{2,j,\varepsilon}|\big)\Big).\nonumber
\end{align}
Here, we use (\ref{e:UOutsideBubblingDiscs}), Corollary \ref{coro:xj-qj}, Remark \ref{r:Dxi=o(1)} and (\ref{def:Aij}). Moreover, we get
\begin{align}
\nabla(\overline{G}(x)-\phi_{j,\varepsilon}^{(1)}(x))=-\frac{4(x-x_{j,\varepsilon}^{(1)})}{|x-x_{j,\varepsilon}^{(1)}|^2}.\nonumber
\end{align}
This implies that
\begin{align}\label{e:DilationTerms1TOTAL}
&\sum_{j=1}^k\int_{\partial B_r(x_{j,\varepsilon}^{(1)})}\bigg[2r\Big(\langle\nu,\nabla \overline{u}_{1,\varepsilon}^{(1)}\rangle \langle\nu,\nabla\xi_{2,\varepsilon}\rangle+\langle\nu,\nabla \overline{u}_{2,\varepsilon}^{(1)}\rangle \langle\nu,\nabla\xi_{1,\varepsilon}\rangle\Big)\nonumber\\
&\quad\quad\quad\quad\quad\quad - r\Big(\langle\nabla \overline{u}_{1,\varepsilon}^{(1)},\nabla\xi_{2,\varepsilon}\rangle+\langle\nabla \overline{u}_{2,\varepsilon}^{(1)},\nabla\xi_{1,\varepsilon}\rangle\Big)\bigg]d\mathcal{H}^1\nonumber\\
&=-4\sum_{j=1}^k\int_{\partial B_r(x_{j,\varepsilon}^{(1)})}\langle\nu,\nabla(\xi_{1,\varepsilon}+\xi_{2,\varepsilon})\rangle+o(\mu_{\varepsilon}^{-2})+O\Big(\mu_{\varepsilon}^{-2}\sum_{j=1}^k\big(|A_{1,j,\varepsilon}|+|A_{2,j,\varepsilon}|\big)\Big).
\end{align}
Now we expand $-4\int_{\partial B_r(x_{j,\varepsilon}^{(1)})}\langle\nu,\nabla(\xi_{1,\varepsilon}+\xi_{2,\varepsilon})\rangle$. Using Green's representation and (\ref{e:xi}), 
\begin{align}
\xi_{1,\varepsilon}(x)+\xi_{2,\varepsilon}(x)&=\int_\Omega(\xi_{1,\varepsilon}+\xi_{2,\varepsilon})+\int_\Omega G(y,x)(f_{1,\varepsilon}(y)+f_{2,\varepsilon}(y))dy\nonumber\\
&=\int_\Omega(\xi_{1,\varepsilon}+\xi_{2,\varepsilon})+\sum_{m=1}^k A_{m,\varepsilon}G(x_{m,\varepsilon}^{(1)},x)+\sum_{m=1}^k\sum_{h=1}^2 B_{m,h,\varepsilon}\partial_h G(x_{m,\varepsilon}^{(1)},x)\nonumber\\
&\quad+\frac{1}{2}\sum_{m=1}^k\sum_{h,l=1}^2 C_{m,h,l,\varepsilon}\partial_h\partial_lG(x_{m,\varepsilon}^{(1)},x)+\sum_{j=1}^k\int_{\Omega_j}\Psi_{j,\varepsilon}(y,x)(f_{1,\varepsilon}(y)+f_{2,\varepsilon}(y))dy.\nonumber
\end{align}
Here,
\begin{align}\label{def:Avarepsilonj}
A_{m,\varepsilon}=\int_{\Omega_m}(f_{1,\varepsilon}(y)+f_{2,\varepsilon}(y))dy,
\end{align}
\begin{align}
B_{m,h,\varepsilon}=\int_{\Omega_m}(y-x_{m,\varepsilon}^{(1)})_h(f_{1,\varepsilon}(y)+f_{2,\varepsilon}(y))dy,\nonumber
\end{align}
\begin{align}
C_{m,h,l,\varepsilon}=\int_{\Omega_m}(y-x_{m,\varepsilon}^{(1)})_h(y-x_{m,\varepsilon}^{(1)})_l(f_{1,\varepsilon}(y)+f_{2,\varepsilon}(y))dy\nonumber
\end{align}
and 
\begin{align}\label{def:Psi}
\Psi_{j,\varepsilon}(y,x)&=G(y,x)-G(x_{j,\varepsilon}^{(1)},x)- \big\langle\nabla G(x_{j,\varepsilon}^{(1)},x),y-x_{j,\varepsilon}^{(1)}\big\rangle\nonumber\\
&\quad - \frac{1}{2}(y-x_{j,\varepsilon}^{(1)})^T D^2 G(x_{j,\varepsilon}^{(1)},x)(y-x_{j,\varepsilon}^{(1)}).
\end{align}
We define
\begin{align}\label{def:OverlineG*}
\overline{G}^*(x)&=\frac{1}{|\Omega|}\int_\Omega(\xi_{1,\varepsilon}+\xi_{2,\varepsilon})+\sum_{m=1}^k A_{m,\varepsilon}G(x_{m,\varepsilon}^{(1)},x)\nonumber\\
&\quad + \sum_{m=1}^k\sum_{h=1}^2 B_{m,h,\varepsilon}\partial_h G(x_{m,\varepsilon}^{(1)},x)+\frac{1}{2}\sum_{m=1}^k\sum_{h,l=1}^2 C_{m,h,l,\varepsilon}\partial_h\partial_lG(x_{m,\varepsilon}^{(1)},x),
\end{align}
a harmonic function in $B_r(x_{j,\varepsilon}^{(1)})\backslash\{x_{j,\varepsilon}^{(1)}\}$. For $x\in \partial B_r(x_{j,\varepsilon}^{(1)})$, we get
\begin{align}
\xi_{1,\varepsilon}(x)+\xi_{2,\varepsilon}(x)-\overline{G}^*(x)&=\sum_{j=1}^k\int_{\Omega_j\backslash B_\theta(x_{j,\varepsilon}^{(1)})}\Psi_{j,\varepsilon}(y,x)(f_{1,\varepsilon}(y)+f_{2,\varepsilon}(y))dy\nonumber\\
&\quad+\sum_{j=1}^k\int_{B_\theta(x_{j,\varepsilon}^{(1)})}\Psi_{j,\varepsilon}(y,x)(f_{1,\varepsilon}(y)+f_{2,\varepsilon}(y))dy\nonumber\\
&=-\sum_{i=1}^2 b_{i,0}\sum_{j=1}^k\int_{\Omega_j\backslash B_\theta(x_{j,\varepsilon}^{(1)})}\Psi_{j,\varepsilon}(y,x)\frac{64 (\mu_{j,\varepsilon}^{(1)})^{-2}e^{f_{i,j,\varepsilon}^{(1)}(y)}}{e^{u_{0,i}(x_{j,\varepsilon}^{(1)})}|y-x_{j,\varepsilon}^{(1)}|^4}dy\nonumber\\
&\quad +o(\mu_{\varepsilon}^{-2})+O\Big(\theta r^{-3}\mu_{\varepsilon}^{-2}\Big).\nonumber
\end{align}
Here, $\theta$ is a positive constant smaller than $r$ and
\begin{align}
f_{i,j,\varepsilon}^{(1)}(y)=u_{0,i}(y)-u_{0,i}(x_{j,\varepsilon}^{(1)})+M_{i,j,\varepsilon}^{(1)}(\gamma(x,x_{j,\varepsilon}^{(1)})-\gamma(x_{j,\varepsilon}^{(1)},x_{j,\varepsilon}^{(1)}))+\sum_{l\neq j}M_{i,l,\varepsilon}^{(1)}(G(x,x_{l,\varepsilon}^{(1)})-G(x_{j,\varepsilon}^{(1)},x_{l,\varepsilon}^{(1)})).\nonumber
\end{align}
The second equality is due to (\ref{def:fi}), (\ref{Asy:c1}), (\ref{Asy:c2}) and Lemma \ref{l:M8pimu-2}.
Taking $\theta=O(r^3)$, we have
\begin{align}
    \xi_{1,\varepsilon}(x)+\xi_{2,\varepsilon}(x)-\overline{G}^*(x)&=-\sum_{i=1}^2 b_{i,0}\sum_{j=1}^k\int_{\Omega_j\backslash B_\theta(x_{j,\varepsilon}^{(1)})}\Psi_{j,\varepsilon}(y,x)\frac{64 (\mu_{j,\varepsilon}^{(1)})^{-2}e^{f_{i,j,\varepsilon}^{(1)}(y)}}{e^{u_{0,i}(x_{j,\varepsilon}^{(1)})}|y-x_{j,\varepsilon}^{(1)}|^4}dy\nonumber\\
    &\quad +o(\mu_{\varepsilon}^{-2})+o_r(1)\mu_{\varepsilon}^{-2}.\nonumber
\end{align}
Denoting
\begin{align}
\xi_\varepsilon^*(x):=-\sum_{i=1}^2 b_{i,0}\sum_{j=1}^k\int_{\Omega_j\backslash B_\theta(x_{j,\varepsilon}^{(1)})}\Psi_{j,\varepsilon}(y,x)\frac{64 (\mu_{j,\varepsilon}^{(1)})^{-2}e^{f_{i,j,\varepsilon}^{(1)}(y)}}{e^{u_{0,i}(x_{j,\varepsilon}^{(1)})}|y-x_{j,\varepsilon}^{(1)}|^4}dy,\nonumber
\end{align}
we get from (\ref{e:DilationTerms1TOTAL})
\begin{align}\label{e:DilationTerms1TOTAL01}
&\sum_{j=1}^k\int_{\partial B_r(x_{j,\varepsilon}^{(1)})}\bigg[2r\Big(\langle\nu,\nabla \overline{u}_{1,\varepsilon}^{(1)}\rangle \langle\nu,\nabla\xi_{2,\varepsilon}\rangle+\langle\nu,\nabla \overline{u}_{2,\varepsilon}^{(1)}\rangle \langle\nu,\nabla\xi_{1,\varepsilon}\rangle\Big)\nonumber\\
&\quad\quad\quad\quad\quad\quad - r\Big(\langle\nabla \overline{u}_{1,\varepsilon}^{(1)},\nabla\xi_{2,\varepsilon}\rangle+\langle\nabla \overline{u}_{2,\varepsilon}^{(1)},\nabla\xi_{1,\varepsilon}\rangle\Big)\bigg]d\mathcal{H}^1\nonumber\\
&=-4\sum_{j=1}^k\int_{\partial B_r(x_{j,\varepsilon}^{(1)})}\langle\nu,\nabla(\overline{G}^*+\xi_\varepsilon^*)\rangle d\mathcal{H}^1+o(\mu_{\varepsilon}^{-2})\nonumber\\
&\quad +o_r(1)\mu_{\varepsilon}^{-2}+O\Big(\mu_{\varepsilon}^{-2}\sum_{j=1}^k\big(|A_{1,j,\varepsilon}|+|A_{2,j,\varepsilon}|\big)\Big).
\end{align}

Now we estimate $\int_{\partial B_r(x_{j,\varepsilon}^{(1)})}\langle\nu,\nabla\overline{G}^*\rangle d\mathcal{H}^1$. Since $\overline{G}^*$ is harmonic in $B_r(x_{j,\varepsilon}^{(1)})\backslash\{x_{j,\varepsilon}^{(1)}\}$, by divergent theorem, we get
\begin{align}
\int_{\partial B_r(x_{j,\varepsilon}^{(1)})}\langle\nu,\nabla\overline{G}^*\rangle d\mathcal{H}^1=\int_{\partial B_\theta(x_{j,\varepsilon}^{(1)})}\langle\nu,\nabla\overline{G}^*\rangle d\mathcal{H}^1.\nonumber
\end{align}
Recall the definition of $\overline{G}^*$ (\ref{def:OverlineG*}).
First, we get
\begin{align}
\int_{\partial B_\theta(x_{j,\varepsilon}^{(1)})}\bigg\langle\nu,\sum_{m=1}^k A_{m,\varepsilon} \nabla G(x_{m,\varepsilon}^{(1)},x)\bigg\rangle d\mathcal{H}^1=-A_{j,\varepsilon}+o_\theta(1).\nonumber
\end{align}
Here, $A_{j,\varepsilon}$ is defined as in (\ref{def:Avarepsilonj}). This is due to the divergent theorem and the property of Dirac measure. Secondly, we see
\begin{align}
&\int_{\partial B_\theta(x_{j,\varepsilon}^{(1)})}\bigg\langle\nu,\sum_{h=1}^2\sum_{m=1}^k B_{m,h,\varepsilon} \nabla \partial_h G(x_{m,\varepsilon}^{(1)},x)\bigg\rangle d\mathcal{H}^1\nonumber\\
&=\frac{1}{4\pi}\sum_{h=1}^2 B_{j,h,\varepsilon}\int_{\partial B_\theta(x_{j,\varepsilon}^{(1)})}\bigg\langle\nu,\nabla\partial_h \ln\frac{1}{|x_{j,\varepsilon}^{(1)}-x|}\bigg\rangle d\mathcal{H}^1+o_\theta(1)\nonumber\\
&=\frac{1}{4\pi}\sum_{h=1}^2\sum_{l=1}^2 B_{j,h,\varepsilon} \int_{\partial B_\theta(x_{j,\varepsilon}^{(1)})}\frac{(x-x_{j,\varepsilon}^{(1)})_l}{|x-x_{j,\varepsilon}^{(1)}|}\partial_l\partial_h \ln\frac{1}{|x_{j,\varepsilon}^{(1)}-x|}d\mathcal{H}^1+o_\theta(1)\nonumber\\
&=\frac{1}{4\pi}\sum_{h=1}^2 B_{j,h,\varepsilon} \int_{\partial B_\theta(x_{j,\varepsilon}^{(1)})}\frac{(x-x_{j,\varepsilon}^{(1)})_{h'}}{|x-x_{j,\varepsilon}^{(1)}|}\partial_{h'}\partial_h \ln\frac{1}{|x_{j,\varepsilon}^{(1)}-x|}d\mathcal{H}^1+o_\theta(1)\nonumber
\end{align}
Here, the index $h'$ satisfies $1'=2$ and $2'=1$. In the third equality, we use the symmetry. Notice that
\begin{align}
\partial_1\partial_2\ln|x|=-\frac{2x_1 x_2}{(x_1^2 +x_2^2)^2},\nonumber
\end{align}
we get
\begin{align}
\int_{\partial B_\theta(x_{j,\varepsilon}^{(1)})}\bigg\langle\nu,\sum_{h=1}^2\sum_{m=1}^k B_{m,h,\varepsilon} \nabla \partial_h G(x_{m,\varepsilon}^{(1)},x)\bigg\rangle d\mathcal{H}^1=o_\theta(1)\nonumber
\end{align}
by the symmetry.
By a similar method, we get
\begin{align}
\int_{\partial B_\theta(x_{j,\varepsilon}^{(1)})}\bigg\langle\nu,\frac{1}{2}\sum_{m=1}^k\sum_{h,l=1}^2 C_{m,h,l,\varepsilon}\partial_h\partial_lG(x_{m,\varepsilon}^{(1)},x)\bigg\rangle d\mathcal{H}^1=o_\theta(1).\nonumber
\end{align}
Together, we get
\begin{align}\label{e:intvnablaG}
\int_{\partial B_r(x_{j,\varepsilon}^{(1)})}\langle\nu,\nabla\overline{G}^*\rangle d\mathcal{H}^1=-A_{j,\varepsilon}+o_\theta(1).
\end{align}

Now we estimate $\int_{\partial B_r(x_{j,\varepsilon}^{(1)})}\langle\nu,\nabla\xi_\varepsilon^*\rangle d\mathcal{H}^1$. Rewrite
\begin{align}
\xi_\varepsilon^*(x)&=-\sum_{i=1}^2 b_{i,0}\Bigg(\int_{B_r(x_{j,\varepsilon}^{(1)})\backslash B_\theta(x_{j,\varepsilon}^{(1)})}+\int_{\Omega_j\backslash B_r(x_{j,\varepsilon}^{(1)})}+\sum_{m\neq j}\int_{\Omega_m}\Bigg)\nonumber\\
&\quad \times \Psi_{j,\varepsilon}(y,x)\frac{64 (\mu_{j,\varepsilon}^{(1)})^{-2}e^{f_{i,j,\varepsilon}^{(1)}(y)}}{e^{u_{0,i}(x_{j,\varepsilon}^{(1)})}|y-x_{j,\varepsilon}^{(1)}|^4}dy\nonumber\\
&=:I_1(x)+I_2(x)+I_3(x).\nonumber
\end{align}
Then, we know that $\int_{\partial B_r(x_{j,\varepsilon}^{(1)})}\langle\nu,\nabla\xi_\varepsilon^*\rangle = \int_{\partial B_r(x_{j,\varepsilon}^{(1)})}\langle\nu,\nabla I_1\rangle + \int_{\partial B_r(x_{j,\varepsilon}^{(1)})}\langle\nu,\nabla I_2\rangle$. As for $\int_{\partial B_r(x_{j,\varepsilon}^{(1)})}\langle\nu,\nabla I_1\rangle$, noticing that
\begin{align}
\Delta_x\Psi_{j,\varepsilon}(x)=-\delta_y+\delta_{x_{j,\varepsilon}^{(1)}}\nonumber
\end{align}
due to (\ref{def:Psi}). Since $y\in B_r(x_{j,\varepsilon}^{(1)})\backslash B_\theta(x_{j,\varepsilon}^{(1)})$, 
we get
\begin{align}
&\int_{\partial B_r(x_{j,\varepsilon}^{(1)})}\langle\nu,\nabla I_1\rangle d\mathcal{H}^1=\int_{B_r(x_{j,\varepsilon}^{(1)})}\Delta_x I_1(x)dx\nonumber\\
&=-\sum_{i=1}^2 b_{i,0}\int_{B_r(x_{j,\varepsilon}^{(1)})}\Delta_x\Psi_{j,\varepsilon}(y,x)dx\frac{64 (\mu_{j,\varepsilon}^{(1)})^{-2}e^{f_{i,j,\varepsilon}^{(1)}(y)}}{e^{u_{0,i}(x_{j,\varepsilon}^{(1)})}|y-x_{j,\varepsilon}^{(1)}|^4}dy=0.\nonumber
\end{align}
Moreover, we get
\begin{align}
\int_{\partial B_r(x_{j,\varepsilon}^{(1)})}\langle\nu,\nabla I_3\rangle d\mathcal{H}^1=0.\nonumber
\end{align}
On the other hand, a similar computation can be proceed for $\int_{\partial B_r(x_{j,\varepsilon}^{(1)})}\langle\nu,\nabla I_2\rangle$ and we find that
\begin{align}
\int_{\partial B_r(x_{j,\varepsilon}^{(1)})}\langle\nu,\nabla I_2\rangle d\mathcal{H}^1=-64\sum_{i=1}^2 b_{i,0}\int_{\Omega_j\backslash B_r(x_{j,\varepsilon}^{(1)})}\frac{(\mu_{j,\varepsilon}^{(1)})^{-2} e^{f_{i,j,\varepsilon}^{(1)}(y)}}{e^{u_{0,i}(x_{j,\varepsilon}^{(1)})}|y-x_{j,\varepsilon}^{(1)}|}dy.\nonumber
\end{align}
This leads us to
\begin{align}\label{e:intvnablaXi}
\int_{\partial B_r(x_{j,\varepsilon}^{(1)})}\langle\nu,\nabla\xi_\varepsilon^*\rangle d\mathcal{H}^1=-64\sum_{i=1}^2 b_{i,0}\int_{\Omega_j\backslash B_r(x_{j,\varepsilon}^{(1)})}\frac{(\mu_{j,\varepsilon}^{(1)})^{-2} e^{f_{i,j,\varepsilon}^{(1)}(y)}}{e^{u_{0,i}(x_{j,\varepsilon}^{(1)})}|y-x_{j,\varepsilon}^{(1)}|}dy.
\end{align}
(\ref{e:DilationTerms1TOTAL01}),
(\ref{e:intvnablaG}), (\ref{def:Avarepsilonj}) and (\ref{e:intvnablaXi}) give us (\ref{e:DilationTerms1}).

\,\,

\noindent{\bf Proof of (\ref{e:DilationTerms2}).}
By (\ref{def:g1}), Lemma \ref{l:M8pimu-2} and Lemma \ref{l:xitobi0}, a direct computation gives
\begin{align}
-\frac{1}{\varepsilon^2}\int_{\partial B_r(x_{j,\varepsilon}^{(1)})}r\cdot g_1d\mathcal{H}^1&=64\sum_{i=1}^2\frac{b_{i,0}(\mu_{j,\varepsilon}^{(1)})^{-2}}{e^{u_{0,i}(x_{j,\varepsilon}^{(1)})}}\int_{\partial B_r(x_{j,\varepsilon}^{(1)})}\frac{dy}{|y-x_{j,\varepsilon}^{(1)}|^3}+o(\mu_{\varepsilon}^{-2})\nonumber\\
&=128\sum_{i=1}^2\frac{b_{i,0}(\mu_{j,\varepsilon}^{(1)})^{-2}}{e^{u_{0,i}(x_{j,\varepsilon}^{(1)})}}\int_{\mathbb{R}^2\backslash B_r(x_{j,\varepsilon}^{(1)})}\frac{dy}{|y-x_{j,\varepsilon}^{(1)}|^4}+o(\mu_{\varepsilon}^{-2}).\nonumber
\end{align}

\,\,

\noindent{\bf Proof of (\ref{e:DilationTerms3}).}
To simplify the notations, let us denote
\begin{align}
\mathcal{D}_\varepsilon := \|u_{1,\varepsilon}^{(1)}-u_{1,\varepsilon}^{(2)}\|_{L^\infty(\Omega)}+\|u_{2,\varepsilon}^{(1)}-u_{2,\varepsilon}^{(2)}\|_{L^\infty(\Omega)}.\nonumber
\end{align}
By (\ref{e:GlobalCancellation}) and (\ref{def:g1}), we get
\begin{align}
\frac{2}{\varepsilon^2}\sum_{j=1}^k\int_{B_r(x_{j,\varepsilon}^{(1)})} g_1dx&=\frac{2}{\varepsilon^2 \mathcal{D}_\varepsilon}\int_{\Omega}\Big(e^{\overline{u}_{1,\varepsilon}^{(1)}+u_{0,1}+\ln h_1+\overline{u}_{2,\varepsilon}^{(1)}+u_{0,2}+\ln h_2}\nonumber\\
&\quad\quad\quad\quad\quad - e^{\overline{u}_{1,\varepsilon}^{(2)}+u_{0,1}+\ln h_1+\overline{u}_{2,\varepsilon}^{(2)}+u_{0,2}+\ln h_2}\Big)dx\nonumber\\
&\quad+\frac{1}{\varepsilon^2 \mathcal{D}_\varepsilon}\sum_{j=1}^k\int_{\Omega_j\backslash B_r(x_{j,\varepsilon}^{(1)})}\Big(e^{\overline{u}_{1,\varepsilon}^{(1)}+u_{0,1}+\ln h_1}-e^{\overline{u}_{1,\varepsilon}^{(2)}+u_{0,1}+\ln h_1}\nonumber\\
&\quad\quad\quad\quad\quad\quad\quad\quad\quad\quad +e^{\overline{u}_{2,\varepsilon}^{(1)}+u_{0,2}+\ln h_2}-e^{\overline{u}_{2,\varepsilon}^{(2)}+u_{0,2}+\ln h_2}\Big)dx.\nonumber
\end{align}
Here, the first integral splits as
\begin{align}
&\frac{2}{\varepsilon^2 \mathcal{D}_\varepsilon}\int_{\Omega}\Big(e^{\overline{u}_{1,\varepsilon}^{(1)}+u_{0,1}+\ln h_1+\overline{u}_{2,\varepsilon}^{(1)}+u_{0,2}+\ln h_2}\nonumber\\
&\quad\quad\quad\quad - e^{\overline{u}_{1,\varepsilon}^{(2)}+u_{0,1}+\ln h_1+\overline{u}_{2,\varepsilon}^{(2)}+u_{0,2}+\ln h_2}\Big)dx\nonumber\\
&=\frac{2}{\varepsilon^2 \mathcal{D}_\varepsilon}\Bigg[\int_{\Omega}\Big(e^{\overline{u}_{1,\varepsilon}^{(1)}+u_{0,1}+\ln h_1+\overline{u}_{2,\varepsilon}^{(1)}+u_{0,2}+\ln h_2}\nonumber\\
&\quad\quad\quad\quad\quad\quad - e^{\overline{u}_{1,\varepsilon}^{(2)}+u_{0,1}+\ln h_1+\overline{u}_{2,\varepsilon}^{(1)}+u_{0,2}+\ln h_2}\Big)dx\nonumber\\
&\quad\quad\quad +\int_{\Omega}\Big(e^{\overline{u}_{1,\varepsilon}^{(2)}+u_{0,1}+\ln h_1+\overline{u}_{2,\varepsilon}^{(1)}+u_{0,2}+\ln h_2}\nonumber\\
&\quad\quad\quad\quad\quad\quad - e^{\overline{u}_{1,\varepsilon}^{(2)}+u_{0,1}+\ln h_1+\overline{u}_{2,\varepsilon}^{(2)}+u_{0,2}+\ln h_2}\Big)dx\Bigg]\nonumber\\
&=2(A+o(1))\sum_{i=1}^2 b_{i,0} \sum_{j=1}^k e^{u_{0,i}(x_{j,\varepsilon}^{(1)})}e^{\beta_{j,\varepsilon}^{(1)}}\nonumber
\end{align}
with $A=8\int_{\mathbb{R}^2}\frac{1-|x|^2}{(1+|x|^2)^5}$, and the second term evaluates to
\begin{align}
&\frac{1}{\varepsilon^2 \mathcal{D}_\varepsilon}\sum_{j=1}^k\int_{\Omega_j\backslash B_r(x_{j,\varepsilon}^{(1)})}\Big(e^{\overline{u}_{1,\varepsilon}^{(1)}+u_{0,1}+\ln h_1}-e^{\overline{u}_{1,\varepsilon}^{(2)}+u_{0,1}+\ln h_1}\nonumber\\
&\quad\quad\quad\quad\quad\quad\quad\quad\quad +e^{\overline{u}_{2,\varepsilon}^{(1)}+u_{0,2}+\ln h_2}-e^{\overline{u}_{2,\varepsilon}^{(2)}+u_{0,2}+\ln h_2}\Big)dx\nonumber\\
&=\sum_{i=1}^2 b_{i,0}\sum_{j=1}^k\frac{128 (\mu_{j,\varepsilon}^{(1)})^{-2}}{e^{u_{0,i}(x_{j,\varepsilon}^{(1)})}}\int_{\Omega_j\backslash B_r(x_{j,\varepsilon}^{(1)})}\frac{e^{f_{i,j,\varepsilon}^{(1)}(y)}}{|y-x_{j,\varepsilon}^{(1)}|}dy+o(\mu_{\varepsilon}^{-2}).\nonumber
\end{align}
Together, they imply (\ref{e:DilationTerms3}).

\,\,

\noindent{\bf Proof of (\ref{e:DilationTerms4}).}
By (\ref{def:g2}), Lemma \ref{l:M8pimu-2} and Lemma \ref{l:RefinedPohozaev}, we get
\begin{align}
\frac{1}{\varepsilon^2} \int_{B_r(x_{j,\varepsilon}^{(1)})} g_2dx&=\sum_{i=1}^2(\mu_{j,\varepsilon}^{(1)})^{-2}\int_{B_{r\mu_{j,\varepsilon}^{(1)}}}\Big(z^T D^2(u_{0,i}+\phi_{j,\varepsilon}^{(1)})(x_{j,\varepsilon}^{(1)})z\Big)e^{u_{0,i}(x_{j,\varepsilon}^{(1)})+U_{i,j,\varepsilon}^{(1)}}\overline{\xi}_{i,\varepsilon,j}dz\nonumber\\
&\quad +o_r(1)\mu_{\varepsilon}^{-2}+o(\mu_{\varepsilon}^{-2})\nonumber\\
&=\frac{1}{2}\sum_{i=1}^2(\mu_{j,\varepsilon}^{(1)})^{-2}\int_{B_{r\mu_{j,\varepsilon}^{(1)}}}\Delta(u_{0,i}+\phi_{j,\varepsilon}^{(1)})(x_{j,\varepsilon}^{(1)})|z|^2e^{u_{0,i}(x_{j,\varepsilon}^{(1)})+U_{i,j,\varepsilon}^{(1)}}\overline{\xi}_{i,\varepsilon,j}dz\nonumber\\
&\quad +o_r(1)\mu_{\varepsilon}^{-2}+o(\mu_{\varepsilon}^{-2})\nonumber\\
&=o_r(1)\mu_{\varepsilon}^{-2}+o(\mu_{\varepsilon}^{-2}).\nonumber
\end{align}
Here, in the last step we use the fact that  $u_{0,i}+\phi_{j,\varepsilon}^{(1)}$ are harmonic.

\begin{flushright}
$\Box$
\end{flushright}

\noindent{\bf Proof of Proposition \ref{prop:PohozaevTerms2}.}

\,\,

\noindent{\bf Proof of (\ref{e:TranslationTerms1}).}
By Lemma \ref{l:xiivarepsilonExpansion} and Claim \ref{c:Aij}, we get
\begin{align}
\nabla\xi_{i,\varepsilon}(x)=-8\pi\sum_{j=1}^k(\mu_{j,\varepsilon}^{(1)})^{-1}\sum_{h=1}^2\partial_h \nabla_x G(x_{j,\varepsilon}^{(1)},x)b_{j,h}+o(\mu_{\varepsilon}^{-1})\nonumber
\end{align}
for $x\in B_r(x_{j,\varepsilon}^{(1)})\backslash B_\delta(x_{j,\varepsilon}^{(1)})$. Define
\begin{align}
\xi_{i,j,\varepsilon}^*(x)=-8\pi\Big((\mu_{j,\varepsilon}^{(1)})^{-1}\sum_{h=1}^2\partial_h \gamma(x_{j,\varepsilon}^{(1)},x)b_{j,h}+\sum_{m=1,m\neq j}^k(\mu_{m,\varepsilon}^{(1)})^{-1}\sum_{h=1}^2\partial_h G(x_{m,\varepsilon}^{(1)},x)b_{m,h}\Big).\nonumber
\end{align}
On the other hand, (\ref{e:uLocalExpansion}) and Lemma \ref{l:M8pimu-2} imply that
\begin{align}
\nabla\overline{u}_{i,\varepsilon}^{(l)}(x)=\frac{-4(x-x_{j,\varepsilon}^{(1)})}{|x-x_{j,\varepsilon}^{(1)}|^2} +O(\mu_{\varepsilon}^{-2})\nonumber
\end{align}
for $x\in B_r(x_{j,\varepsilon}^{(1)})\backslash B_\delta(x_{j,\varepsilon}^{(1)})$, $l,i=1,2$ and $j=1,\cdots,k$. Together, they imply that
\begin{align}
&\int_{\partial B_r(x_{j,\varepsilon}^{(1)})}\bigg( \langle\nu,\nabla\xi_{1,\varepsilon}\rangle\partial_h\overline{u}_{2,\varepsilon}^{(1)}+\langle\nu,\nabla\xi_{2,\varepsilon}\rangle\partial_h\overline{u}_{1,\varepsilon}^{(1)}+\langle\nu,\nabla\overline{u}_{1,\varepsilon}^{(2)}\rangle\partial_h\xi_{2,\varepsilon}+\langle\nu,\nabla\overline{u}_{2,\varepsilon}^{(2)}\rangle\partial_h\xi_{1,\varepsilon}\nonumber\\
&\quad\quad\quad\quad\quad - \langle\nabla\xi_{1,\varepsilon},\nabla\overline{u}_{2,\varepsilon}^{(1)}\rangle\nu_h-\langle\nabla\overline{u}_{1,\varepsilon}^{(2)},\nabla\xi_{2,\varepsilon}\rangle\nu_h\bigg)d\mathcal{H}^1\nonumber\\
&=\int_{\partial B_r(x_{j,\varepsilon}^{(1)})}\bigg( \langle\nu,\nabla\xi^*_{1,j,\varepsilon}\rangle\partial_h\overline{u}_{2,\varepsilon}^{(1)}+\langle\nu,\nabla\xi^*_{2,j,\varepsilon}\rangle\partial_h\overline{u}_{1,\varepsilon}^{(1)}+\langle\nu,\nabla\overline{u}_{1,\varepsilon}^{(2)}\rangle\partial_h\xi^*_{2,j,\varepsilon}+\langle\nu,\nabla\overline{u}_{2,\varepsilon}^{(2)}\rangle\partial_h\xi^*_{1,j,\varepsilon}\nonumber\\
&\quad\quad\quad\quad\quad - \langle\nabla\xi^*_{1,j,\varepsilon},\nabla\overline{u}_{2,\varepsilon}^{(1)}\rangle\nu_h-\langle\nabla\overline{u}_{1,\varepsilon}^{(2)},\nabla\xi^*_{2,j,\varepsilon}\rangle\nu_h\bigg)d\mathcal{H}^1+o(\mu_{\varepsilon}^{-1})\nonumber\\
&=64\pi^2\bigg[(\mu_{j,\varepsilon}^{(1)})^{-1}\sum_{h'=1}^2\partial_h\partial_{h'} \gamma(x_{j,\varepsilon}^{(1)},x_{j,\varepsilon}^{(1)})b_{j,h'}\nonumber\\
&\quad\quad\quad\quad +\sum_{m=1,m\neq j}^k(\mu_{m,\varepsilon}^{(1)})^{-1}\sum_{h'=1}^2\partial_h\partial_{h'} G(x_{m,\varepsilon}^{(1)},x_{j,\varepsilon}^{(1)})b_{m,h'}\bigg]\nonumber\\
&\quad+o_r(1)\mu_{\varepsilon}^{-1}+o(\mu_{\varepsilon}^{-1}),
\end{align}
There, the first equality is due to symmetry while the second is due to Taylor expansion.

\,\,

\noindent{\bf Proof of (\ref{e:TranslationTerms2}).}
By Lemma \ref{l:RefinedPohozaev}, (\ref{def:fi}), (\ref{Asy:c1}), (\ref{Asy:c2}) and Lemma \ref{l:M8pimu-2}, we get
\begin{align}
&\int_{B_r(x_{j,\varepsilon}^{(1)})}\Big[f_{1,\varepsilon}\partial_h(u_{0,2}+\ln h_2)+f_{2,\varepsilon}\partial_h(u_{0,2}+\ln h_2)\Big]dx\nonumber\\
&=(\mu_{j,\varepsilon}^{(1)})^{-1}\Bigg(\sum_{i=1}^2\int_{\mathbb{R}^2}e^{u_{0,i}(x_{j,\varepsilon}^{(1)})}e^{U_{i,j}}\Big(b_{i,0}\phi_{i,j}+b_{j,1}\frac{\partial U_{j,i}}{\partial x_1}+b_{j,2}\frac{\partial U_{j,i}}{\partial x_2}\Big)\nonumber\\
&\quad\quad\quad\quad \times \sum_{h'=1}^2x_{h'} \partial_h\partial_{h'}(u_{0,i}+\phi_{j,\varepsilon}^{(1)})(x_{j,\varepsilon}^{(1)})\Bigg)+o(\mu_{\varepsilon}^{-2})\nonumber\\
&=-8\pi\sum_{i=1}^2(\mu_{j,\varepsilon}^{(1)})^{-1}\big\langle\nabla\partial_h(u_{0,i}+\phi_{j,\varepsilon}^{(1)})(x_{j,\varepsilon}^{(1)}),\vec{b}_{j}\big\rangle+o(\mu_{\varepsilon}^{-1}).\nonumber
\end{align}
Here, $\vec{b}_{j}=(b_{j,1},b_{j,2})$.

\,\,

\noindent{\bf Proof of (\ref{e:TranslationTerms3}).}
The proof of (\ref{e:TranslationTerms3}) is similar to the one for (\ref{e:DilationTerms2}).

\begin{flushright}
$\Box$
\end{flushright}

\section*{Acknowledgements}

\noindent{{\bf Funding.} 
Zetao Cheng is supported by Basic Science Research Program
through the National Research Foundation of Korea (NRF) funded by the Ministry of Science and
ICT (2020R1C1C1A01010133, RS-2025-00558417). Lei Zhang is partially supported by Simon's Foundation grant SFI-MPS-TSM-00013752.

\end{document}